\documentclass[10pt, reqno]{amsart}
\usepackage{hhline}
\usepackage{tikz-cd}
\usepackage{verbatim}
\usepackage{amsmath}
\usepackage{amsxtra}
\usepackage{amscd}
\usepackage{amsthm}
\usepackage{amsfonts}
\usepackage{amssymb}
\usetikzlibrary{arrows}
\usepackage{array}

\usepackage[all, cmtip]{xy}
\usepackage{stmaryrd}
\usepackage{color}
\usepackage{amsthm,amsfonts,amssymb,amsmath,amsxtra}
\usepackage[all]{xy}
\SelectTips{cm}{}
\usepackage{xr-hyper}
\usepackage[colorlinks=
   citecolor=Black,
   linkcolor=Red,
   urlcolor=Blue,
   backref=page]{hyperref}
\usepackage{verbatim}

\usepackage[margin=1.25in]{geometry}

\usepackage{mathrsfs}

\RequirePackage{xspace}
\RequirePackage{etoolbox}
\RequirePackage{varwidth}
\RequirePackage{enumitem}
\RequirePackage{tensor}
\RequirePackage{mathtools}
\RequirePackage{longtable}
\RequirePackage{multirow}

\def\<{\langle}
\def\>{\rangle}

\newcommand{\fka}{\ensuremath{\mathfrak{a}}\xspace}

\newcommand{\fkd}{\ensuremath{\mathfrak{d}}\xspace}

\newcommand{\fkp}{\ensuremath{\mathfrak{p}}\xspace}

\newcommand{\fkD}{\ensuremath{\mathfrak{D}}\xspace}

\newcommand{\nat}{{\natural}}

\newcommand{\BA}{\ensuremath{\mathbb {A}}\xspace}

\newcommand{\BC}{\ensuremath{\mathbb {C}}\xspace}

\newcommand{{\BG}}{\ensuremath{\mathbb {G}}\xspace}

\newcommand{{\BK}}{\ensuremath{\mathbb {K}}\xspace}

\newcommand{\BM}{\ensuremath{\mathbb {M}}\xspace}

\newcommand{\BQ}{\ensuremath{\mathbb {Q}}\xspace}
\newcommand{\BR}{\ensuremath{\mathbb {R}}\xspace}
\newcommand{\BS}{\ensuremath{\mathbb {S}}\xspace}

\newcommand{\BZ}{\ensuremath{\mathbb {Z}}\xspace}

\newcommand{\CF}{\ensuremath{\mathcal {F}}\xspace}

\newcommand{\CK}{\ensuremath{\mathcal {K}}\xspace}
\newcommand{\CL}{\ensuremath{\mathcal {L}}\xspace}
\newcommand{\CM}{\ensuremath{\mathcal {M}}\xspace}

\newcommand{\CO}{\ensuremath{\mathcal {O}}\xspace}
\newcommand{\CP}{\ensuremath{\mathcal {P}}\xspace}
\newcommand{\CQ}{\ensuremath{\mathcal {Q}}\xspace}
\newcommand{\CR}{\ensuremath{\mathcal {R}}\xspace}

\newcommand{\CV}{\ensuremath{\mathcal {V}}\xspace}

\newcommand{\CZ}{\ensuremath{\mathcal {Z}}\xspace}

\newcommand{\RH}{\ensuremath{\mathrm {H}}\xspace}

\newcommand{\RT}{\ensuremath{\mathrm {T}}\xspace}
\newcommand{\RU}{\ensuremath{\mathrm {U}}\xspace}
\newcommand{\RV}{\ensuremath{\mathrm {V}}\xspace}

\newcommand{\ad}{{\mathrm{ad}}}

\newcommand{\aform}{\ensuremath{\langle\text{~,~}\rangle}\xspace}

\DeclareMathOperator{\charac}{char}
\DeclareMathOperator{\cok}{cok}

\DeclareMathOperator{\diag}{diag}

\DeclareMathOperator{\End}{End}

\DeclareMathOperator{\Gal}{Gal}
\newcommand{\GL}{\mathrm{GL}}

\newcommand{\GSp}{\mathrm{GSp}}
\newcommand{\GU}{\mathrm{GU}}

\newcommand{\Herm}{\mathrm{Herm}}
\DeclareMathOperator{\Hom}{Hom}
\newcommand{\HT}{\ensuremath{\mathrm{HT}}\xspace}

\renewcommand{\i}{^{-1}}
\newcommand{\id}{\ensuremath{\mathrm{id}}\xspace}

\newcommand{\inv}{{\mathrm{inv}}}
\DeclareMathOperator{\Isom}{Isom}

\DeclareMathOperator{\Lie}{Lie}
\newcommand{\LNSch}{\ensuremath{(\mathrm{LNSch})}\xspace}

\newcommand{\naive}{\ensuremath{\mathrm{naive}}\xspace}

\DeclareMathOperator{\Nm}{Nm}

\DeclareMathOperator{\ord}{ord}

\DeclareMathOperator{\rank}{rank}

\newcommand{\ram}{\ensuremath{\mathrm{ram}}\xspace}

\DeclareMathOperator{\Res}{Res}
\DeclareMathOperator{\Ros}{Ros}

\newcommand{\sform}{\ensuremath{(\text{~,~})}\xspace}
\DeclareMathOperator{\sig}{sig}

\DeclareMathOperator{\Spec}{Spec}

\newcommand{\ssm}{\smallsetminus}

\newcommand{\surj}{\twoheadrightarrow}

\DeclareMathOperator{\tr}{tr}

\newcommand{\U}{\mathrm{U}}
\DeclareMathOperator{\uHom}{\underline{Hom}}
\DeclareMathOperator{\uIsom}{\underline{Isom}}
\newcommand{\un}{\mathrm{un}}

\newcommand{\wt}{\widetilde}
\newcommand{\wh}{\widehat}

\newcommand{\ov}{\overline}

\newcommand{\lra}{\longrightarrow}

\newcommand{\la}{\langle}
\newcommand{\ra}{\rangle}

\newtheorem{theorem}{Theorem}

\newtheorem{lemma}[theorem]{Lemma}

\theoremstyle{definition}
\newtheorem{definition}[theorem]{Definition}
\newtheorem{example}[theorem]{Example}

\newtheorem{remark}[theorem]{Remark}

\newenvironment{altenumerate}
   {\begin{list}
      {\textup{(\theenumi)} }
      {\usecounter{enumi}
       \setlength{\labelwidth}{0pt}
       \setlength{\labelsep}{0pt}
       \setlength{\leftmargin}{0pt}
       \setlength{\itemsep}{\the\smallskipamount}
       \renewcommand{\theenumi}{\roman{enumi}}
      }}
   {\end{list}}
\newenvironment{altitemize}
   {\begin{list}
      {$\bullet$}
      {\setlength{\labelwidth}{0pt}
	   \setlength{\itemindent}{5pt}
       \setlength{\labelsep}{5pt}
       \setlength{\leftmargin}{0pt}
       \setlength{\itemsep}{\the\smallskipamount}
      }}
   {\end{list}}

\numberwithin{equation}{section}
\numberwithin{theorem}{section}

\setitemize[0]{leftmargin=*,itemsep=\the\smallskipamount}
\setenumerate[0]{leftmargin=*,itemsep=\the\smallskipamount}

\renewcommand{\to}{%
   \ifbool{@display}{\longrightarrow}{\rightarrow}%
   }
\let\shortmapsto\mapsto
\renewcommand{\mapsto}{%
   \ifbool{@display}{\longmapsto}{\shortmapsto}%
   }
\newcommand{\hooklongrightarrow}{\mathrel{\mkern 0.5mu\lhook\mkern -3.5mu\relbar\mkern -3mu \rightarrow }}
\newcommand{\inj}{%
   \ifbool{@display}{\hooklongrightarrow}{\hookrightarrow}
   }
\newcommand{\isoarrow}{%
   \ifbool{@display}{\overset{\sim}{\longrightarrow}}{\xrightarrow\sim}%
   }
\newlength{\olen}
\newlength{\ulen}
\newlength{\xlen}
\newcommand{\xra}[2][]{%
   \ifbool{@display}%
      {\settowidth{\olen}{$\overset{#2}{\longrightarrow}$}%
       \settowidth{\ulen}{$\underset{#1}{\longrightarrow}$}%
       \settowidth{\xlen}{$\xrightarrow[#1]{#2}$}%
       \ifdimgreater{\olen}{\xlen}%
          {\underset{#1}{\overset{#2}{\longrightarrow}}}%
          {\ifdimgreater{\ulen}{\xlen}%
             {\underset{#1}{\overset{#2}{\longrightarrow}}}
             {\xrightarrow[#1]{#2}}}}%
      {\xrightarrow[#1]{#2}}
   }
\makeatother
\newcommand{\xyra}[2][]{%
   \settowidth{\xlen}{$\xrightarrow[#1]{#2}$}%
   \ifbool{@display}%
      {\settowidth{\olen}{$\overset{#2}{\longrightarrow}$}%
       \settowidth{\ulen}{$\underset{#1}{\longrightarrow}$}%
       \ifdimgreater{\olen}{\xlen}%
          {\mathrel{\xymatrix@M=.12ex@C=3.2ex{\ar[r]^-{#2}_-{#1} &}}}%
          {\ifdimgreater{\ulen}{\xlen}%
             {\mathrel{\xymatrix@M=.12ex@C=3.2ex{\ar[r]^-{#2}_-{#1} &}}}
             {\mathrel{\xymatrix@M=.12ex@C=\the\xlen{\ar[r]^-{#2}_-{#1} &}}}}}%
      {\mathrel{\xymatrix@M=.12ex@C=\the\xlen{\ar[r]^-{#2}_-{#1} &}}}%
   }
\makeatletter
\newcommand{\xla}[2][]{%
   \ifbool{@display}%
      {\settowidth{\olen}{$\overset{#2}{\longleftarrow}$}%
       \settowidth{\ulen}{$\underset{#1}{\longleftarrow}$}%
       \settowidth{\xlen}{$\xleftarrow[#1]{#2}$}%
       \ifdimgreater{\olen}{\xlen}%
          {\underset{#1}{\overset{#2}{\longleftarrow}}}%
          {\ifdimgreater{\ulen}{\xlen}%
             {\underset{#1}{\overset{#2}{\longleftarrow}}}
             {\xleftarrow[#1]{#2}}}}%
      {\xleftarrow[#1]{#2}}
   }
\renewcommand{\lra}{%
   \ifbool{@display}{\longleftrightarrow}{\leftrightarrow}%
   }
\newcommand{\undertilde}{\raisebox{0.4ex}{\smash[t]{$\scriptstyle\sim$}}}

\begin{document}

\title{On Shimura varieties for unitary groups}

\author{M. Rapoport}
\address{Mathematisches Institut der Universit\"at Bonn, Endenicher Allee 60, 53115 Bonn, Germany, and University of Maryland, Department of Mathematics, College Park, MD 20742, USA}
\email{rapoport@math.uni-bonn.de}
\author{B. Smithling}
\address{University of Maryland, Department of Mathematics, College Park, MD 20742, USA}
\email{bds@umd.edu}
\author{W. Zhang}
\address{Massachusetts Institute of Technology, Department of Mathematics, 77 Massachusetts Avenue, Cambridge, MA 02139, USA}
\email{weizhang@mit.edu}
\thanks{M.R. is supported by a grant from  the Deutsche Forschungsgemeinschaft through the grant SFB/TR 45 and by funds connected with the Brin E-Nnovate Chair at the University of Maryland. B.S. is supported by NSA Grant H98230-16-1-0024 and Simons Foundation Grant $\#585707$. W.Z. is supported by NSF DMS $\#$1901642.}

\date{\today}

\begin{abstract}
This is a largely expository article based on our paper \cite{RSZ3} on arithmetic diagonal cycles on unitary Shimura varieties. We define a class of Shimura varieties closely related to unitary groups which represent a moduli problem of abelian varieties with additional structure, and which admit interesting algebraic cycles.   We generalize  to arbitrary signature type the results of loc.\ cit.\ valid under special signature conditions. We  compare our Shimura varieties with other unitary Shimura varieties.
\end{abstract}

\dedicatory{To David Mumford on his 81st birthday}

\date{\today}
\maketitle
\tableofcontents
\section{Introduction}\label{s:intro}
In \cite{DelBour}, Deligne gives the definition of a Shimura variety $(S(G, \{h \})_{K})$ (a tower of quasi-projective complex varieties indexed by sufficiently small compact open subgroups $K\subset G(\BA_f)$) starting from a Shimura datum  $(G, \{ h \})$. He defines the associated Shimura field $E(G, \{h \})$ and proves that there is at most one \emph{canonical model} of $(S(G, \{h \})_{K})$ over $E(G, \{h \})$. He then goes on to construct the canonical model for some Shimura varieties associated to classical groups, by giving an interpretation of the varieties in the tower as moduli spaces of abelian varieties with additional structure. The basic example is given by the group $G$ of symplectic similitudes with its natural conjugacy class $\{h\}$ (the \emph{Siegel case})---this case leads to the moduli space of principally polarized abelian varieties. Let $p$ be a prime number. For open compact subgroups $K$ of the form $K=K^p \times K_p$, where $K^p\subset G(\BA_f^p)$ and where $K_p\subset G(\BQ_p)$ is the stabilizer of a self-dual lattice (i.e., $K_p$ is hyperspecial), this moduli description allows one to  extend the model over $\BQ$ of $(S(G, \{h \})_{K})$ to a model over the localization $\BZ_{(p)}$ with good reduction modulo $p$.

Another class of examples is  related to unitary groups.   The case of the group of unitary similitudes is treated briefly by  Deligne  in \cite{DelBour} and  in detail by Kottwitz  in \cite{K-points}. It is considerably more difficult than the Siegel case due to the failure of the Hasse principle for these groups. A special case of these Shimura varieties is used in all proofs  of the local Langlands conjecture for $p$-adic fields (the \emph{Harris-Taylor} case). It turns out that if $n$ is even, the Shimura variety represents a moduli problem of abelian varieties with additional structure; if $n$ is odd, the analogous moduli problem is represented by a finite disjoint sum of copies of the Shimura variety. If $K=K^p \times K_p$, where $K_p$ is hyperspecial, Kottwitz defines a $p$-integral model of the corresponding model over $E(G, \{h\})$. 

Another case is given by the unitary groups. This case is considered in \cite{GGP}, in the \emph{fake Drinfeld signature} case (see Example \ref{eg:special types} below), and is used  in the formulation of the Arithmetic Gan--Gross--Prasad conjecture. This class of Shimura varieties is not of \emph{PEL type}, i.e., is not represented by a moduli problem of abelian varieties with additional structure. It is, however, of \emph{abelian type}. This entails that its canonical model is defined in Deligne \cite{DCorv}. It also implies that, by Kisin--Pappas \cite{KP}, it has a $p$-integral model when $K=K^p \times K_p$, where $K_p$ is an arbitrary parahoric subgroup.  However, both constructions are rather indirect and yield models  which are difficult to analyze. This seems to be a major impediment to progress on these Shimura varieties. 
 
In this paper, following \cite{RSZ3},  we formulate a new  variant of the Deligne--Kottwitz Shimura varieties and compare it with the previous two classes. Our variant has the advantage of always representing a moduli problem of abelian varieties. By extending the moduli problem, we also define $p$-integral models when $K=K^p \times K_p$, with $K_p$ a parahoric subgroup. In fact, under certain special circumstances, we even define global integral models (in the fake Drinfeld case). Another advantage of our  variant is that it always accommodates the algebraic cycles that appear in the Gan--Gross--Prasad intersection problem and in the Kudla--Rapoport intersection problem. We refer to \cite{RSZ3}, where we give a variant of the Arithmetic Gan--Gross--Prasad conjecture and solve it in certain  low-dimensional cases. In the case of an imaginary quadratic field, our variant Shimura variety appears also in \cite{BHKRY}, with similar aims. However, this special case does not bring out all the  features of our definition; in particular, the \emph{sign invariant} of \cite{RSZ3} does not play any role in that case, comp.\ the table in Section \ref{sumtable} below. 

Our paper is largely expository. One of our aims is to show that the definitions in \cite{RSZ3} extend from the fake Drinfeld signature case to the case of general signature. We also go beyond \cite{RSZ3} in that we also discuss the problem of flatness, resp.\ of smoothness, resp.\ of regularity, of the $p$-integral models in general. Our hope is that the global integral models constructed here will find applications in  arithmetic intersection problems in analogy with those mentioned above.  

The lay-out of the paper is as follows. In Section \ref{s:DK}, we set the stage by recalling the Shimura varieties attached to groups of unitary similitudes and to unitary groups. In Section \ref{s:rsz}, we introduce the variant of these Shimura varieties introduced in \cite{RSZ3}. In Section \ref{s:semiglob}, we define $p$-integral models of these last Shimura varieties. In Section \ref{s:flatsm}, we discuss flatness and smoothness of these $p$-integral models. In Section \ref{s:global}, we discuss global integral models.  In Appendix \ref{s:LM}, we show how the formalism of the local model diagram of \cite{RZ1} holds for all primes $p$ (including $p=2$) in the case of unramified PEL data of type $A$.

The paper is a vastly extended version of the talk with the same title given by one of us (M.R.) at the \emph{SuperAG} in Bonn 2017, at the occasion of his retirement from the University of Bonn. It is also related to his talk at the 2018 Simons conference \emph{Periods and $L$-values of motives}, and to Appendix C in Y. Liu's article \cite{Liu}. 

We dedicate the paper to D.~Mumford on the occasion of his 81st birthday. He is the founder of the theory of moduli spaces of abelian varieties and has been a main force in moving this subject to the forefront of mathematics. On a personal level, one of us (M.R.) owes more to him than can be said in a few lines; he is happy to express here his gratitude.   We also thank the referee for his/her remarks on the paper.

\subsection*{Notation}
In order to deal with \emph{all} congruence subgroups $K$, not only ones that are small enough, we consider the tower $(S(G, \{h \})_K)$ as a tower of \emph{orbifolds}. However, abusing language, we continue to refer to this tower as a Shimura variety.

We write $\BA_F$, $\BA_{F,f}$, and $\BA_{F,f}^p$ for the respective rings of adeles, finite adeles, and finite adeles away from $p$ of a number field $F$. When $F = \BQ$, we abbreviate these to \BA, $\BA_f$, and $\BA_f^p$, respectively, and we set $\ov \BA := \BA \otimes_\BQ \ov\BQ$. Here $\ov\BQ$ denotes the algebraic closure of $\BQ$ in $\BC$.  We fix once and for all an element $\sqrt{-1} \in \BC$.

We take all hermitian forms to be linear in the first variable and conjugate-linear in the second, and we always assume that they are nondegenerate.

When working with vector spaces over $F$, we use a superscript $*$ to denote $F$-linear dual spaces and homomorphisms. We also use a superscript $*$ to denote dual lattices (in both the global and local contexts) with respect to a hermitian form.  By contrast, we always use a superscript $\vee$ to denote dual lattices with respect to a $\BQ$- or $\BQ_p$-valued bilinear form. The following situation (say, in the global context; the local context is completely analogous) will arise repeatedly throughout this paper.  Let $(W,\sform)$ be a hermitian space for $F$ with respect to an order $2$ automorphism (in the paper, $F$ will always be a CM field), let $\zeta \in F$ be a traceless element for this automorphism, and consider the alternating $\BQ$-valued form $\tr_{F/\BQ}\zeta\sform$ on $W$.  Then for any $O_F$-lattice $\Lambda \subset W$, we have $\Lambda^\vee = \zeta\i \fkD\i \Lambda^*$, where $\fkD$ denotes the different of $F/\BQ$.

In the context of abelian schemes, we use a superscript $\vee$ to denote the dual abelian scheme and dual morphisms.  A \emph{quasi-polarization} (sometimes called a \emph{$\BQ$-polarization} in the literature) on an abelian scheme $A$ is a symmetric quasi-isogeny $\lambda\colon A \to A^\vee$ such that, Zariski-locally on the base, $n\lambda$ is an honest polarization for some positive integer $n$. We denote the Rosati involution of a quasi-polarization $\lambda\colon A \to A^\vee$ by $\Ros_\lambda$.  We denote by $\Hom^0(A,B)$, resp.\ $\Hom_{(p)}(A,B)$, the Zariski-sheafification of $U \mapsto \Hom(A_U,B_U) \otimes_\BZ \BQ$, resp.\ $U \mapsto \Hom(A_U,B_U) \otimes_\BZ \BZ_{(p)}$, for $U$ an open subscheme of the base (the homomorphism group in the isogeny category, resp.\ the prime-to-$p$ isogeny category).  We similarly define $\End^0(A)$ and $\End_{(p)}(A)$.  When the base $S$ locally noetherian, we denote by $\wh\RT (A)=\prod\nolimits_{\ell} \RT_\ell(A)$, resp.\ $\wh \RV(A)=\wh\RT (A)\otimes\BQ$, resp.\ $\wh\RV^p(A)=(\prod\nolimits_{\ell\neq p} \RT_\ell(A))\otimes\BQ$, the Tate module, resp.\ the rational Tate module, resp.\ the rational Tate module prime to $p$, all regarded as smooth sheaves on $S$ (assuming that $S$ is a $\BQ$-scheme in the case of $\wh\RT(A)$ and $\wh\RV(A)$, and a $\BZ_{(p)}$-scheme in the case of $\wh\RV^p(A)$).  Similarly, when a number field acts on $A$ up to isogeny and $v$ is a finite place of the number field whose residue characteristic $\ell$ is invertible on $S$, we denote by $\RV_v(A)$ the $v$-factor of $\RV_\ell(A)$.  When furthermore the localization at $\ell$ of the ring of integers of the number field acts on $A$ up to prime-to-$\ell$ isogeny, we denote by $\RT_v(A)$ the $v$-factor of $\RT_\ell(A)$.

We often use a subscript $S$ to denote base change to $S$, and when $S=\Spec A$, we often use the subscript $A$ instead. Similarly, we sometimes write $X\otimes_B A$ to denote $X\times_{\Spec B}\Spec A$. 

We write $\LNSch_{/R}$ for the category of locally noetherian schemes over $\Spec R$ for a ring $R$.

\section{The Shimura varieties of Deligne and Kottwitz, and of Gan--Gross--Prasad}\label{s:DK}

\subsection{The group of symplectic similitudes}\label{ss:Sp}
As motivation, we start with the Siegel case. Let $(W, \aform)$ be a nonzero symplectic vector space of dimension $n=2m$ over $\BQ$. Let $\GSp = \GSp(W, \aform)$ be the  group of symplectic similitudes. Choose a symplectic basis of $W$, i.e., a basis with respect to which the matrix of
$\aform$ is given by 
\[
   H_n=\begin{bmatrix}
   0_m&-1_m\\
   1_m&0_m
   \end{bmatrix},
\]
where the displayed entries are $m\times m$ block matrices. The conjugacy class $\{ h_\GSp \}$ in the Shimura datum is the $\GSp(\BR)$-conjugacy class of the homomorphism
\[
   h_\GSp\colon
   \begin{tikzcd}[baseline=(x.base),row sep=-.5ex]
      |[alias=x]| \BC^\times \ar[r]  &  \GSp(\BR)\\
      a+b\sqrt{-1} \ar[r, mapsto]  &  a 1_n+b H_n.
   \end{tikzcd}
\]

Let $K\subset \GSp(\BA_f)$ be a compact open subgroup. Let $\CF_K$ be the category fibered in groupoids over $\LNSch_{/\BQ}$ which associates to each $\BQ$-scheme $S$ the groupoid of triples $(A, \lambda, \ov\eta)$,
where 
\begin{altitemize}
\item $A$ is an abelian scheme over $S$;
\item $\lambda$ is a quasi-polarization on $A$; and
\item $\ov\eta$ is a $K$-orbit of symplectic similitudes
\begin{equation}\label{level}
   \eta\colon \wh \RV(A) \isoarrow W\otimes \BA_f,
\end{equation}
\end{altitemize}
where $K$ acts via the tautological representation of $\GSp(\BA_f)$ on $W \otimes \BA_f$, cf.\ \cite[\S5]{K-points}.  The morphisms $(A, \lambda, \ov\eta)\to (A', \lambda', \ov\eta')$ in this groupoid are the quasi-isogenies $\mu\colon A\to A'$ such that, Zariski-locally on $S$, the pullback of $\lambda'$ is a $\BQ^\times$-multiple of $\lambda$, and such that the pullback of $\ov\eta'$ is $\ov\eta$.

Note that the Weil form on the rational Tate module $\wh \RV(A)$ defined by $\lambda$ naturally takes values in $\BA_f(1)$; to compare the symplectic forms on both sides of \eqref{level}, it is necessary to choose a trivialization $\BA_f(1) \isoarrow \BA_f$, which is unique up to a factor in $\BA_f^\times$. 

The theorem in this context, which is the model of all other theorems in this paper, is the following.
\begin{theorem}\label{thmSp}
The  moduli problem $\CF_K$ is representable by a Deligne--Mumford stack $M_K$ over $\Spec \BQ$, and
$$
   M_K(\BC)=S\bigl(\GSp, \{ h_\GSp \}\bigr)_K ,
$$
compatible with changing $K$. \qed
\end{theorem}

In fact, the tower $(M_K)$ is the canonical model of the Shimura variety $(S(\GSp, \{ h_\GSp \})_K)$.

\subsection{The group of unitary similitudes}\label{ss:GU}
Let $F$ be a CM number field with maximal totally real subfield $F_0$ and nontrivial $F/F_0$-automorphism $a \mapsto \ov a$. Let $n$ be a positive integer. A \emph{generalized CM type of rank $n$} is a function $r\colon \Hom_\BQ(F, \ov\BQ)\to \BZ_{\geq 0}$, denoted $\varphi \mapsto r_\varphi$, such that
\begin{equation}
   r_\varphi+r_{\ov\varphi}=n \quad \text{for all}\quad \varphi ,
\end{equation}
comp.\ \cite{KRnew}.  Here $\ov\varphi$ denotes the precomposition of $\varphi$ by the nontrivial $F/F_0$-automorphism. When $n$ is understood, we also refer to $r$ as a \emph{signature  type}. 
When $n=1$, a generalized CM type is ``the same'' as a usual CM type (i.e., a half-system $\Phi$ of complex embeddings of $F$), via 
\[
   \Phi= \bigl\{\, \varphi\in \Hom_\BQ(F, \ov\BQ) \bigm| r_\varphi=1 \,\bigr\}.
\] 

Fix a CM type $\Phi$ of $F$, and let $(W, \sform)$ be an $F/F_0$-hermitian vector space of dimension $n$. The signatures of $W$ at the archimedean places determine a generalized CM type $r$ of rank $n$, by writing
\[
   \sig W_\varphi=(r_\varphi, r_{\ov\varphi}), \quad \varphi \in \Phi, \quad W_\varphi := W \otimes_{F,\varphi}\BC.
\]
Let $G^\BQ$ be the group of unitary similitudes of $(W, \sform)$, considered as a linear algebraic group over $\BQ$ (with similitude factor in $\BG_m$). For each $\varphi\in\Phi$, choose a $\BC$-basis of $W_\varphi$ with respect to which the matrix of $\sform$ is given by
\begin{equation}\label{diagr}
   \diag(1_{r_\varphi}, -1_{r_{\ov\varphi}}).
\end{equation}
The conjugacy class $\{ h_{G^\BQ} \}$ in the Shimura datum is the $G^\BQ(\BR)$-conjugacy class of the homomorphism $h_{G^\BQ}=(h_{G^\BQ,\varphi})_{\varphi\in\Phi}$, where the components $h_{G^\BQ,\varphi}$ are defined with respect to the inclusion
\begin{equation}\label{G^Q decomp}
   G^\BQ(\BR) \subset \GL_{F \otimes \BR}(W \otimes \BR) \xra[\undertilde]{\Phi} \prod_{\varphi \in \Phi} \GL_\BC(W_\varphi),
\end{equation}
and where each component is defined on $\BC^\times$ by
\[
   h_{G^\BQ,\varphi}\colon z \mapsto \diag(z \cdot 1_{r_\varphi}, \ov z \cdot 1_{r_{\ov\varphi}}).
\]
Then the reflex field $E(G^\BQ, \{ h_{G^\BQ} \})$ is the reflex field $E_r$ of $r$, which is the subfield of $\ov\BQ$ defined by
\begin{equation}
   \Gal(\ov\BQ/E_r)= \bigl\{\, \sigma\in \Gal(\ov\BQ/\BQ) \bigm| \sigma^*(r)=r \,\bigr\} .
\end{equation}

Let $K\subset G^\BQ(\BA_f)$ be a compact open subgroup. Let $\CF_K$ be the category fibered in groupoids over $\LNSch_{/E_r}$ which associates to each $E_r$-scheme $S$ the groupoid of quadruples $(A, \iota, \lambda, \ov\eta)$,
where 
\begin{altitemize}
\item $A$ is an abelian scheme over $S$;
\item $\iota\colon F\to\End^0(A)$ is an action of $F$ on $A$ up to isogeny;
\item $\lambda$ is a quasi-polarization on $A$; and
\item $\ov\eta$ is a $K$-orbit of $\BA_{F,f}$-linear symplectic  similitudes
\[
   \eta\colon \wh \RV(A) \isoarrow W\otimes_\BQ \BA_f,
\]
\end{altitemize}
cf.\ \cite[\S5]{K-points}. Here, as in the Introduction, we equip $W$ with the $\BQ$-symplectic form $\aform =\tr_{F/\BQ} \zeta \sform$ for some fixed element $\zeta \in F^\times$ satisfying $\ov\zeta=-\zeta$. We impose the conditions that
\begin{equation}\label{Ros}
   \Ros_\lambda \bigl(\iota(a)\bigr)=\iota(\ov a)
   \quad\text{for all}\quad
   a\in F,
\end{equation}
and that $A$ satisfies the Kottwitz condition of signature type $r$,
\begin{equation}\label{kottcond}
    \charac\bigl(\iota(a) \mid \Lie A\bigr) = \prod_{\varphi \in \Hom(F,\ov\BQ)} \bigl(T-\varphi(a)\bigr)^{r_\varphi}
    \quad\text{for all}\quad
    a\in F.
\end{equation}
Here the left-hand side in \eqref{kottcond} denotes the characteristic polynomial of the action of $\iota(a)$ on the locally free $\CO_S$-module $\Lie A$;  the right-hand side, which is a priori a polynomial with coefficients in $E_r$, is regarded as an element of $\CO_S[T]$ via the structure morphism. 
The morphisms $(A, \iota, \lambda, \ov\eta) \to (A', \iota', \lambda', \ov\eta')$ in this groupoid are the  $F$-linear quasi-isogenies $\mu\colon A\to A'$ such that, Zariski-locally on $S$, the pullback of $\lambda'$ is a $\BQ^\times$-multiple of $\lambda$, and such that the pullback of $\ov\eta'$ is $\ov\eta$.

The analog of Theorem \ref{thmSp} is as follows.
\begin{theorem}[Kottwitz]\label{thmGU}
The  moduli problem $\CF_K$ is representable by a Deligne--Mumford stack $M_K$ over $\Spec E_r$, and if $n$ is even, 
$$
   M_K(\BC)=S(G^\BQ, \{ h_{G^\BQ} \})_K ,
$$
compatible with changing $K$. If $n$ is odd, then $M_K(\BC)$ is a finite disjoint union of copies of $S(G^\BQ, \{ h_{G^\BQ} \})_K$, again compatible with changing $K$; these copies are enumerated by
\begin{flalign*}
   \phantom{\qed} & &
   \ker^1(\BQ, G^\BQ) := \ker \bigl[ \RH^1\bigl(\BQ, G^\BQ(\ov\BQ)\bigr)\to \RH^1\bigl(\BQ, G^\BQ(\ov\BA)\bigr)\bigr] .
   & & \qed
\end{flalign*}
\end{theorem}

As in the Siegel case, the tower $(M_K)$ is in fact the canonical model of the Shimura variety $(S(G^\BQ, \{ h_{G^\BQ} \})_K)$. 

\begin{example}
\begin{altenumerate}
\item\label{eg:fake drinfeld} (\emph{Fake Drinfeld type}) We say that the generalized CM type $r$ of rank $n$ is of \emph{fake Drinfeld type}
relative to a distinguished element $\varphi_0 \in \Hom_\BQ(F, \ov \BQ)$ if
\[
   r_\varphi=
   \begin{cases}
      n-1,  &  \text{$\varphi = \varphi_0$};\\
        1,  &  \text{$\varphi = \ov\varphi_0$};\\
     \text{$0$ or  $n$},  &  \varphi \in  \Hom_\BQ(F, \ov \BQ)\ssm \{\varphi_0, \ov\varphi_0\}.
   \end{cases}
\]
In this case, $F$ embeds into $E_r$ via $\varphi_0$ for $n\geq 3$. For $n=2$, at least $F_0$ embeds into $E_r$ via $\varphi_0$. For $n = 1$, all we can say is that $E_r$ is the reflex field of the CM type which is the support of $r$.
\item\label{strictFK}(\emph{Strict fake Drinfeld type}) We say that $r$ is of \emph{strict fake Drinfeld type} for the CM type $\Phi$ and an element $\varphi_0 \in \Phi$, and we write $r = r^{(\Phi,\varphi_0)}$, if
\begin{equation}\label{SFDsig}
   r_\varphi^{(\Phi, \varphi_0)}=
   \begin{cases}
      n-1,  &  \text{$\varphi = \varphi_0$};\\
      n,  &  \varphi \in \Phi \ssm \{\varphi_0\}.
   \end{cases}
\end{equation}
\item\label{eg:HT} (\emph{Harris--Taylor type}) This is a special case of strict fake Drinfeld type. Suppose that $F=K_0F_0$, where $K_0$ is an imaginary quadratic field embedded in $\ov\BQ$. Let $\Phi$ be the induced CM type of $F$, i.e.,\ the set of embeddings $F \to \ov\BQ$ whose restriction to $K_0$ is the given embedding.  We fix $\varphi_0\in\Phi$. Then we define $r^\HT := r^{(\Phi, \varphi_0)}$ and $h_{G^\BQ}^\HT := h_{G^\BQ}$. In this case, we can be explicit about the reflex field: $\varphi_0$ identifies $F \isoarrow E_{r^{(\Phi, \varphi_0)}}$, unless $F_0 = \BQ$ and $n = 2$ (then $E_{r^{(\Phi, \varphi_0)}} = \BQ$) or $F_0$ is quadratic over $\BQ$ and $n = 1$ (then $E_{r^{(\Phi, \varphi_0)}}$ identifies via $\varphi_0$ with the unique quadratic subfield of $F$ distinct from $K_0$ and $F_0$).  Note that the book \cite{HT} is about the tower of moduli stacks $(M_K)$, and not about the Shimura variety $(S(G^\BQ, \{ h_{G^\BQ} \})_K)$ (despite the title of the book!).
\end{altenumerate}\label{eg:special types}
\end{example}

\subsection{The unitary group}\label{ss:U}
We continue with the notation of the last subsection, but this time we consider the group
\begin{equation}\label{G}
   G:=\Res_{F_0/\BQ}\U\bigl(W, \sform\bigr).
\end{equation}
The conjugacy class in the Shimura datum for $G$ is the conjugacy class of the homomorphism $h_G=(h_{G,\varphi})_{\varphi\in\Phi}$, where
\[
   h_{G, \varphi} \colon z\mapsto  \diag\bigl(1_{r_\varphi}, (\ov z/z)1_{{r_{\ov\varphi}}}\bigr),
\]
and where the components are defined as in \eqref{G^Q decomp}. The reflex field is the reflex field $E_{r^\natural}$ of the function $r^\natural$,
\begin{equation}
   \Gal(\ov\BQ/E_{r^\natural})= \bigl\{\, \sigma\in\Gal(\ov\BQ/\BQ)\bigm| \sigma^*(r^\natural)=r^\natural \,\bigr\},
\end{equation}
where we define
\[
   r^\natural \colon
   \begin{tikzcd}[baseline=(x.base), row sep=-.5ex, ampersand replacement=\&]
      |[alias=x]| \Hom_\BQ(F, \ov\BQ) \ar[r]  \&  \BZ_{\geq 0}\\
      \varphi \ar[r, mapsto]
        \&  \begin{cases}
              0,  &  \varphi\in \Phi;\\
              r_\varphi,  &  \varphi \in \ov\Phi.
           \end{cases}
   \end{tikzcd}
\]
(Note that $r^\natural$ need not be a generalized CM type.)  The resulting Shimura variety $(S(G, \{ h_G \})_K)$ is not of PEL type, i.e., it is not related to a moduli problem of abelian varieties (this can be seen already from the fact that the restriction of $\{ h_G \}$ to $\BG_m\subset \BS$ is not mapped via the identity map to the center of $G$). However, this Shimura variety is of abelian type.

\begin{example}\label{eg:fake drinfeld unitary}
In the fake Drinfeld case of Example \ref{eg:special types}\eqref{eg:fake drinfeld}, $\varphi_0$ embeds $F$ into $E_{r^\natural}$ for $n \geq 2$.  In the strict fake Drinfeld case relative to $(\Phi,\varphi_0)$ of Example \ref{eg:special types}\eqref{strictFK}, we have $\varphi_0 \colon F \isoarrow E_{r^\natural}$ for all $n \geq 1$.
\end{example}

\begin{remark}\label{GGPsit}
Suppose that $n \geq 2$, and let $u\in W$ be a totally positive vector, i.e., $(u, u)$ is a totally positive element in $F_0$.  Consider the hermitian space $W^\flat := (u)^\perp$.  Then the inclusion $W^\flat\subset W$ induces an inclusion $\U(W^\flat)\subset \U(W)$ of unitary groups (identifying $\U(W^\flat)$ with the stabilizer of $u$). Let
\[
   H := \Res_{F_0/\BQ}\U (W^\flat),
\]
with its Shimura datum $\{ h_H \}$. Using that $u$ is totally positive, one verifies that the inclusion $H\subset G$ is compatible with the Shimura data $\{ h_H \}$ and $\{ h_G \}$. Hence there is an induced morphism of Shimura varieties,
\begin{equation}\label{ggporig morph}
   \Bigl(S\bigl( H,\{h_{ H}\}\bigr)_{K_H}\Bigr) \inj \Bigl(S\bigl( G,\{h_{ G}\}\bigr)_{K_G}\Bigr).
\end{equation}
The resulting cycle leads to the \emph{Kudla--Rapoport cycles} \cite{KR-U2, BHKRY} (in the context of locally symmetric spaces, these are the  \emph{Kudla-Millson cycles}, cf.~\cite{KM}). These are cycles of codimension $\sum_{\varphi\in\Phi} r_{\ov\varphi}$. The most interesting case is the strict fake Drinfeld case of Example \ref{eg:special types}\eqref{strictFK}. In this case, we obtain divisors on the ambient variety. 

Taking the graph morphism of \eqref{ggporig morph}, we obtain a closed embedding of towers,
\begin{equation}\label{ggporig cycle}
   \Bigl(S\bigl( H,\{h_{ H}\}\bigr)_{K_{ H}}\Bigr) \inj \Bigl(S\bigl( H,\{h_{ H}\}\bigr)_{K_{ H}}\Bigr) \times \Bigl(S\bigl( G,\{h_{ G}\}\bigr)_{K_{ G}}\Bigr).
\end{equation}
The resulting cycle in the target of \eqref{ggporig cycle} is the \emph{GGP cycle}, cf.~\cite{GGP}. This is a cycle of codimension $\sum_{\varphi\in\Phi}r_\varphi r_{\ov\varphi}$. The most interesting case is again the strict fake Drinfeld case. In this case, the product variety has dimension $2n-3$, and the cycle has codimension $n-1$, i.e.,\ the codimension is just more than half the dimension of the ambient variety.

Both of these constructions generalize to the case where the totally positive vector $u$ is replaced by an $m$-tuple of totally positive vectors $u_1, \ldots, u_m$ which generate a totally definite subspace of $W$. 
 \end{remark}

\begin{remark}\label{adv/disadv}
Let us discuss some of the advantages and disadvantages of the above Shimura varieties. 
\begin{altenumerate}
\item First consider the Shimura varieties associated to $(G^\BQ, \{ h_{G^\BQ} \})$. On the positive side we note the following.

\smallskip

\begin{altitemize}
\item These Shimura varieties are close to moduli problems---but when $n$ is odd, they are not quite represented by a moduli problem in general. 
\item The method of Kottwitz works for all Shimura varieties of PEL type.
\end{altitemize}
\smallskip

\noindent On the negative side we note the following.

\smallskip
 
\begin{altitemize}
\item For odd $n$, the Shimura varieties are not given by a moduli problem in general.
\item It is difficult to construct integral or $p$-integral models of these Shimura varieties. More precisely, Kottwitz succeeds in constructing a $p$-integral model over $O_{E_r, (p)}$ only under the assumption that all data are unramified at $p$. This last condition means that $p$ is unramified in $F$, that $W$ is split at all places of $F_0$ over $p$, and that $K$ is of the form $K=K^p \times K_p \subset G^\BQ(\BA_f) = G^\BQ(\BA_f^p) \times G^\BQ(\BQ_p)$, where $K^p$ is arbitrary and $K_p$ is the stabilizer of a self-dual lattice in $W\otimes_\BQ \BQ_p$.  But allowing ramification in various forms leads to many new complications.
\item In the context of Remark \ref{GGPsit}, the   inclusion $\U(W^\flat)\subset \U(W)$ of unitary groups does not extend to an inclusion $\GU(W^\flat)\subset \GU(W)$ of groups of unitary similitudes, nor, after Weil restriction down to $\BQ$, to an inclusion of the unitary $\BG_m$-similitude group for $W^\flat$ into $G^\BQ$. Hence there is no Gan--Gross--Prasad set-up in the context of Kottwitz's Shimura varieties, comp.\ \cite{GGP, RSZ3}. Similarly, there are no Kudla--Rapoport cycles   on these Shimura varieties, cf.\ \cite{KR-U2, BHKRY}. See Section \ref{ss:cycles} for the analog of this discussion in the context of the RSZ Shimura varieties.
\end{altitemize}
\item Now let us discuss the Shimura varieties associated to $(G, \{ h_{G} \})$.  On the positive side we note the following.
\smallskip

\begin{altitemize}
 \item The KR cycles and  GGP cycles can be defined for them.
 \item In the strict fake Drinfeld case of Example \ref{eg:special types}\eqref{strictFK}, the Shimura field  is very simple: it identifies with $F$.
\end{altitemize}
\smallskip
\noindent On the negative side we note the following.
\smallskip

\begin{altitemize}
\item Since these Shimura varieties are not of PEL type, it is difficult to construct and control $p$-integral models of them. Since they are at least of abelian type, by Kisin--Pappas \cite{KP} they do have $p$-integral models when $K$ is of the form $K=K^p K_p$, where $K_p$ is a parahoric subgroup. However, these models are not very explicit. 
In particular, it seems difficult to address for  these $p$-integral models the  \emph{Arithmetic Gan--Gross--Prasad conjecture} \cite{GGP}, the \emph{Arithmetic intersection conjecture} of \cite{RSZ3}, and the \emph{Kudla--Rapoport intersection conjecture} \cite[Conj.~11.10]{KR-U2}.

\end{altitemize}
\end{altenumerate}
\end{remark}

\section{The RSZ Shimura varieties}\label{s:rsz}
We continue with the notation $F/F_0$, $r$, $r^\natural$, and $(W, \sform)$ from Sections \ref{ss:GU}--\ref{ss:U}. Again we fix a CM type $\Phi$ of $F$.

\subsection{The torus $Z^\BQ$ and its Shimura variety}\label{ss:Z^Q}
We are first going to consider the Shimura varieties of Section \ref{ss:GU} in the special case that $n = 1$ and $(W,\sform)= (W_0,\sform_0)$ is totally positive definite, i.e., $W_0$ has signature $(1,0)$ at each archimedean place.  In this case, we write $Z^\BQ := G^\BQ$ (a torus over \BQ) and $h_{Z^\BQ} := h_{G^\BQ}$.  Explicitly,
\[
   Z^\BQ = \bigl\{\, z\in \Res_{F/\BQ} \BG_m \bigm| \Nm_{F/F_0} (z)\in \BG_m \,\bigr\},
\]
and the homomorphism $h_{Z^\BQ}\colon \BC^\times \to Z^\BQ(\BR)$ identifies with the diagonal embedding into $(\BC^\times)^\Phi$ with respect to the isomorphism
\[
   Z^\BQ(\BR) \isoarrow \bigl\{\,(z_\varphi)\in (\BC^\times)^\Phi \bigm| |z_\varphi| = |z_{\varphi'}| \ \text{for all}\ \varphi,\varphi' \in \Phi \,\bigl\} \subset (\BC^\times)^\Phi
\]
induced by the isomorphism $\smash{F \otimes \BR \xra[\undertilde]\Phi \BC^\Phi}$.\footnote{We  note  that  \cite{RSZ3}  adopts  the  convention  that $h_{Z^\BQ}$ is  the  analogous  embedding  defined in the case that $W_0$ is totally \emph{negative} definite, which means that it is our $h_{Z^\BQ}$ precomposed by complex conjugation. This difference of convention  results in a number of further differences with \cite{RSZ3} throughout the rest of Section \ref{s:rsz}.\label{foot}}  The reflex field of $(Z^\BQ,\{h_{Z^\BQ}\})$ is $E_\Phi$, the reflex field of $\Phi.$

Let $K_{Z^\BQ} \subset Z^\BQ(\BA_f)$ be a compact open subgroup.  Then we obtain the Deligne--Mumford stack, which we denote by $M_{0,K_{Z^\BQ}}$, representing the moduli problem of Section \ref{ss:GU} for the hermitian space $W_0$. It is a finite \'etale stack over $\Spec E_\Phi$.

By Theorem \ref{thmGU}, the complex fiber $M_{0,K_{Z^\BQ}} \otimes_{E_\Phi} \BC$ is isomorphic to a finite number of copies of the Shimura variety $S(Z^\BQ, \{h_{Z^\BQ}\})_{K_{Z^\BQ}}$.  To make this decomposition more explicit, let us first introduce the following definition.

\begin{definition}
An element $a \in F$ is \emph{$\Phi$-adapted} if $\varphi(a)$ is an $\BR_{>0}$-multiple of $\sqrt{-1} \in \BC$ for all $\varphi \in \Phi$.
\end{definition}

Thus any $F/F_0$-traceless element\footnote{Here the notation $\sqrt\Delta$ reflects the fact that any $\Phi$-adapted element must be a square root of some totally negative element $\Delta \in F_0^\times$, but we note that the element $\Delta$ itself will never play any explicit role for us.}  $\sqrt\Delta \in F^\times$ determines a unique CM type for $F$ to which $\sqrt\Delta$ is adapted, and conversely, by weak approximation, any CM type admits elements adapted to it.  In particular, let us fix a $\Phi$-adapted element $\sqrt\Delta$ for our fixed CM type $\Phi$.  In the notation of the Introduction and Section \ref{ss:GU}, take $\zeta = \sqrt\Delta\i$, so that we endow $W_0$ with the $\BQ$-alternating form $\tr_{F/\BQ} \sqrt\Delta\i \sform_0$ in the definition of the level structure for $M_{0,K_{Z^\BQ}}$. Let $\CR_{W_0,\sqrt\Delta}$ be the set of isometry classes of pairs $(U_0,\aform_0)$ consisting of a one-dimensional $F$-vector space $U_0$ equipped with a nondegenerate $\BQ$-alternating form $\aform_0 \colon U_0 \times U_0 \to \BQ$ such that $\la ax, y \ra_0 = \la x, \ov a y\ra_0$ for all $x,y \in U_0$ and $a \in F$, such that $x \mapsto \la \sqrt\Delta x, x \ra_0$ is a positive definite quadratic form on $U_0$, and such that for all finite primes $p$, the localization $U_0 \otimes \BQ_p$ endowed with its $\BQ_p$-alternating form is $F \otimes \BQ_p$-linearly similar to $(W_0, \tr_{F/\BQ} \sqrt\Delta\i \sform_0)\otimes \BQ_p$ up to a factor in $\BQ_p^\times$. (Thus the pair $(W_0,\tr_{F/\BQ}\sqrt\Delta\i\sform_0)$ tautologically defines a class in $\CR_{W_0,\sqrt\Delta}$.)  Then $\BQ_{>0}$ acts on $\CR_{W_0,\sqrt\Delta}$ by multiplying the form, and by \cite[\S8]{K-points},
\begin{equation}\label{decompM0Cisog}
   M_{0,K_{Z^\BQ}} \otimes_{E_\Phi} \BC \simeq \coprod_{\CR_{W_0,\sqrt\Delta}/\BQ_{>0}} S\bigl(Z^\BQ, \{h_{Z^\BQ}\}\bigr)_{K_{Z^\BQ}}.
\end{equation}
(In terms of Theorem \ref{thmGU}, the set $\CR_{W_0,\sqrt\Delta}/\BQ_{>0}$ is in bijection with $\ker^1(\BQ,Z^\BQ)$ by taking the class of $(W_0,\smash[t]{\tr_{F/\BQ}\sqrt\Delta\i\sform_0})$ as basepoint.) Here the index associated to a $\BC$-valued point $(A_0,\iota_0,\lambda_0, \ov\eta_0)$ of $M_{0,K_{Z^\BQ}}$ is given by the $\BQ^\times$-class of the $F$-vector space $\RH_1(A_0, \BQ)$ endowed with its natural $\BQ$-valued Riemann form induced by $\lambda_0$.  The decomposition on the right-hand side of \eqref{decompM0Cisog} descends to $E_\Phi$, and we accordingly write
\begin{equation}\label{decompM0isog}
   M_{0,K_{Z^\BQ}} = \coprod_{\tau \in \CR_{W_0,\sqrt\Delta}/\BQ_{>0}} M_{0,K_{Z^\BQ}}^\tau.
\end{equation}

\subsection{The RSZ Shimura varieties}\label{ss:tilde G}
The Shimura varietes of \cite{RSZ3} are attached to the group
\begin{equation}\label{wtG}
   \wt G := Z^\BQ \times_{\BG_m} G^\BQ,
\end{equation}
where the maps from the factors on the right-hand side to $\BG_m$ are respectively given by $\Nm_{F/F_0}$ and the similitude character.  In terms of the Shimura data already defined, we obtain a Shimura datum for $\wt G$ by defining the Shimura homomorphism to be
\[
 h_{\wt G}\colon \BC^\times \xra{(h_{Z^\BQ}, h_{G^\BQ})} \wt G(\BR) .	
\]
It is easy to see that $(\wt G, \{h_{\wt G}\})$ has reflex field $E \subset \ov\BQ$ characterized by
\begin{equation}\label{defE}
\begin{aligned}
   \Gal(\ov\BQ/E) &= \bigl\{\, \sigma\in\Gal(\ov\BQ/\BQ) \bigm| \sigma \circ\Phi=\Phi \text{ and } \sigma^*(r)=r \,\bigr\}\\
                  &= \bigl\{\, \sigma\in\Gal(\ov\BQ/\BQ) \bigm| \sigma \circ\Phi=\Phi \text{ and } \sigma^*(r^\natural)=r^\natural \,\bigr\} .
\end{aligned}
\end{equation}
In other words, the reflex field is the common composite $E = E_\Phi E_r = E_\Phi E_{r^\natural}$.

\begin{example}\label{eg:fake drinfeld E}
In the fake Drinfeld case of Example \ref{eg:special types}\eqref{eg:fake drinfeld}, $\varphi_0$ embeds $F$ into $E$ for $n \geq 2$ since $\varphi_0\colon F \to E_{r^\natural} \subset E$, cf.\ Example \ref{eg:fake drinfeld unitary}.  When $n = 1$, the same statement holds in  the strict fake Drinfeld case relative to $\Phi$ of Example \ref{eg:special types}\eqref{strictFK}, but $F$ may fail to embed in $E$ for other signature types of fake Drinfeld type.  In the Harris--Taylor case of Example \ref{eg:special types}\eqref{eg:HT}, we have $\varphi_0 \colon F \isoarrow E$ for any $n \geq 1$.
\end{example}

The various relations between the groups we have introduced give rise to the following relations between Shimura varieties.
\begin{altenumerate}
\item By definition, the natural projection $\wt G \to G^\BQ$ induces a morphism of Shimura data $(\wt G, \{h_{\wt G}\}) \to (G^\BQ, \{h_{G^\BQ}\})$.  Hence there is an induced morphism of Shimura varieties (i.e.,\ a morphism of pro-varieties)
\[
   \Bigl( S\bigl(\wt G, \{h_{\wt G}\}\bigr)_{K_{\wt G}} \Bigr) \to \Bigl( S\bigl(G^\BQ, \{h_{G^\BQ}\}\bigr)_{K_{G^\BQ}} \Bigr),
\]
compatible with the inclusion $E_r\subset E$.
\item The torus $Z^\BQ$ embeds naturally as a central subgroup of $G^\BQ$, which gives rise to a product decomposition
\begin{equation}\label{proddec}
   \begin{gathered}
   \begin{tikzcd}[row sep=0ex]
      \wt G \ar[r, "\sim"]  &  Z^\BQ \times G\\
      (z, g) \ar[r, mapsto]  &  (z, z^{-1}g),
   \end{tikzcd}
   \end{gathered}
\end{equation}
where $G \subset G^\BQ$ is the unitary group \eqref{G}.  The isomorphism \eqref{proddec} extends to a product decomposition of Shimura data,
\begin{equation}\label{reltogross}
   \bigl(\wt G, \{h_{\wt G}\}\bigr) \cong \bigl(Z^\BQ,\{h_{Z^\BQ}\}\bigr) \times \bigl(G,\{h_G\}\bigr),
\end{equation}
and hence there is a product decomposition of Shimura varieties,
\[
   \Bigl( S\bigl(\wt G, \{h_{\wt G}\}\bigr)_{K_{\wt G}} \Bigr) 
      \cong \Bigl( S\bigl(Z^\BQ,\{h_{Z^\BQ}\}\bigr)_{K_{Z^\BQ}} \Bigr) \times \Bigl( S\bigl(G,\{h_G\}\bigr)_{K_G} \Bigr),
\]
compatible with the inclusions $E_\Phi \subset E$ and $E_{r^\natural} \subset E$.
\end{altenumerate}

\subsection{The RSZ moduli problem in terms of isogeny classes}
We are now going to give a moduli interpretation for the canonical model of the Shimura variety $S(\wt G, \{h_{\wt G}\})_{K_{\wt G}}$ over $\Spec E$.  We will only consider subgroups $K_{\wt G}$ which, with respect to the product decomposition \eqref{proddec}, are of the form
\begin{equation}\label{K_wtG*}
   K_{\wt G} = K_{Z^\BQ} \times K_G \subset \wt G(\BA_f) = Z^\BQ(\BA_f) \times G(\BA_f),
\end{equation}
for arbitrary open compact subgroups $K_{Z^\BQ} \subset Z^\BQ(\BA_f)$ and $K_G\subset G(\BA_{f})$.

For the definition of level structures in the moduli problem, we fix a one-dimensional, totally positive definite $F/F_0$-hermitian space $(W_0,\sform_0)$ as in Section \ref{ss:Z^Q}.  We fix a $\Phi$-adapted element $\sqrt\Delta \in F$ and a class $\tau \in \CR_{W_0,\sqrt\Delta}/\BQ_{>0}$, and we recall from \eqref{decompM0isog} the stack $M_{0,K_{Z^\BQ}}^\tau$ attached to $W_0$. (Of course we may take $\tau$ to be the class of $(W_0,\tr_{F/\BQ}\sqrt\Delta\i\sform_0)$, but we do not require this.) We furthermore introduce the $n$-dimensional $F$-vector space
\begin{equation}\label{V}
   V := \Hom_F(W_0,W).
\end{equation}
The space $V$ carries a natural $F/F_0$-hermitian form, under which elements $x,y \in V$ pair to the composite
\begin{equation}\label{V pairing}
   \Bigl[W_0 \xra{x} W \xra{w \shortmapsto (-,w)} W^* \xra{y^*} W_0^* \xra{[w_0 \shortmapsto (-,w_0)_0]\i} W_0 \Bigr] \in \End_F(W_0) \cong F.
\end{equation}
The group $\wt G$ acts naturally by unitary transformations on $V$, given in terms of the defining presentation \eqref{wtG} by $(z,g) \cdot x = gxz\i$. This action factors through the quotient $\wt G \to G$ via \eqref{proddec} and induces $G \cong \Res_{F_0/\BQ}\U(V)$.

We define the following category fibered in groupoids $\CF_{K_{\wt G}}(\wt G)$ over $\LNSch_{/E}$. To lighten notation, we suppress the dependence of this category functor on the element $\tau$.

\begin{definition}\label{def:isogtuple}
The category functor $\CF_{K_{\wt G}}(\wt G)$ associates to each scheme $S$ in $\LNSch_{/E}$ the groupoid of tuples $(A_0,\iota_0,\lambda_0,\ov\eta_0, A,\iota,\lambda,\ov\eta)$, where
\begin{altitemize}
\item $(A_0, \iota_0, \lambda_0, \ov\eta_0)$ is an object of $M_{0,K_{Z^\BQ}}^\tau(S)$;
\item $A$ is an abelian scheme over $S$;
\item $\iota\colon F\to\End^0(A)$ is an action of $F$ on $A$ up to isogeny satisfying the Kottwitz condition \eqref{kottcond}; 
\item $\lambda$ is a quasi-polarization on $A$ whose Rosati involution satisfies condition \eqref{Ros}; and
\item $\ov\eta$ is a $K_G$-orbit (equivalently, a ${K_{\wt G}}$-orbit, where $K_{\wt G}$ acts through its projection $K_{\wt G} \to K_G$) of isometries of $\BA_{F, f}/\BA_{F_0,f}$-hermitian modules 
\begin{equation}\label{eta}
   \eta\colon \wh \RV(A_0, A) \isoarrow V\otimes_F\BA_{F, f}.
\end{equation}
\end{altitemize}
Here
\begin{equation}\label{V(A_0,A)}
   \wh \RV(A_0,A) := \Hom_{\BA_{F,f}}\bigl(\wh \RV(A_0), \wh \RV(A)\bigr) ,
\end{equation}
endowed with its natural $\BA_{F,f}$-valued hermitian form $h$,
\begin{equation}\label{h_A def}
   h(x, y) := \lambda_0\i \circ y^\vee\circ\lambda\circ x\in\End_{\BA_{F, f}}\bigl(\wh \RV(A_0)\bigr)=\BA_{F, f}, \quad x,y \in \wh \RV(A_0,A),
\end{equation}
cf.\ \cite[\S2.3]{KR-U2}.\footnote{To be clear, $y^\vee \colon \wh\RV(A^\vee) \to \wh\RV(A_0^\vee)$ denotes the adjoint of $y$ with respect to the Weil pairings on $\wh\RV(A) \times \wh\RV(A^\vee)$ and $\wh\RV(A_0) \times \wh\RV(A_0^\vee)$.} Furthermore, for any geometric point $\ov s \to S$, the orbit $\ov\eta$ is required to be $\pi_1(S,\ov s)$-stable with respect to the $\pi_1(S,\ov s)$-action on the fiber $\wh \RV(A_0,A)(\ov s) = \Hom_{\BA_{F,f}}(\wh \RV(A_{0, \ov s}), \wh \RV(A_{\ov s}))$ (a condition which holds for all $\ov s$ on a given connected component $S_0$ of $S$ as soon as it holds for a single $\ov s$ on $S_0$); comp.\ \cite[\S5]{K-points} or \cite[Rem.\ 4.2]{KR-U2}.  A morphism $(A_0,\iota_0,\lambda_0,\ov\eta_0,A,\iota,\lambda,\ov\eta) \to (A'_0,\iota'_0,\lambda'_0,\ov\eta_0',A',\iota',\lambda',\ov\eta')$ in this groupoid is given by a pair of $F$-linear quasi-isogenies $\mu_0 \colon A_0 \to A_0'$ and $\mu\colon A \to A'$ such that $\mu_0$ is an isomorphism $(A_0,\iota_0,\lambda_0,\ov\eta_0) \isoarrow (A_0',\iota_0',\lambda_0',\ov\eta_0')$ in $M_{0,K_{Z^\BQ}}^\tau(S)$, such that $\mu^*(\lambda')$ is the same $\BQ^\times$-multiple of $\lambda$ as $\mu_0^*(\lambda_0')$ is of $\lambda_0$ at each point of $S$ (the \emph{multiplier condition}), and such that under the natural isomorphism $\wh\RV(A_0,A) \isoarrow \wh\RV(A_0',A')$ sending $x \mapsto \mu \circ x \circ \mu_0\i$ (which is an isometry by the multiplier condition), $\ov\eta'$ pulls back to $\ov \eta$.
\end{definition}

\begin{remark}\label{V rem}
Let us comment further on the space $V$ introduced in \eqref{V}.  The hermitian forms $\sform_0$ and $\sform$ determine a conjugate-linear isomorphism $V \isoarrow \Hom_F(W,W_0)$, $x \mapsto x^\ad$, characterized by the formula
\[
   (xw_0,w) = (w_0,x^\ad w)_0,
   \quad
   x \in V,\ w_0 \in W_0,\ w \in W.
\]
Then the pairing \eqref{V pairing} on $V$ can be expressed succinctly as sending $x,y$ to $y^\ad x$.  Alternatively, the adjoint $x^\ad$ defined in this way is the same as the adjoint with respect to the $\BQ$-valued forms $\tr_{F/\BQ} \sqrt\Delta\i \sform$ and $\tr_{F/\BQ} \sqrt\Delta\i \sform_0$.

Concretely, upon choosing a basis vector in $W_0$, the hermitian form on $W_0$ is represented by a totally positive element $a \in F_0$, and we obtain an $F$-linear isomorphism $V \simeq W$.  With respect to this isomorphism, the form on $V$ is then given by $a\i \sform$.  In particular, $V$ has the same signature at each archimedean place as $W$, and in the special case that $W_0$ equals $F$ endowed with its norm form, there is a canonical isometry $V \cong W$.  This last case recovers the case taken in the definition of level structures in \cite[\S3.2]{RSZ3} (modulo the sign conventions alluded to previously in footnote \ref{foot}).  But even in this special case, it is often helpful to distinguish between $V$ and $W$, and more generally, it can be desirable to allow other possibilities for $W_0$.
\end{remark}

The following theorem is the analog for $(\wt G, \{h_{\wt G}\})$ of Theorems \ref{thmSp} and \ref{thmGU}.

\begin{theorem}\label{moduliforshim}
The moduli problem $\CF_{K_{\wt G}}(\wt G)$ is representable by a Deligne--Mumford stack  $M_{K_{\wt G}}(\wt G)$ over $\Spec E$,  and   
\[
   M_{K_{\wt G}}(\wt G)(\BC) = S\bigl(\wt G, \{h_{\wt G}\}\bigr)_{K_{\wt G}},
\]
compatible with changing $K_{\wt G}$ of the form \eqref{K_wtG}.
\end{theorem}

\begin{proof}
This is the extension to the case of arbitrary signature types of \cite[Prop.\ 3.7]{RSZ3}.\footnote{Strictly speaking, the statement and proof in loc.\ cit.\ is for the moduli problem given in Definition \ref{def:tuple} below, in the case of a particular signature type. But the argument transposes to the present situation almost unchanged.} The key point is the following. Define for $(A_0,\iota_0,\lambda_0,\ov\eta_0,A,\iota,\lambda,\ov\eta)$ in  $M_{K_{\wt G}}(\wt G)(\BC)$ a hermitian space ${\rm V}(A_0, A)$ over $F$ in analogy with  $\wh \RV(A_0,A)$, but by using Betti homology groups instead of rational Tate modules. Then $ \wh \RV(A_0,A)={\rm V}(A_0, A)\otimes_F \BA_{F, f}$. By the level structure $\ov\eta$, the two hermitian spaces $V$ and ${\rm V}(A_0, A)$ are isomorphic at all finite places. At an archimedean place corresponding to $\varphi\in\Phi$, by the Kottwitz condition \eqref{kottcond} and the analogous condition of signature $((1,0)_{\varphi \in \Phi})$ for $M_{0,K_{Z^\BQ}}$, the signature of ${\rm V}(A_0, A)$ is $(r_\varphi, r_{\ov\varphi})$. Hence, by the Hasse principle for hermitian spaces, ${\rm V}(A_0, A)$ and $V$ are isomorphic. The choice of an isomorphism $j$ between them allows one to define a map $M_{K_{\wt G}}(\wt G)(\BC) \to S(\wt G, \{h_{\wt G}\})_{K_{\wt G}}$ which one shows to be an isomorphism independent of the choice of $j$.
\end{proof}

\subsection{Variant moduli problems in terms of isomorphism classes}\label{ss:isom variant}
In this section we give some ``isomorphism class'' variants of the moduli problems introduced above.

We begin with the moduli problem for $Z^\BQ$.  For simplicity, we restrict to the case that the level subgroup $K_{Z^\BQ} = K_{Z^\BQ}^\circ \subset {Z^\BQ}(\BA_f)$ is the (unique) maximal compact open subgroup,
\begin{equation}\label{K_Z^BQ}
   K^\circ_{Z^\BQ} := \bigl\{\,z\in (O_F\otimes \wh\BZ)^\times \bigm| \Nm_{F/F_0}(z)\in \wh\BZ^\times \,\bigr\}.
\end{equation}
We define the following category fibered in groupoids $\CF_0$ over $\LNSch_{/E_\Phi}$.

\begin{definition}\label{def:F_0}
The category functor $\CF_0$ associates to each scheme $S$ in $\LNSch_{/E_\Phi}$ the groupoid of triples $(A_0, \iota_0, \lambda_0)$, where 
\begin{altitemize}
\item $A_0$ is an abelian scheme over $S$;
\item $\iota_0\colon O_F\to \End(A_0)$ is an $O_F$-action satisfying the Kottwitz condition \eqref{kottcond} in the case of signature $((1,0)_{\varphi\in\Phi})$ for elements in $O_F$,
\begin{equation}\label{kottcondA_0}
   \charac\bigl(\iota(a)\mid\Lie A_0\bigr) = \prod_{\varphi\in\Phi}\bigl(T-\varphi(a)\bigr)
   \quad\text{for all}\quad
   a\in O_F;
\end{equation}
and
\item $\lambda_0$ is a principal polarization on $A_0$ whose Rosati involution satisfies condition \eqref{Ros} on $O_F$ with respect to $\iota_0$.
\end{altitemize}
A morphism $(A_0,\iota_0,\lambda_0) \to (A_0',\iota_0',\lambda_0')$ in this groupoid is an $O_F$-linear isomorphism of abelian schemes $\mu_0\colon A_0 \isoarrow A_0'$ such that the pullback of $\lambda_0'$ is $\lambda_0$.
\end{definition}

By  the proof of \cite[Prop.\ 3.1.2]{Ho-kr}, $\CF_0$ is representable by a DM stack $M_0$ which is finite and \'etale over $\Spec E_\Phi$.
 
Unfortunately, it may happen that $M_0$ is empty. In order to circumvent this issue, we introduce the following variant of $M_0$, cf.\ \cite[Def.\ 3.1.1]{Ho-kr}. Fix a non-zero ideal $\fka$ of $O_{F_0}$. Then we define the Deligne--Mumford stack $M_0^\fka$  of triples $(A_0, \iota_0, \lambda_0)$ as in Definition \ref{def:F_0}, except that we replace the condition that $\lambda_0$ is principal by the condition that $\lambda_0$ is a polarization satisfying $\ker \lambda_0=A_0[\fka]$. Then, again, $M_0^\fka$ is finite and \'etale over $\Spec E_\Phi$, cf.\ \cite[Prop.\ 3.1.2]{Ho-kr}.  

If $M_0^\fka$ is non-empty, then, like the case of the moduli stack $M_{0,K_{Z^\BQ}^\circ}$ in Section \ref{ss:Z^Q}, its complex fiber is a finite disjoint union of copies of $S(Z^\BQ, \{h_{Z^\BQ}\})_{K_{Z^\BQ}^\circ}$. More precisely, let $\CL_\Phi^\fka$ be the set of isomorphism classes of pairs $(\Lambda_0,\aform_0)$ consisting of a locally free $O_F$-module $\Lambda_0$ of rank one equipped with a nondegenerate alternating form $\aform_0\colon \Lambda_0 \times \Lambda_0 \to \BZ$ such that $\la ax, y \ra_0 = \la x, \ov a y\ra_0$ for all $x,y \in \Lambda_0$ and $a \in O_F$, such that the dual lattice $\Lambda_0^\vee$ of $\Lambda_0$ inside $\Lambda_0 \otimes_\BZ \BQ$ equals $\fka\i \Lambda_0$, and such that $x \mapsto \la \sqrt\Delta x, x \ra_0$ is a positive definite quadratic form on $\Lambda_0$ for some (equivalently, any) $\Phi$-adapted element $\sqrt\Delta \in F$. Then $\CL_\Phi^\fka$ is a finite set, in natural bijection with the isomorphism classes of objects in $M_0^\fka(\BC)$, cf.\ \cite[\S3.2]{RSZ3}.  Given $\Lambda_0,\Lambda_0' \in \CL_\Phi^\fka$ (as is customary, we often suppress the pairings when denoting elements in $\CL_\Phi^\fka$), define $\Lambda_0 \sim \Lambda_0'$ if $\Lambda_0 \otimes_\BZ \wh\BZ$ and $\Lambda_0' \otimes_\BZ \wh\BZ$ are $\wh O_F$-linearly similar up to a factor in $\wh\BZ^\times$ and $\Lambda_0 \otimes_\BZ \BQ$ and $\Lambda_0' \otimes_\BZ \BQ$ are $F$-linearly similar up to a (necessarily positive) factor in $\BQ^\times$. Then
\begin{equation}\label{decompM0C}
   M_0^\fka \otimes_{E_\Phi} \BC\simeq \coprod_{\CL_\Phi^\fka/{\sim}} S\bigl(Z^\BQ, \{h_{Z^\BQ}\}\bigr)_{K_{Z^\BQ}^\circ},
\end{equation}
cf.\ Lem.\ 3.4 and the paragraph following it in \cite{RSZ3}.
Here the index associated to a $\BC$-valued point $(A_0,\iota_0,\lambda_0)$ of $M_0^\fka $ is given by the class in $\CL_\Phi^\fka/{\sim}$ of the $O_F$-module $\RH_1(A_0, \BZ)$ endowed with its natural $\BZ$-valued Riemann form induced by the polarization.  The decomposition on the right-hand side of \eqref{decompM0C} descends to $E_\Phi$, and we accordingly write
\begin{equation}\label{decompM0}
   M_0^\fka = \coprod_{\xi \in \CL_\Phi^\fka/{\sim}} M_0^{\fka,\xi}.
\end{equation}

\begin{remark}
\begin{altenumerate}
\item\label{rem exa satisfied} If $F/F_0$ is ramified at some finite place, then $M_0^\fka$ is non-empty for any $\fka$, cf.\ \cite[pf.\ of Prop.\ 3.1.6]{Ho-kr}.  A special case of this is when $F = K_0 F_0$, where $K_0$ is an imaginary quadratic field and  the discriminants of $K_0/\BQ$ and $F_0/\BQ$ are relatively prime.
\item\label{Ma unram case} If $F/F_0$ is unramified at every finite place and $M^{O_{F_0}}_0 = \emptyset$, then it is easy to deduce from loc.\ cit.\ and class field theory that $M_0^\fkp$ is non-empty for any prime ideal $\fkp \subset O_{F_0}$ which is inert in $F$.  For example, this case arises when $F_0 = \BQ(\sqrt 3)$ and $F = F_0(\sqrt{-1})$.
\item\label{fka rel pr} In particular, given finitely many prime numbers $p_1, \ldots, p_r$, there always exists $\fka$ relatively prime to $p_1,\ldots, p_r$ such that $M_0^\fka$ is non-empty.
\item When $F_0=\BQ$, the set $\CL_\Phi^\fka/{\sim}$ has only one element, so that the decomposition \eqref{decompM0} is trivial.
\end{altenumerate}\label{remMa}
\end{remark}

To directly compare $M_0^\fka$ and $M_{0,K_{Z^\BQ}^\circ}$ (or more precisely, the summands occurring on the respective right-hand sides of \eqref{decompM0isog} and \eqref{decompM0}), let $\fka$ be such that $M_0^\fka \neq \emptyset$, and let $\sqrt\Delta$ be any $\Phi$-adapted element in $F$. Fix a class $\xi \in \CL_\Phi^\fka/{\sim}$.  Let $(\Lambda_0,\aform_0)$ be a representative of $\xi$ in $\CL_\Phi^\fka$, and set $W_0 := \Lambda_0 \otimes_\BZ \BQ$.  Let $\tau$ denote the class of $(W_0, \aform_0\otimes\BQ)$ in $\CR_{W_0,\sqrt\Delta}/\BQ_{>0}$. (Here we are implicitly endowing $W_0$ with the unique $F/F_0$-hermitian form $\sform_0$ such that $\aform_0 \otimes \BQ = \tr_{F/\BQ} \sqrt\Delta\i \sform_0$.)  Then the set $\CR_{W_0,\sqrt\Delta}$ and the class $\tau$ are independent of the choice of representative of $\xi$.  We define an isomorphism
\begin{equation}\label{compisom}
   M_0^{\fka,\xi} \isoarrow M_{0,K_{Z^\BQ}^\circ}^\tau
\end{equation}
as follows.  Let $S$ be a locally noetherian $E_\Phi$-scheme, and let $(A_0,\iota_0,\lambda_0)$ be an $S$-point on $M_0^{\fka,\xi}$.  By the definition of the summands in the decomposition \eqref{decompM0} (see \cite[pf.\ of Lem.\ 3.4]{RSZ3}), at each geometric point $\ov s$ of $S$ there exists an $\wh O_F$-linear symplectic similitude (up to a factor in $\wh \BZ$)
\[
   \wh\RT(A_{0,\ov s}) \isoarrow \Lambda_0 \otimes_\BZ \wh\BZ.
\]
The set of all such similitudes is a $K_{Z^\BQ}^\circ$-orbit, and upon extending scalars to $\BA_f$ they define a level structure $\ov\eta_0$ of similitudes
\[
   \wh\RV(A_0) \isoarrow W_0 \otimes_\BQ \BA_f.
\]
Then the morphism \eqref{compisom} sends $(A_0,\iota_0,\lambda_0) \mapsto (A_0,\iota_0,\lambda_0,\ov\eta_0)$.  This morphism is an isomorphism by an obvious modification of the argument in \cite[Prop.\ 4.4]{KR-U2}, or see \cite[Prop.\ 1.4.3.4]{Lan}.

Keeping $W_0$ fixed, it is not hard to show that every class $\tau' \in \CR_{W_0,\sqrt\Delta}/\BQ_{>0}$ is represented by a space of the form $\Lambda_0' \otimes_\BZ \BQ$ for some $\Lambda_0' \in \CL_\Phi^\fka$. (Since we will make no essential use of this fact later in the paper, we leave the details to the reader.)  Choosing such a $\Lambda_0'$ for each $\tau'$, and taking the $\sim$-class of $\Lambda_0'$, we obtain an injection $\CR_{W_0,\sqrt\Delta}/\BQ_{>0} \inj \CL_\Phi^\fka/{\sim}$.  In this way, combined with the previous paragraph, we may identify $M_{0,K_{Z^\BQ}^\circ}$ with an open and closed substack of $M_0^\fka$.  (Note however that the choice of each $\Lambda_0'$, and hence the embedding $M_{0,K_{Z^\BQ}^\circ} \inj M_0^\fka$, is not canonical.)

We now turn to a couple of variants of the moduli problem attached to $\wt G$ in Definition \ref{def:isogtuple}.  We consider a subgroup $K_{\wt G}$ of the form \eqref{K_wtG*} with $K_{Z^\BQ} = K_{Z^\BQ}^\circ$, so that
\begin{equation}\label{K_wtG}
   K_{\wt G} = K^\circ_{Z^\BQ} \times K_G,
\end{equation}
still with $K_G \subset G(\BA_f)$ an arbitrary open compact subgroup.  Fix $\fka$, $\sqrt\Delta$, $\xi$, and $W_0$ all as before \eqref{compisom}. Set $V := \Hom_F(W_0,W)$, endowed with its natural hermitian form. We define the following category fibered in groupoids $\CF_{K_{\wt G}}'(\wt G)$ over $\LNSch_{/E}$. To lighten notation, we suppress the dependence on the ideal $\fka$ and the element $\xi$.

\begin{definition}\label{def:tuple}
The category functor $\CF_{K_{\wt G}}'(\wt G)$ associates to each scheme $S$ in $\LNSch_{/E}$ the groupoid of tuples $(A_0,\iota_0,\lambda_0,A,\iota,\lambda,\ov\eta)$, where
\begin{altitemize}
\item $(A_0, \iota_0, \lambda_0)$ is an object of $M^{\fka, \xi}_0(S)$; and
\item the tuple $(A,\iota,\lambda,\ov\eta)$ is as in Definition \ref{def:isogtuple}.
\end{altitemize}
A morphism $(A_0,\iota_0,\lambda_0,A,\iota,\lambda,\ov\eta) \to (A'_0,\iota'_0,\lambda'_0,A',\iota',\lambda',\ov\eta')$ in this groupoid is  given by an  isomorphism $\mu_0\colon (A_0, \iota_0, \lambda_0) \isoarrow (A'_0, \iota'_0, \lambda'_0)$ in $M_0^{\fka,\xi}(S)$  and an $F$-linear quasi-isogeny $\mu\colon A\to A'$ pulling $\lambda'$ back to $\lambda$ and $\ov\eta'$ back to $\ov\eta$. 
\end{definition}

The morphism \eqref{compisom} induces a natural comparison morphism of category functors,
\begin{equation}\label{compisom2}
   \begin{tikzcd}[row sep=0ex]
      \CF_{K_{\wt G}}'(\wt G) \ar[r]  &  \CF_{K_{\wt G}}(\wt G)\\
      (A_0,\iota_0,\lambda_0,A,\iota,\lambda,\ov\eta) \ar[r, mapsto]  &  (A_0,\iota_0,\lambda_0,\ov\eta_0,A,\iota,\lambda,\ov\eta).
   \end{tikzcd}
\end{equation}
The fact that \eqref{compisom} is an isomorphism easily implies that \eqref{compisom2} is an isomorphism as well.  In this way, $\CF_{K_{\wt G}}'(\wt G)$ gives another moduli interpretation of the stack $M_{K_{\wt G}}(\wt G)$ over $\Spec E$, cf.\ Theorem \ref{moduliforshim}.

We give a third moduli interpretation of $M_{K_{\wt G}}(\wt G)$, in which \emph{all} of the data is taken up to isomorphism, as follows.  We continue with $\fka$, $\xi$, $\sqrt\Delta$, $W_0$, and $V$ as above.  Let $\Lambda_0 \in \CL_\Phi^\fka$ denote the representative of $\xi$ used to define $W_0$ as before \eqref{compisom}. Fix any $O_F$-lattice $\Lambda \subset W$, and define the $O_F$-lattice
\[
   L := \Hom_{O_F}(\Lambda_0,\Lambda) \subset V.
\]
We again take the subgroup $K_{\wt G} = K^\circ_{Z^\BQ} \times K_G$ of the form \eqref{K_wtG}, and we assume that $L \otimes_{O_F} \wh O_F$ is $K_G$-stable inside $V \otimes_F \BA_{F,f}$ (which is equivalent to $\Lambda \otimes_{O_F} \wh O_F$ being $K_G$-stable inside $W \otimes_F \BA_{F,f}$).  Let $N$ be a positive integer such that the principal congruence subgroup mod $N$ for $L$,
\begin{equation}\label{K^L,N}
   K^{L,N} := \bigl\{\, g \in G(\BA_f) \bigm| (g-1) \cdot L \otimes_{O_F} \wh O_F \subset N L \otimes_{O_F} \wh O_F \,\bigr\},
\end{equation}
is contained in $K_G$.  We define the following moduli problem.  As before we suppress the dependence on $\fka$ and $\xi$ in the notation.

\begin{definition}\label{isom variant}
The category functor $\CF_{K_{\wt G}}^{L,N}(\wt G)$ associates to each scheme $S$ in $\LNSch_{/E}$ the groupoid of tuples $(A_0,\iota_0,\lambda_0,B,\iota,\lambda,\ov\eta_N)$, where
\begin{altitemize}
\item $(A_0, \iota_0, \lambda_0)$ is an object of $M^{\fka, \xi}_0(S)$;
\item $B$ is an abelian scheme over $S$;
\item $\iota\colon O_F\to\End(B)$ is an action of $O_F$ on $B$ satisfying the Kottwitz condition \eqref{kottcond} for all $a \in O_F$; 
\item $\lambda$ is a quasi-polarization on $B$ whose Rosati involution satisfies condition \eqref{Ros} for all $a \in O_F$; and
\item $\ov\eta_N$ is an \'etale closed subscheme
\[
   \ov\eta_N \subset \uIsom_{O_F} \bigl(\uHom_{O_F}(A_0[N],B[N]) , (L/N L)_S\bigr)
\]
over $S$ such that for every geometric point $\ov s \to S$, the fiber $\ov\eta_N(\ov s)$ identifies with a 
$K_G/K^{L,N}$-orbit of isomorphisms
\[
   \eta_N(\ov s) \colon \Hom_{O_F}\bigl(A_0[N](\ov s),B[N](\ov s)\bigr) \isoarrow L/N L
\]
which lift to $\wh O_F$-linear isometries of hermitian modules\footnote{Note that we have made no assumption on the restriction of the hermitian form on $V$ to $L$; all we can say is that this restriction takes values in some fractional ideal $\fkd$ of $F$. Similarly, since $\lambda_0$ need not be principal and $\lambda$ is only required to be a quasi-polarization, the hermitian form on $\wh\RT(A_0,B)$ need not be $\wh O_F$-valued.}
\begin{equation}\label{lift isom}
   \wh\RT(A_0, B)(\ov s) \isoarrow L\otimes_{O_F}\wh O_F.
\end{equation}
\end{altitemize}
Here
\begin{equation}\label{T(A0,B)}
   \wh\RT(A_0,B) := \Hom_{\wh O_F}\bigl(\wh\RT(A_0), \wh\RT(B)\bigr),
\end{equation}
regarded as a smooth $\wh O_F$-sheaf on $S$, and endowed with its natural hermitian form as in \eqref{h_A def}.  Furthermore, the notion of ``lift'' is with respect to the evident reduction-mod-$N$ maps $\wh\RT(A_0, B)(\ov s) \surj \Hom_{O_F}(A_0[N](\ov s),B[N](\ov s))$ and $L\otimes_{O_F}\wh O_F \surj L/NL$.  A morphism
\[
   (A_0,\iota_0,\lambda_0,B,\iota,\lambda,\ov\eta_N) \to (A'_0,\iota'_0,\lambda'_0,B',\iota',\lambda',\ov\eta_N')
\]
in this groupoid is given by an isomorphism $\mu_0\colon (A_0, \iota_0, \lambda_0) \isoarrow (A'_0, \iota'_0, \lambda'_0)$ in $M_0^{\fka,\xi}(S)$ and an $O_F$-linear isomorphism of abelian schemes $\mu\colon B\isoarrow B'$ pulling $\lambda'$ back to $\lambda$ and $\ov\eta_N'$ back to $\ov\eta_N$. 
\end{definition}

\begin{remark}
As usual, the condition on the level structure $\ov\eta_N$ in Definition \ref{isom variant} holds for all geometric points $\ov s \to S$ as soon as it holds for a single geometric point on each connected component of $S$.  The proof is similar to \cite[Lems.\ 1.3.6.5, 1.3.6.6, Cor.\ 1.3.6.7]{Lan}.
\end{remark}

\begin{remark}\label{relpos}
Note that the (quasi-)polarization type of $\lambda$ in Definition \ref{isom variant} is determined by $\Lambda$, in the sense that the existence of the isometries \eqref{lift isom} implies that $\wh\RT(A)$ and $\wh\RT(A)^\vee$ (the dual lattice of $\wh\RT(A)$ inside $\wh\RV(A)$ with respect to $\lambda$ and the Weil pairing) have the same relative position as $\Lambda$ and $\Lambda^\vee$ (the dual lattice of $\Lambda$ inside $W$ with respect to $\tr_{F/\BQ}\sqrt\Delta\i\sform$) have in $W$.  In particular, $\lambda$ is an honest polarization if and only if $\Lambda \subset \Lambda^\vee$.
\end{remark}

\begin{remark}\label{V rem 2}
In the special case that $W_0 = F$ with $(x,y)_0 = \Nm_{F/F_0} (x\ov y)$, take $\fka = \sqrt\Delta\i \fkD$, where $\fkD$ denotes the different of $F/\BQ$, and where $\sqrt\Delta$ is any $\Phi$-adapted element such that $\fka$ is an integral ideal.  Then $\fka$ is the image in $O_F$ of an ideal in $O_{F_0}$, and $O_F^\vee = \fka\i O_F$.  Hence we may take $\xi$ to be the class defined by $\Lambda_0 = O_F$ in the above discussion, and we obtain canonical isometries $V \cong W$ and $L \cong \Lambda$.   Thus in this case, one may formulate Definition \ref{isom variant} purely in terms of $\Lambda$, without needing to introduce $L$.
\end{remark}

The moduli problem $\CF_{K_{\wt G}}^{L,N}(\wt G)$ is related to $\CF_{K_{\wt G}}'(\wt G)$ via a natural comparison equivalence
\begin{equation}\label{comp isom}
\begin{gathered}
   \varphi \colon
   \begin{tikzcd}[baseline=(x.base),row sep=0ex]
      |[alias=x]| \CF_{K_{\wt G}}^{L,N}(\wt G) \ar[r, "\sim"]  &  \CF_{K_{\wt G}}'(\wt G)\\
      (A_0,\iota_0,\lambda_0,B,\iota,\lambda,\ov\eta_N) \ar[r, mapsto]  &  (A_0,\iota_0,\lambda_0,B,\iota,\lambda,\ov\eta),
   \end{tikzcd}
\end{gathered}
\end{equation}
where $\ov\eta$ is the $K_G$-orbit of isometries $\wh\RV(A_0,B) \isoarrow V \otimes_F \BA_{F,f}$ induced by extension of scalars from the lifts \eqref{lift isom} of the sections of $\ov\eta_N$.  (Note that, given any geometric point $\ov s \to S$, stability of the orbit $\ov\eta$ under the action of $\pi_1(S,\ov s)$ follows from finite \'etaleness of $\ov\eta_N$.)  The inverse of $\varphi$ can be explicitly described in a way similar to the proof of \cite[Prop.\ 4.4]{KR-U2}; see also \cite[\S4.12]{DelBour} or \cite[Prop.\ 1.4.3.4]{Lan}.  Let $S$ be a locally noetherian scheme over $\Spec E$, and let $(A_0,\iota_0,\lambda_0,A,\iota,\lambda,\ov\eta)$ be an $S$-valued point on $\CF_{K_{\wt G}}'(\wt G)$.  Since we assume that $L \otimes_{O_F} \wh O_F$ is $K_G$-stable, there exists a unique $\wh O_F$-submodule $T \subset \wh\RV(A)$ such that the submodule $\Hom_{\wh O_F}(\wh\RT(A_0), T) \subset \wh\RV(A_0,A)$ identifies with $L \otimes_{O_F} \wh O_F \subset V \otimes_F \BA_{F,f}$ under each $\eta \in \ov\eta$.  The submodule $T$ gives rise to an abelian scheme $B$ over $S$ with $O_F$-action $\iota_B$ and an $F$-linear quasi-isogeny $\mu\colon B \to A$ such that $\mu_*(\wh\RT(B)) = T$ inside $\wh\RV(A)$.  The pullback of $\lambda$ along $\mu$ defines the quasi-polarization $\lambda_B$, and the reduction mod $N$ of the isometries
\[
   \wh\RT(A_0,B) \xra[\undertilde]{\mu_*} \Hom_{O_F}\bigl(\wh\RT(A_0),T\bigr) \xra[\undertilde]{\eta} L \otimes_{O_F} \wh O_F,
\]
for $\eta \in \ov\eta$, defines the finite \'etale scheme $\ov\eta_N$ (using that $\ov\eta$ is $\pi_1(S,\ov s)$-stable with respect to any geometric point $\ov s \to S$).  Then $(A_0,\iota_0,\lambda_0,B,\iota_B,\lambda_B,\ov\eta_N)$ is the image of $(A_0,\iota_0,\lambda_0,A,\iota,\lambda,\ov\eta)$ under the inverse of $\varphi$.  The equivalence \eqref{comp isom} shows that $\CF_{K_{\wt G}}^{L,N}(\wt G)$ gives a third moduli interpretation of $M_{K_{\wt G}}(\wt G)$, and that, up to canonical equivalence, $\CF_{K_{\wt G}}^{L,N}(\wt G)$ is independent of the choice of $L$ and $N$ such that $K^{L,N} \subset K_G$.

\subsection{GGP and KR cycles}\label{ss:cycles}
To conclude Section \ref{s:rsz}, we give the definition of GGP and KR cycles in the context of the RSZ Shimura varieties.  In the case of the GGP cycles, let $n \geq 2$.  We return to the situation of Remark \ref{GGPsit}, with $u \in W$ a fixed totally positive vector and $W^\flat = (u)^\perp$.  Replacing $W$ by $W^\flat$ in the discussion in Section \ref{ss:tilde G}, we get a Shimura datum $(\wt H,\{h_{\wt H}\})$, where $\wt H$ denotes the analog of the group $\wt G$ for $W^\flat$.  The inclusion $W^\flat \subset W$ then induces a natural inclusion $\wt H \subset \wt G$.

Since $u$ is totally positive, the inclusion $\wt H \subset \wt G$ induces a morphism of the Shimura data we have defined.  Hence there is a morphism of towers
\begin{equation}\label{ggp morph}
   \Bigl(S\bigl(\wt H,\{h_{\wt H}\}\bigr)_{K_{\wt H}}\Bigr) \to \Bigl(S\bigl(\wt G,\{h_{\wt G}\}\bigr)_{K_{\wt G}}\Bigr),
\end{equation}
which lies over the tower $(S(Z^\BQ,\{h_{Z^\BQ}\})_{K_{Z^\BQ}})$.  Taking the graph morphism, we obtain a closed embedding of towers,
\begin{equation}\label{ggp cycle}
   \Bigl(S\bigl(\wt H,\{h_{\wt H}\}\bigr)_{K_{\wt H}}\Bigr) \inj \Bigl(S\bigl(\wt H,\{h_{\wt H}\}\bigr)_{K_{\wt H}}\Bigr) \times_{(S(Z^\BQ,\{h_{Z^\BQ}\})_{K_{Z^\BQ}})} \Bigl(S\bigl(\wt G,\{h_{\wt G}\}\bigr)_{K_{\wt G}}\Bigr).
\end{equation}
The resulting cycle in the target of \eqref{ggp cycle} is the \emph{GGP cycle} in the present context.

In terms of the moduli description of Theorem \ref{moduliforshim}, the above morphisms admit the following simple descriptions.  Using that $u$ is totally positive, it is easy to see that $E$ is also the reflex field of $(\wt H,\{h_{\wt H}\})$.  Take $W_0 := (u)$ in moduli problem of Definition \ref{def:isogtuple} for both groups, so that we consider the spaces $V = \Hom_F(W_0,W)$ and $V^\flat = \Hom_F(W_0,W^\flat)$. Similarly fix the same class $\tau$ in both moduli problems.  Then for level subgroups $K_{\wt G} = K_{Z^\BQ} \times K_G$ and $K_{\wt H} = K_{Z^\BQ} \times K_H$ of the form \eqref{K_wtG*}, and such that $K_H \subset K_G$, the morphism \eqref{ggp morph} is given by the finite and unramified morphism over $E$,
\begin{equation}\label{ggp morph moduli isog}
   \begin{tikzcd}[row sep=0ex]
      M_{K_{\wt H}}(\wt H) \ar[r]  &  M_{K_{\wt G}}(\wt G)\\
      ( A_0,\iota_0,\lambda_0,\ov\eta_0,A^\flat,\iota^\flat,\lambda^\flat,\ov\eta^\flat ) \ar[r, mapsto]  &  ( A_0,\iota_0,\lambda_0,\ov\eta_0,A^\flat \times  A_0,\iota^\flat \times \iota_0,\lambda^\flat \times \lambda_0,\ov{\eta}).
   \end{tikzcd}
\end{equation}
Here the level structure $\ov\eta$ is the $K_G$-orbit of the isometries
\begin{equation}\label{ov eta def}
\begin{split}
   \wh\RV(A_0,A^\flat \times A_0) &= \wh\RV(A_0,A^\flat) \oplus \End_{\BA_{F,f}}\bigl(\wh\RV(A_0)\bigr)   \xra[\undertilde]{\eta^\flat \oplus [\id \shortmapsto \id]}\\
   &\hspace{20ex}\bigl(V^\flat \otimes_F \BA_{F,f}\bigr) \oplus \bigl(\End_F(W_0) \otimes_F  \BA_{F,f}\bigr) = V \otimes_F \BA_{F,f}
\end{split}
\end{equation}
for $\eta^\flat \in \ov\eta^\flat$.  The GGP cycle for the given levels is the graph of \eqref{ggp morph moduli isog},
\begin{equation}\label{ggp cycle isog}
   M_{K_{\wt H}}(\wt H) \inj M_{K_{\wt H}}(\wt H) \times_{M_{0,K_{Z^\BQ}}^\tau} M_{K_{\wt G}}(\wt G).
\end{equation}

The morphisms \eqref{ggp morph moduli isog} and \eqref{ggp cycle isog} admit analogous descriptions in terms of the alternative moduli interpretations of Section \ref{ss:isom variant}.  In the case of the moduli problems $\CF_{K_{\wt H}}'(\wt H)$ and $\CF_{K_{\wt G}}'(\wt G)$ of Definition \ref{def:tuple}, let the level subgroups $K_{\wt G} = K_{Z^\BQ}^\circ \times K_G$ and $K_{\wt H} = K_{Z^\BQ}^\circ \times K_H$ be of the form \eqref{K_wtG}, and again assume that $K_H \subset K_G$.  Fix an ideal $\fka$ such that $M_0^\fka \neq \emptyset$, and again take $W_0 = (u)$ in both moduli problems. The proof of \cite[Prop.\ 3.1.6]{Ho-kr} shows that for an appropriate $\Phi$-adapted $\sqrt\Delta \in F$, there exists an $O_F$-lattice $\Lambda_0 \subset W_0$ such that $\Lambda_0^\vee = \fka\i \Lambda_0$ with respect to the form $\tr_{F/\BQ} \sqrt\Delta \i \sform|_{W_0}$; fix such a $\sqrt\Delta$ and $\Lambda_0$, and let $\xi$ denote the class of $\Lambda_0$ in $\CL_\Phi^\fka/{\sim}$.  Then the morphism \eqref{ggp morph moduli isog} is given by the morphism of moduli problems $\CF_{K_{\wt H}}'(\wt H) \to \CF_{K_{\wt G}}'(\wt G)$ sending
\begin{equation}\label{ggp morph moduli}
   ( A_0,\iota_0,\lambda_0,A^\flat,\iota^\flat,\lambda^\flat,\ov\eta^\flat ) \mapsto( A_0,\iota_0,\lambda_0,A^\flat \times  A_0,\iota^\flat \times \iota_0,\lambda^\flat \times \lambda_0,\ov{\eta}),
\end{equation}
where the level structure $\ov\eta$ is defined exactly as in \eqref{ov eta def}.  The GGP cycle is again the graph of this morphism, as in \eqref{ggp cycle isog} (with the fibered product over $M_0^{\fka,\xi}$).  The descriptions of the morphisms \eqref{ggp morph moduli isog} and \eqref{ggp cycle isog} in terms of the moduli problems $\CF_{K_{\wt H}}^{L,N}(\wt H)$ and $\CF_{K_{\wt G}}^{L,N}(\wt G)$ are completely analogous.

To define KR cycles, we again give differing versions according to our differing moduli interpretations of the RSZ Shimura varieties.  First consider the moduli problem $\CF_{K_{\wt G}}(\wt G)$, for any choice of defining data in Definition \ref{def:isogtuple}.  Let $(A_0,\iota_0,\lambda_0,\ov\eta_0,A,\iota,\lambda,\ov\eta)$ be an $S$-point on $\CF_{K_{\wt G}}(\wt G)$ for a connected, locally noetherian $E$-scheme $S$.  Then $\Hom_F^0(A_0,A)$ is well-defined, and by the multiplier condition in Definition \ref{def:isogtuple}, it carries a natural well-defined $F/F_0$-hermitian form $h'$, in analogy with \eqref{h_A def},
\begin{equation}\label{h'}
   h'(x,y) := \lambda_0\i \circ y^\vee \circ \lambda \circ x \in \End_{O_F}^0(A_0) \cong F, \quad x,y \in \Hom_F^0(A_0,A).
\end{equation}
Note that passing to Tate modules defines an isometric embedding $\Hom_F^0(A_0,A) \inj \wh\RV(A_0,A)$.

Now let $m$ be a positive integer, let $T \in \Herm_m(F)$ be an $m \times m$ hermitian matrix which is positive semidefinite at all archimedean places, and let $L$ be an $O_F$-lattice in the vector space $V$ of Definition \ref{def:isogtuple}.  The \emph{KR cycle} $Z(T,L)$ is the stack of tuples
\begin{equation}\label{KR tuple}
   (A_0,\iota_0,\lambda_0,\ov\eta_0,A,\iota,\lambda,\ov\eta;\mathbf{x}),
\end{equation}
where $(A_0,\iota_0,\lambda_0,\ov\eta_0,A,\iota,\lambda,\ov\eta)$ is an object in $\CF_{K_{\wt G}}(\wt G)$ and $\mathbf{x} = (x_1,\dotsc,x_m) \in \Hom_F^0(A_0,A)^m$ is an $m$-tuple of quasi-homomorphisms such that $(h'(x_i,x_j)) = T$, and such that for each $i=1,\dotsc,m$ and each $\eta\in \ov\eta$, the quasi-homomorphism $x_i$ identifies with an element of $L \otimes_{O_F} \wh O_F$ under the composite
\[
   \Hom_F^0(A_0,A) \inj \wh\RV(A_0,A) \xra[\undertilde]\eta V \otimes_F \BA_{F,f}.
\]
(Note that if $L \otimes_{O_F} \wh O_F$ is stable inside $V \otimes_F \BA_{F,f}$ under the action of the subgroup $K_G$ of \eqref{K_wtG*}, then this last condition is independent of $\eta \in \ov\eta$.)  A morphism $(A_0,\iota_0,\lambda_0,\ov\eta_0,A,\iota,\lambda,\ov\eta;\mathbf{x}) \to (A_0',\iota_0',\lambda_0',\ov\eta_0',A',\iota',\lambda',\ov\eta';\mathbf{x}')$ consists of quasi-isogenies $\mu_0\colon A_0 \to A_0'$ and $\mu\colon A \to A'$ as in Definition \ref{def:isogtuple} which pull $\mathbf{x}'$ back to $\mathbf{x}$.  The proof of \cite[Prop.\ 2.9]{KR-U2} transposes to the present setting to show that $Z(T,L)$ is representable by a DM stack which is finite and unramified over $\CF_{K_{\wt G}}^{L,N}(\wt G) \cong M_{K_{\wt G}}(\wt G)$ (in the present setting, one uses the lattice $L$ in the moduli problem to deduce finiteness). If $T$ is totally positive definite and $1 \leq m \leq \min\{r_\varphi \mid \varphi \in \Phi\}$, then $Z(T,L)$ has codimension $m\sum_{\varphi\in\Phi}r_{\ov\varphi}$ over $M_{K_{\wt G}}(\wt G)$. In particular, in the strict fake Drinfeld case relative to $\Phi$
of Example \ref{eg:special types}\eqref{strictFK}, the codimension is $m$.

In the case of the moduli problem $\CF_{K_{\wt G}}'(\wt G)$ (for any choice of defining data in Definition \ref{def:tuple}), the \emph{KR cycle $Z'(T,L)$}, for $T$ and $L$ as above, is the stack of tuples $(A_0,\iota_0,\lambda_0,A,\iota,\lambda,\ov\eta;\mathbf{x})$, where $(A_0,\iota_0,\lambda_0,A,\iota,\lambda,\ov\eta)$ is an object in $\CF_{K_{\wt G}}'(\wt G)$ and $\mathbf{x}$ is exactly as in \eqref{KR tuple}.  Then the equivalence $\CF_{K_{\wt G}}'(\wt G) \isoarrow \CF_{K_{\wt G}}(\wt G)$ of \eqref{compisom2} induces a natural equivalence $Z'(T,L) \isoarrow Z(T,L)$.

In the case of the moduli problem $\CF_{K_{\wt G}}^{L,N}(\wt G)$ (for any choice of defining data in Definition \ref{isom variant}), let $(A_0,\iota_0,\lambda_0,B,\iota,\lambda,\ov\eta_N)$ be an object in $\CF_{K_{\wt G}}^{L,N}(\wt G)$.  Then the group $\Hom_{O_F}(A_0,B)$ is well-defined (since both $A_0$ and $B$ are taken up to isomorphism as abelian schemes), and this group carries a natural $O_F/O_{F_0}$-hermitian form $h'$, defined as in \eqref{h'}, which takes values in some fractional ideal $\fkd$ of $F$.  In this case, passing to Tate modules defines an isometric embedding $\Hom_{O_F}(A_0,B) \inj \wh\RT(A_0,B)$. (Therefore, by the existence of a level structure in the moduli problem, if the hermitian form on $V$ is $O_F$-valued on $L$, then $h'$ will be $O_F$-valued too.)  The \emph{KR cycle $Z^{L,N}(T)$}, for $T$ as above, is the stack of tuples $(A_0,\iota_0,\lambda_0,B,\iota,\lambda,\ov\eta_N;\mathbf{x})$, where $(A_0,\iota_0,\lambda_0,B,\iota,\lambda,\ov\eta_N)$ is an object in $\CF_{K_{\wt G}}^{L,N}(\wt G)$ and $\mathbf{x} = (x_1,\dotsc,x_m) \in \Hom_{O_F}(A_0,B)^m$ is an $m$-tuple of homomorphisms such that $(h'(x_i,x_j)) = T$.  In this case, the equivalence $\CF_{K_{\wt G}}^{L,N}(\wt G) \isoarrow \CF_{K_{\wt G}}'(\wt G)$ of \eqref{comp isom} induces a natural equivalence $Z^{L,N}(T) \isoarrow Z'(T,L)$.

\section{$p$-integral models of RSZ Shimura varieties}\label{s:semiglob}
\subsection{Semi-global models}\label{ss:semiglob}
In this subsection we define a ``semi-global'' integral model of the stack $M_{K_{\wt G}}(\wt G)$ over $\Spec O_{E,(p)}$ for certain level subgroups $K_{\wt G}$.

We first need some preparatory notions.  Let $k$ be any algebraically closed field which is an $O_E$-algebra.  Let $(A_0,\iota_0,\lambda_0)$ and $(A,\iota,\lambda)$ be two triples, each consisting of an abelian variety over $k$, an $F$-action up to isogeny, and a quasi-polarization whose Rosati involution induces the nontrivial $F/F_0$-automorphism on $F$.  Suppose that $(A_0,\iota_0)$ satisfies the Kottwitz condition \eqref{kottcondA_0}, and that $(A,\iota)$ satisfies the Kottwitz condition \eqref{kottcond} relative to a fixed choice of a generalized CM type $r$ of rank $n$; in particular, this implies that $A_0$ and $A$ have respective dimensions $[F_0:\BQ]$ and $n\cdot [F_0 : \BQ]$. Let $v$ be a finite place of $F_0$ which does not split in $F$. Then \cite[App.\ A]{RSZ3} defines a sign invariant $\inv_v^r(A_0,\iota_0,\lambda_0, A,\iota,\lambda)\in \{\pm 1\}$.\footnote{Since we take the Kottwitz condition \eqref{kottcondA_0} for $A_0$ to be the opposite of the one used in loc.\ cit.,\ we need to use the version of $\inv_v^r$ modified as in \cite[Rem.\ A.2]{RSZ3} when the residue characteristic of $v$ equals $\charac k$.}  If the residue characteristic of $v$ does not equal $\charac k$, then $\inv_v^r$ is simply the Hasse invariant of the $F_v/F_{0,v}$-hermitian space
\begin{equation}\label{V_v(A_0,A)}
   \RV_v(A_0,A) := \Hom_{F_v}\bigl(\RV_v(A_0), \RV_v(A)\bigr),
\end{equation}
where the hermitian form is the obvious $v$-adic analog of \eqref{h_A def} (and hence $\RV_v(A_0,A)$ is the $v$-factor of \eqref{V(A_0,A)} when $\charac k = 0$). If  the residue characteristic of $v$ equals $\charac k$, then $\inv_v^r$ is defined similarly in terms of the highest exterior power of the Hom space of the rational Dieudonn\'e modules of $A_0$ and $A$, with a further correction factor in terms of the function $r$. The sign invariant depends only on the tuple $(A_0, \iota_0, \lambda_0; A, \iota, \lambda)$ up to isogeny, and it is locally constant in families over $O_E$-schemes \cite[Prop.\ A.1]{RSZ3}. 

We next note that the definition of the moduli space $M^\fka_0$ over $\Spec E$ of Section \ref{ss:isom variant} extends word-for-word to a moduli space $\CM^\fka_0$  over $\Spec O_E$. Then $\CM^\fka_0$ is a Deligne--Mumford stack, finite and \'etale over $\Spec O_E$, cf.\ \cite[Prop.\ 3.1.2]{Ho-kr}. It follows that the decomposition \eqref{decompM0} extends to a disjoint union decomposition of $\CM^\fka_0$,
\begin{equation}\label{intdecompshim}
  \CM_0^\fka=\coprod_{\xi \in \CL_\Phi^\fka/{\sim}}\CM_0^{\fka, \xi}.
\end{equation}

For the rest of this section we fix a prime number $p$. We denote by $\CV_p$ the set of places of $F_0$ over $p$. If $p = 2$, then we assume that every $v \in \CV_p$ is unramified in $F$.  We fix $\fka$, $\sqrt\Delta$, $\xi$, $\Lambda_0$, and $W_0$ as before \eqref{compisom}.  We continue with the $n$-dimensional hermitian space $W$, and as usual we set $V = \Hom_F(W_0,W)$.  For each $v \in \CV_p$, we endow the $F_v/F_{0,v}$-hermitian space $W_v := W \otimes_F F_v$ with the $\BQ_p$-valued alternating form $\tr_{F_v/\BQ_p} \sqrt\Delta\i\sform$, 
and we fix a \emph{vertex lattice} $\Lambda_v \subset W_v$ with respect to this form, i.e., $\Lambda_v$ is an $O_{F,v}$-lattice such that
\begin{equation*}%\label{Lambda_v}
   \Lambda_v\subset \Lambda_v^\vee \subset\pi_v\i\Lambda_v .
\end{equation*}
Here $\pi_v$ denotes a uniformizer in $F_v$ (if $v$ splits in $F$, this means the image in $F_v$ of a uniformizer for $F_{0,v}$), and $\Lambda_v^\vee \subset W_v$ denotes the dual lattice with respect to $\tr_{F_v/\BQ_p} \sqrt\Delta\i\sform$.\footnote{We remind the reader that $\Lambda_v^\vee$ and the dual lattice $\Lambda_v^*$ with respect to the hermitian form on $W_v$ need not be equal, but they are at least scalar multiples of each other. The lattice $\Lambda_v^\vee$ is more natural to use in connection with polarizations. We also point out that vertex lattices in \cite{RSZ3} are always taken with respect to \emph{hermitian} forms.}

We consider a subgroup $K_{\wt G} = K^\circ_{Z^\BQ} \times K_G$ as in \eqref{K_wtG}. We assume that $K_G \subset G(\BA_f)$ is of the form $K_G = K_G^p \times K_{G, p}$, where $K_G^p\subset G(\BA_f^p)$ is arbitrary and where
\[
   K_{G, p}=\prod_{v\in \CV_p}K_{G, v}\subset G(\BQ_p) = \prod_{v \in \CV_p} \U(W)(F_{0,v}),
\]
with 
\begin{equation}\label{K_G,p decomp}
   K_{G, v} :=  {\rm Stab}_{\U(W)(F_{0, v})}(\Lambda_v) . 
\end{equation}
We note that if $v$ is unramified in $F$, then $K_{G,v}$ is a maximal parahoric subgroup of $\U(W)(F_{0, v})$.  If $v$ ramifies in $F$ (recall that in this case we assume that $v \nmid 2$), then $K_{G,v}$ is a maximal compact subgroup of $\U(W)(F_{0, v})$ which contains a (maximal) parahoric subgroup with index $2$, unless $n$ is even and $\Lambda_v$ is $\pi_v$-modular, in which case $K_{G,v}$ is itself maximal parahoric; see \cite[\S4.a]{PR-TLG}.  Here \emph{$\pi_v$-modular} means that $\Lambda_v^\vee = \pi_v\i\Lambda_v$.

We now define the following category fibered in groupoids $\CF^{\naive}_{K_{\wt G}}(\wt G)$  over $\LNSch_{/O_{E, (p)}}$. As before, to lighten  notation, we suppress the ideal $\fka$ and the element $\xi$.

\begin{definition}\label{def:inttuple}
The category functor $\CF^\naive_{K_{\wt G}}(\wt G)$ associates to each scheme $S$ in $\LNSch_{/O_{E,(p)}}$ the groupoid of tuples 
$(A_0,\iota_0,\lambda_0,A,\iota,\lambda,\ov\eta^p)$, where
\begin{altitemize}
\item $(A_0,\iota_0,\lambda_0)$ is an object of $\CM^{\fka, \xi}_0(S)$;
\item $A$ is an abelian scheme over $S$;
\item $\iota\colon O_{F,(p)} \to \End_{(p)}(A)$ is an action up to prime-to-$p$ isogeny satisfying the Kottwitz condition \eqref{kottcond} on $O_{F,(p)}$;
\item $\lambda \in \Hom_{(p)}(A,A^\vee)$ is a quasi-polarization on $A$ whose Rosati involution satisfies condition \eqref{Ros} on $O_{F,(p)}$; and
\item $\ov\eta^p$ is a $K_G^p$-orbit of isometries of $\BA_{F,f}^p/\BA_{F_0,f}^p$-hermitian modules
\begin{equation}\label{levelprimetop}
   \eta^p\colon \wh \RV^p(A_0,A) \isoarrow V \otimes_F \BA_{F,f}^p,
\end{equation}
where
\begin{equation}
   \wh \RV^p(A_0,A) := \Hom_{\BA_{F,f}^p}\bigl(\wh \RV^p(A_0), \wh \RV^p(A)\bigr),
\end{equation}
and where the hermitian form on $\wh \RV^p(A_0,A)$ is the obvious prime-to-$p$ analog of \eqref{h_A def}.
\end{altitemize}
We impose the following further conditions on the above tuples.
\begin{altenumerate}
\item\label{sglob cond i} Consider the decomposition of $p$-divisible groups
\begin{equation}\label{decofpdivgp}
   A[p^\infty] =  \prod_{v \in \CV_p} A[v^\infty]
\end{equation}
induced by the action of $O_{F_0}\otimes\BZ_p \cong \prod_{v \in \CV_p} O_{F_0,v}$.
Since $\Ros_\lambda$ is trivial on $O_{F_0}$, $\lambda$ induces a polarization $\lambda_v \colon A[v^\infty] \to A^\vee[v^\infty] \cong A[v^\infty]^\vee$ of $p$-divisible groups for each $v$. The condition we impose is that $\ker\lambda_v$ is contained in $A[\iota(\pi_v)]$ of rank $\#(\Lambda_v^\vee/ \Lambda_v)$ for each $v \in \CV_p$.
\item\label{sglob cond ii} We require that at every geometric point $\ov s$ of $S$ the following \emph{sign condition} holds for every non-split place $v \in \CV_p$,
\begin{equation}\label{condsign}
   \inv^r_v(A_{0, \ov s},\iota_{0, \ov s},\lambda_{0, \ov s},A_{\ov s},\iota_{\ov s},\lambda_{\ov s})=\inv_v(V),
\end{equation}
where the right-hand denotes the Hasse invariant of the hermitian space $V$ at $v$.
\end{altenumerate}
A morphism  $(A_0, \iota_0, \lambda_0, A, \iota, \lambda, \ov\eta^p) \to (A_0', \iota_0', \lambda_0', A', \iota', \lambda', \ov\eta'^p)$ in this groupoid is given by an isomorphism $\mu_0\colon (A_0,\iota_0,\lambda_0) \isoarrow (A_0',\iota_0',\lambda_0')$ in $\CM^{\fka,\xi}_0(S)$ and an $O_{F,(p)}$-linear quasi-isogeny $\mu\colon A \to A'$ inducing an isomorphism $A[p^\infty] \isoarrow A'[p^\infty]$, pulling $\lambda'$ back to $\lambda$, and pulling $\ov\eta^{p\prime}$ back to $\ov\eta^p$.
\end{definition}

\begin{remark}\label{excone}
Let $\mathbf{A} = (A_{0, \ov s},\iota_{0, \ov s},\lambda_{0, \ov s},A_{\ov s},\iota_{\ov s},\lambda_{\ov s},  \ov\eta^p)$ be a tuple as in Definition \ref{def:inttuple}, except where we don't impose the sign condition in \eqref{sglob cond ii}, and suppose that $\mathbf{A}$ lifts to  characteristic zero.  Then, by the product formula and the Hasse principle for hermitian vector spaces, by the Kottwitz condition and the existence of the prime-to-$p$ level structure $\ov\eta^p$, and by local constancy of $\inv_v^r$, the sign condition for all non-split $v \in \CV_p$ except one implies  the sign condition for $\mathbf{A}$ at the remaining place. In particular, when $F_0=\BQ$, the sign condition  is empty, provided that $\CF^\naive_{K_{\wt G}}(\wt G)$ is topologically flat over $O_{E, (p)}$.   We refer to Remark \ref{redsign}\eqref{inert sign} for other instances when the sign condition is redundant. In these instances  the \emph{naive} moduli problem considered here is replaced  by more sophisticated integral models  which are flat over $O_{E, (p)}$. 
\end{remark}

\begin{remark}
As in Remarks \ref{V rem} and \ref{V rem 2}, when $(W_0,\sform_0) = (F,\Nm_{F/F_0})$, we may replace $V$ by $W$ everywhere in Definition \ref{def:inttuple}.
\end{remark}

The following theorem shows that the moduli functor $\CF^\naive_{K_{\wt G}}(\wt G)$ defines an extension of $M_{K_{\wt G}}(\wt G)$ with reasonable properties. 
\begin{theorem}\label{intmoduliforshim}
The moduli problem $\CF^\naive_{K_{\wt G}}(\wt G)$ is representable by a Deligne--Mumford stack $\CM^\naive_{K_{\wt G}}(\wt G)$ over $\Spec O_{E,(p)}$,  and   
$$
   \CM^\naive_{K_{\wt G}}(\wt G)\times_{\Spec O_{E, (p)}}\Spec E \cong M_{K_{\wt G}}(\wt G)  .
$$
Furthermore,
\begin{altenumerate}
\item\label{intmoduliforshim i} If $M_{K_{\wt G}}(\wt G)$ is proper over $\Spec E$, then $\CM^\naive_{K_{\wt G}}(\wt G)$ is proper over $\Spec O_{E, (p)}$. 
\item\label{intmoduliforshim ii} If $p$ is unramified in $F$ and the vertex lattice $\Lambda_v$ is self-dual for all $v\in \CV_p$, then 
$\CM^\naive_{K_{\wt G}}(\wt G)$ is smooth over $\Spec O_{E, (p)}$. 
\end{altenumerate}
\end{theorem}

\begin{proof} 
This is the extension of \cite[Th.~4.1]{RSZ3} to the case of arbitrary signature type.  The key point is the  statement which compares $M_{K_{\wt G}}(\wt G)$ with the generic fiber of $\CM^\naive_{K_{\wt G}}(\wt G)$.  Using the moduli interpretation $\CF_{K_{\wt G}}'(\wt G)$ of $M_{K_{\wt G}}(\wt G)$ of Definition \ref{def:tuple}, this comes down to completing the prime-to-$p$ level structure $\ov\eta^p$ as in \eqref{levelprimetop} to a level structure $\ov\eta$ as in \eqref{eta}, for a point $(A_0,\iota_0,\lambda_0,A,\iota,\lambda,\ov\eta^p)$ of $\CM^\naive_{K_{\wt G}}(\wt G)$ over an $E$-scheme. Indeed, by the sign condition  \eqref{condsign}, there exists an isomorphism between $\RV_v(A_0,A) = \wh \RV(A_0,A)_v$ and $V_v$ for all $v\in\CV_p$.  Now, recall that prior to defining $\CF^\naive_{K_{\wt G}}(\wt G)$, we fixed a lattice $\Lambda_0 \subset W_0$ such that $\Lambda_0^\vee = \fka\i\Lambda_0$. The $K_{G, v}$-equivalence class $\ov\eta_v$ of the isomorphism $\RV_v(A_0,A) \simeq V_v$ is then singled out by stipulating that it takes the lattice $\Hom_{O_{F,v}}(\RT_v(A_0), \RT_v(A))$ in $\RV_v(A_0,A)$ to the lattice $L_v := \Hom_{O_{F,v}}(\Lambda_{0,v},\Lambda_v)$ in $V_v$, where $\Lambda_{0,v} := \Lambda \otimes_{O_F} O_{F,v}$; note that $K_{G,v}$ is the stabilizer of $L_v$ with respect to the isomorphism $\U(W)(F_{0,v}) \cong \U(V)(F_{0,v})$.  We remark that in loc.\ cit.,\ this part of the argument is carried out when $p = 2$ under the assumption that all $v \in \CV_2$ are split in $F$.  The argument extends to the case that all $v \in \CV_2$ are unramified in $F$ by \cite[Th.\ 7.1]{Jac}, which says that all vertex lattices of the same type are conjugate under the unitary group in any hermitian space attached to an unramified extension of local fields of characteristic not $2$.  We also remark that the smoothness assertion in \eqref{intmoduliforshim ii} follows as in the proof of \cite[Th.~4.1]{RSZ3}, using Theorem \ref{normal forms} in the appendix below to extend the formalism of local models to the case $p = 2$ under the assumption that all $v \in \CV_2$ are unramified in $F$.

The properness assertion in \eqref{intmoduliforshim i} follows as in  \cite[end of \S 5]{K-points}.
\end{proof}

\begin{remark}\label{r:deeper level}
The lattice stabilizer groups $K_{G,p}$ appearing as $p$-factors of the level subgroups in Theorem \ref{intmoduliforshim} are all maximal. It is possible to extend the definitions above to general \emph{lattice multichain} stabilizer groups at $p$ (the $v$-factors of which, for each $v \in \CV_p$, continue to contain a parahoric subgroup with index $1$ or $2$, as after \eqref{K_G,p decomp}), by replacing the entry $A$ in $(A_0,\iota_0,\lambda_0,A,\iota,\lambda,\ov\eta^p)$ by a \emph{multichain} of abelian varieties, cf.\ \cite[Def.\ 6.9]{RZ1} (in the context of general PEL moduli problems) or \cite[\S 1.4]{PR3} (in the context of unitary moduli problems with $F_0 = \BQ$).
\end{remark}

\subsection{Semi-global GGP and KR cycles}
In this section we give semi-global versions of the GGP and KR cycles of Section \ref{ss:cycles}.

Let us again start with the GGP cycles.  As before, we take $n \geq 2$, we fix a vector $u \in W$ of totally positive norm, we set $W^\flat = (u)^\perp$, and we consider the resulting Shimura datum $(\wt H, \{h_{\wt H}\})$ for $W^\flat$.  We may then define a semi-global integral model $\CM_{K_{\wt H}}^\naive (\wt H)$ over $\Spec O_{E,(p)}$ as in Section \ref{ss:semiglob}.  The definition of this stack depends on the choice of a vertex lattice $\Lambda_v^\flat \subset W_v^\flat$ for each $v \in \CV_p$; to define a semi-global version $\CM_{K_{\wt H}}^\naive (\wt H) \to \CM_{K_{\wt G}}^\naive (\wt G)$ of the morphism \eqref{ggp morph moduli}, these lattices and the lattices $\Lambda_v \subset W_v$ in the definition of $\CM_{K_{\wt G}}^\naive (\wt G)$ need to be suitably related. As in Section \ref{ss:cycles}, we set $W_0 = (u)$, $V = \Hom_F(W_0,W)$, and $V^\flat = \Hom_F(W_0,W^\flat)$.  By Remark \ref{remMa}\eqref{rem exa satisfied}\eqref{Ma unram case}, the stack $\CM_0^\fka$ is non-empty for $\fka$ equal to $O_{F_0}$ or to an inert prime ideal; we fix such an $\fka$, and we further fix $\sqrt\Delta$, $\Lambda_0$, and $\xi$ as before \eqref{ggp morph moduli}. Then the localization $\Lambda_{0,v} = \Lambda_0 \otimes_{O_F} O_{F,v}$ is a vertex lattice in $W_{0,v}$ for every place $v \in \CV_p$. For simplicity, let us now assume that $\Lambda_v$ and $\Lambda_v^\flat$ satisfy the relation
\begin{equation}\label{latt relation}
   \Lambda_v = \Lambda_v^\flat \oplus \Lambda_{0,v} \subset W_v = W_v^\flat \oplus W_{0,v}
\end{equation}
for each $v \in \CV_p$.  We further assume that the prime-to-$p$ level subgroups satisfy $K_H^p \subset K_G^p$.  Then the morphism \eqref{ggp morph moduli} extends to a morphism of $p$-integral models,
\begin{equation}\label{semiglob morph}
   \begin{tikzcd}[row sep=0ex]
      \CM_{K_{\wt H}}^\naive(\wt H) \ar[r]  &  \CM_{K_{\wt G}}^\naive(\wt G)\\
      \bigl( A_0,\iota_0,\lambda_0,A^\flat,\iota^\flat,\lambda^\flat,\ov\eta^{p,\flat} \bigr) \ar[r, mapsto]  &  \bigl( A_0,\iota_0,\lambda_0,A^\flat \times  A_0,\iota^\flat \times \iota_0,\lambda^\flat \times \lambda_0,\ov\eta^p\bigr).
   \end{tikzcd}
\end{equation}
Here the $K_G^p$-orbit $\ov\eta^p$ is defined in terms of the obvious prime-to-$p$ analog of \eqref{ov eta def}, and it is easy to see that the sign condition on $\CM_{K_{\wt H}}^\naive(\wt H)$ implies that the formula in \eqref{semiglob morph} indeed produces points satisfying the sign condition on $\CM_{K_{\wt G}}^\naive(\wt G)$.  The \emph{$p$-integral GGP cycle} for the given levels is the graph of \eqref{semiglob morph},
\begin{equation}
   \CM_{K_{\wt H}}^\naive(\wt H) \inj \CM_{K_{\wt H}}^\naive(\wt H) \times_{\CM_0^{\fka,\xi}} \CM_{K_{\wt G}}^\naive(\wt G).
\end{equation}

\begin{remark}
Our assumption that the lattices in question satisfy the relation \eqref{latt relation} for all $v \in \CV_p$ is, in certain cases, a serious one.  For example, if $n$ is even and there is a place $v \in \CV_p$ which ramifies in $F$, then it is impossible to choose $\Lambda_v$ and $\Lambda_v^\flat$ in this way such that $\CM_{K_{\wt G}}(\wt G)$ has good reduction (at least outside of zero dimensional cases).  We refer to \cite[\S4.4]{RSZ3} for more general definitions of GGP cycles in this and further such contexts.
\end{remark}

Now let us define the semi-global KR cycles.  In fact, we will give two versions of the definition.  The first is based directly on the moduli problem $\CF_{K_{\wt G}}^\naive(\wt G)$, for any choice of defining data in Definition \ref{def:inttuple}.  Fix a global $O_F$-lattice $\Lambda \subset W$ whose localization $\Lambda \otimes_{O_F} O_{F,v}$, for each $v \in \CV_p$, equals the vertex lattice $\Lambda_v \subset W_v$ we fixed before Definition \ref{def:inttuple}.  Set $L := \Hom_{O_F}(\Lambda_0,\Lambda) \subset V$.  Let $m$ be a positive integer, and let $T \in \Herm_m(F)$ be a hermitian matrix which is positive semidefinite at all archimedean places. Then the \emph{$p$-integral KR cycle $\CZ^\naive(T,L)$} is the stack of tuples $(A_0,\iota_0,\lambda_0,A,\iota,\lambda,\ov\eta^p;\mathbf{x})$, where $(A_0,\iota_0,\lambda_0,A,\iota,\lambda,\ov\eta^p)$ is an object in $\CF_{K_{\wt G}}^\naive(\wt G)$ and $\mathbf{x} = (x_1,\dotsc,x_m) \in \Hom_{(p),O_{F,(p)}}(A_0,A)^m$ is an $m$-tuple of $O_{F,(p)}$-linear quasi-homomorphisms such that $(h'(x_i,x_j)) = T$ and such that each $x_i$ identifies with an element of $L \otimes_{O_F} \wh O_F^p$ under the composite
\[
   \Hom_{(p),O_{F,(p)}}(A_0,A) \inj \wh\RV^p(A_0,A) \xra[\undertilde]{\eta^p} V \otimes_F \BA_{F,f}^p,
\]
for each $\eta^p \in \ov\eta^p$.  Here the hermitian form $h'$ on $\Hom_{(p),O_{F,(p)}}(A_0,A)$ is defined as in \eqref{h'}, and $\wh O_F^p := (O_F[\frac 1 p])\sphat\phantom{|} \subset \BA_{F,f}^p$.  The proof of Theorem \ref{intmoduliforshim} extends to show that the generic fiber of $\CZ^\naive(T,L)$ is canonically isomorphic to the KR cycle $Z'(T,L)$ defined in Section \ref{ss:cycles}.

To give the second version of the $p$-integral KR cycle, we first need to introduce a $p$-integral version of Definition \ref{isom variant}.  Keep all the notation of the previous paragraph, and assume that $L \otimes_{O_F} \wh O_F^p$ is $K_G^p$-stable inside $V \otimes_F \BA_{F,f}^p$ (which is equivalent to $\Lambda \otimes_{O_F} \wh O_F^p$ being $K_G^p$-stable inside $W \otimes_F \BA_{F,f}^p$). For $N$ a positive integer prime to $p$, define $K^{p,L,N} \subset G(\BA_{F,f}^p)$ as the obvious prime-to-$p$ analog of $K^{L,N}$ in \eqref{K^L,N}.  Then $K^{L,N} = K^{p,L,N} \times K_{G,p}$. Choose $N$ such that $K^{p,L,N} \subset K_G^p$.  We define the following moduli problem.

\begin{definition}
The category functor  $\CF_{K_{\wt G}}^{\naive,L,N}(\wt G)$ associates to each scheme $S$ in $\LNSch_{/O_{E,(p)}}$ the groupoid of tuples
$(A_0,\iota_0,\lambda_0,B,\iota,\lambda,\ov\eta_N^p)$, where
\begin{altitemize}
\item $(A_0,\iota_0,\lambda_0)$ is an object of $\CM^{\fka, \xi}_0(S)$;
\item $B$ is an abelian scheme over $S$;
\item $\iota\colon O_F \to \End(B)$ is an action of $O_F$ on $B$ satisfying the Kottwitz condition \eqref{kottcond} on $O_F$;
\item $\lambda \in \Hom_{(p)}(B,B^\vee)$ is a quasi-polarization on $B$ whose Rosati involution satisfies condition \eqref{Ros} on $O_F$; and
\item $\ov\eta_N^p$ is a closed \'etale subscheme
\[
   \ov\eta_N^p \subset \uIsom_{O_F} \bigl(\uHom_{O_F}(A_0[N],B[N]) , (L/N L)_S\bigr)
\]
over $S$ such that for every geometric point $\ov s \to S$ (or equivalently, for a single geometric point on each connected component of $S$), the fiber $\ov\eta_N^p(\ov s)$ identifies with a $K_G^p/K^{p,L,N}$-orbit of isomorphisms
\[
   \eta_N^p(\ov s) \colon \Hom_{O_F}\bigl(A_0[N](\ov s),B[N](\ov s)\bigr) \isoarrow L/N L
\]
which lift to $\wh O_F^p$-linear isometries of hermitian modules
\[
   \wh\RT^p(A_0, B)(\ov s) \isoarrow L\otimes_{O_F}\wh O_F^p.
\]
Here
\begin{equation}
   \wh\RT^p(A_0,B) := \Hom_{\wh O_F^p}\bigl(\wh\RT^p(A_0), \wh\RT^p(B)\bigr)
\end{equation}
is the obvious prime-to-$p$ analog of \eqref{T(A0,B)}.
\end{altitemize}
We require that the tuples $(A_0,\iota_0,\lambda_0,B,\iota,\lambda,\ov\eta_N^p)$ satisfy conditions \eqref{sglob cond i} and \eqref{sglob cond ii} from Definition \ref{def:inttuple} (with $B$ in place of $A$). A morphism $(A_0,\iota_0,\lambda_0,B,\iota,\lambda,\ov\eta_N^p) \to (A'_0,\iota'_0,\lambda'_0,B',\iota',\lambda',\ov\eta_N^{p\prime})$ in this groupoid is given by an isomorphism $\mu_0\colon (A_0, \iota_0, \lambda_0) \isoarrow (A'_0, \iota'_0, \lambda'_0)$ in $\CM_0^{\fka,\xi}(S)$ and an $O_F$-linear isomorphism of abelian schemes $\mu\colon B\isoarrow B'$ pulling $\lambda'$ back to $\lambda$ and $\ov\eta_N^{p\prime}$ back to $\ov\eta_N^p$. 
\end{definition}

The obvious prime-to-$p$ analog of the morphism \eqref{comp isom} defines a natural equivalence of moduli problems
\[
   \CF_{K_{\wt G}}^{\naive,L,N}(\wt G) \isoarrow \CF_{K_{\wt G}}^\naive(\wt G).
\]
Hence $\CF_{K_{\wt G}}^{\naive,L,N}(\wt G)$ gives a second moduli interpretation of the stack $\CM_{K_{\wt G}}^\naive (\wt G)$.

In terms of the moduli functor $\CF_{K_{\wt G}}^{\naive,L,N}(\wt G)$, we now define the \emph{$p$-integral KR cycle $\CZ^{L,N}(T)$} word-for-word as in the case of $Z^{L,N}(T)$ in Section \ref{ss:cycles}, simply replacing $\CF_{K_{\wt G}}^{L,N}(\wt G)$ everywhere by $\CF_{K_{\wt G}}^{\naive,L,N}(\wt G)$.  Then $\CZ^{L,N}(T)$ is canonically equivalent to $\CZ(T,L)$, and its generic fiber canonically identifies with $Z^{L,N}(T)$.

\subsection{Summary table}\label{sumtable}
The following table summarizes some properties of the various unitary Shimura varieties we have introduced above.  In the last column, by ``cycle property'' we mean whether there exists  a KR cycle in the Shimura variety and a GGP cycle in an appropriate product of the Shimura varieties, in analogy with the discussion in Section \ref{ss:cycles}. In the last row, by ``BHKRY'' we mean the special case of RSZ Shimura varieties where $F_0=\BQ$, and of (strict fake) Drinfeld type, cf.\ \cite{BHKRY}. (In loc.~cit.\  only the case of principal polarization is considered.) In this case,  $F=K_0$ is an imaginary quadratic field. The term ``no sign necessary'' refers to the case when $\CF_{K_{\wt G}}^\naive(\wt G)$ is topologically flat over $O_{K_0, (p)}$, cf.\ Remark \ref{excone}. 

\begin{footnotesize}
\begin{center}
	\begin{tabular}{|c|c|c|c|c|c|}
		\hline
		Name  &  Shimura datum  &  Reflex field  &  Moduli problem  &  
         \begin{varwidth}{\linewidth}
            \centering
            $p$-integral\smash[b]{\strut}\\
            moduli problem\smash[t]{\strut}
         \end{varwidth}  
         &  
         \begin{varwidth}{\linewidth}
            \centering
            Cycle\smash[b]{\strut}\\
            property\smash[t]{\strut}
         \end{varwidth}   \\
		\hline
		D/K (\S\ref{ss:GU})  &  $(G^\BQ, \{h_{G^\BQ}\})$  &  $E_r$  &  
         \begin{varwidth}{\linewidth}
            \centering
   	      \emph{yes} for $n$ even,\smash[b]{\strut}\\ 
	         \emph{almost} for $n$ odd\smash[t]{\strut}
         \end{varwidth}
           &  
         \begin{varwidth}{\linewidth}
            \centering
            \emph{yes} if $p$ totally\\
            unramified
         \end{varwidth}
           &  \emph{no}\\
		\hline
		GGP  (\S\ref{ss:U})  &  $(G, \{h_G\})\vphantom{\Bigm|}$  &  $E_{r^\natural}$  &  \emph{no}  &  \emph{no}  &  \emph{yes}\\
		\hline 
		RSZ  (\S\ref{ss:tilde G})  &  $(\wt G, \{h_{\wt G}\})\vphantom{\Bigm|}$  &  $E = E_\Phi E_r = E_\Phi E_{r^\natural}$  &  \emph{yes}  &
         \begin{varwidth}{\linewidth}
            \centering
            \emph{yes} for $K_{G,p}$ a lattice\smash[b]{\strut}\\
            multichain stabilizer\smash[t]{\strut}
         \end{varwidth}
           &  \emph{yes}\\
		\hline
      \begin{varwidth}{\linewidth}
         \centering
         HT\\
         ( Ex.\ \ref{eg:special types}\eqref{eg:HT})
      \end{varwidth}
         &  $(G^\BQ, \{h^\HT_{G^\BQ}\})$\strut  &  
         \begin{varwidth}{\linewidth}
            \centering
            Usually $F=K_0F_0$;\\
            see Ex.\ \ref{eg:special types}\eqref{eg:HT}
         \end{varwidth}
           &  as in D/K  &
         \begin{varwidth}{\linewidth}
            \centering
   	      as in D/K, all\smash[b]{\strut}\\ 
	         levels if $p$ split in $K_0$\smash[t]{\strut}
         \end{varwidth}
           &  as in D/K\\
		\hline
		BHKRY  &  as in RSZ  &  $K_0$   &  as in RSZ  &
         \begin{varwidth}{\linewidth}
            \centering
   	      as in RSZ,\smash[b]{\strut}\\ 
	         but no sign necessary\smash[t]{\strut}
         \end{varwidth}
           &  as in RSZ\\
		\hline
	\end{tabular} 
\end{center}
\end{footnotesize}

\section{Flat and smooth $p$-integral models of RSZ Shimura varieties}\label{s:flatsm}
The $p$-integral model $\CM^\naive_{K_{\wt G}}(\wt G)$ of $M_{K_{\wt G}}(\wt G)$ defined in Section \ref{s:semiglob} is  not always flat over $\Spec O_{E, (p)}$. In this section, we first give some cases where it is known to be flat. We then give some cases where, upon imposing further conditions on the Lie algebra of the abelian variety $(A, \iota, \lambda)$ in the moduli problem, we obtain a closed substack $\CM_{K_{\wt G}}(\wt G)$ of $\CM^\naive_{K_{\wt G}}(\wt G)$ which is flat, with the same generic fiber. Finally, we give some cases, beyond the totally unramified case appearing in Theorem \ref{intmoduliforshim}\eqref{intmoduliforshim ii}, where $\CM_{K_{\wt G}}(\wt G)$ is even regular or smooth. We continue to assume that every $v \in \CV_p$ is unramified in $F$ if $p = 2$, and we again take $K_{\wt G}$ as in Section \ref{ss:semiglob}, i.e.,\ of the form $K^\circ_{Z^\BQ} \times K_G^p \times K_{G,p}$, with $K_G^p$ arbitrary and $K_{G,p}$ a product of stabilizers of vertex lattices $\Lambda_v$ for $v \in \CV_p$.

Let us note at the outset that when using terminology relating a lattice to its dual in this section, we will always mean the dual with respect to the $\BQ_p$-valued form, e.g., $\Lambda_v$ being self-dual means that $\Lambda_v = \Lambda_v^\vee$.  Strictly speaking, this usage differs from all the papers on unitary local models we will refer to (where the dual is with respect to the hermitian form).  But since $\Lambda_v^\vee$ and $\Lambda_v^*$ are scalar multiples of each other, the periodic lattice chains generated by $\{\Lambda_v, \Lambda_v^\vee\}$ and $\{\Lambda_v,\Lambda_v^*\}$ are the same, and there is ultimately no essential difference.

\subsection{Flatness of $\CM^\naive_{K_{\wt G}}(\wt G)$}\label{ss:naive flat}
The following result gives some cases where $\CM^\naive_{K_{\wt G}}(\wt G)$ is known to be flat.

\begin{theorem}\label{naive flat thm}
Suppose that $p$ is unramified in $F$ (without any condition on the vertex lattices $\Lambda_v$), or that the following three conditions hold:
\begin{altenumerate}
\renewcommand{\theenumi}{\emph{\arabic{enumi}}}
\item\label{inert split cond} each place $v \in \CV_p$ which is unramified in $F$ has ramification index $e \leq 2$ over $p$;\footnote{ This hypothesis can now be removed; see Remark \ref{naive flatness rem}\eqref{naive flatness rem ramification} and footnote \ref{mwy} below.}
\item each place $v \in \CV_p$ which ramifies in $F$ is unramified over $p$, and the lattice $\Lambda_v$ for such $v$ is self-dual; 
\item\label{r_varphi cond} the integers $r_\varphi$ for varying $\varphi \in \Hom_\BQ(F,\ov\BQ)$ differ by at most one.
\end{altenumerate} 
Then $\CM^\naive_{K_{\wt G}}(\wt G)$ is flat over $\Spec O_{E,(p)}$.
\end{theorem}

In fact, Theorem \ref{naive flat thm} is a consequence of the following more precise statement, which however requires some more notation to set up. Let $\nu$ be a place of $E$ over $p$, and choose an embedding $\alpha \colon \ov\BQ \to \ov \BQ_p$ inducing $\nu$.  Then $\alpha$ induces an identification
\begin{equation}\label{alpha}
\begin{gathered}
   \alpha_* \colon
   \begin{tikzcd}[baseline=(x.base),row sep=0ex]
      |[alias=x]| \Hom(F,\ov\BQ) \ar[r, "\sim"]  &  \Hom(F,\ov\BQ_p)\\
      \varphi \ar[r, mapsto]  &  \alpha \circ \varphi.
   \end{tikzcd}
\end{gathered}
\end{equation}
For each $p$-adic place $w$ of $F$, let
\begin{equation}
   \Hom_w(F,\ov\BQ) := \bigl\{\, \varphi \in \Hom(F,\ov\BQ) \bigm| \text{$\alpha \circ \varphi$ induces $w$} \,\bigr\}.
\end{equation}
Then, under the identification $\alpha_*$, the sets $\Hom_w(F,\ov\BQ)$ are the $\Gal(\ov\BQ_p/\BQ_p)$-orbits in $\Hom(F,\ov\BQ)$, and hence are independent of the choice of $\alpha$ inducing $\nu$.  For each $w$, let $F_w^t$ denote the maximal unramified extension of $\BQ_p$ in $F_w$.  For each $\psi \in \Hom_{\BQ_p}(F_w^t,\ov\BQ_p)$, let $\Hom_{w,\psi}(F,\ov\BQ) \subset \Hom_w(F,\ov\BQ)$ denote the fiber over $\psi$ of the composite
\begin{equation}\label{nogoodname}
   \Hom_w(F,\ov\BQ) \xra[\undertilde]{\alpha_*} \Hom_{\BQ_p}(F_w,\ov\BQ_p) \xra{\text{restrict}} \Hom_{\BQ_p}(F_w^t,\ov\BQ_p).  
\end{equation}
Then, under the identification $\alpha_*$, the sets $\Hom_{w,\psi}(F,\ov\BQ)$ are the $I_p$-orbits in $\Hom(F,\ov\BQ)$, where $I_p \subset \Gal(\ov\BQ_p/\BQ_p)$ denotes the inertia subgroup.  The label $\psi$ in $\Hom_{w,\psi}(F,\ov\BQ)$ therefore generally depends on the choice of $\alpha$ inducing $\nu$, but the partition
\begin{equation}\label{Hom(F,ovBQ) decomp}
   \Hom(F,\ov\BQ) = \coprod_{\substack{w \mid p\\ \psi\colon\! F_w^t \rightarrow \ov\BQ_p}} \Hom_{w,\psi}(F,\ov\BQ)
\end{equation}
depends only on $\nu$ up to labeling.  We now have the following result on flatness of the base change $\CM_{K_{\wt G}}^\naive(\wt G)_{O_{E,(\nu)}}$.

\begin{theorem}\label{naive flat O_E,nu}
Suppose that the following three conditions hold.
\begin{altenumerate}
\renewcommand{\theenumi}{\emph{\arabic{enumi}}}
\item\label{unram hypoth} For each $v \in \CV_p$ which is unramified in $F$, the ramification index $e$ of $v$ over $p$ satisfies $e \leq 2$ or, for each of the one or two places $w$ of $F$ over $v$ and each $\psi \in \Hom_{\BQ_p}(F_w^t,\ov\BQ_p)$,\footnote{ This hypothesis can now be removed; see Remark \ref{naive flatness rem}\eqref{naive flatness rem ramification} and footnote \ref{mwy} below.}
\[
   e \geq \min \Biggl\{\sum_{\varphi \in \Hom_{w,\psi}(F,\ov\BQ)} r_\varphi, \sum_{\varphi \in \Hom_{w,\psi}(F,\ov\BQ)} r_{\ov\varphi} \Biggr\}.
\]
\item\label{ram hypoth} For each place $v \in \CV_p$ which ramifies in $F$, $v$ is unramified over $p$ and the lattice $\Lambda_v$ is self-dual.
\item\label{diff hypoth} For each place $w$ of $F$ over $p$ and each $\psi \in \Hom_{\BQ_p}(F_w^t,\ov\BQ_p)$, the integers $r_\varphi$ for varying $\varphi \in \Hom_{w,\psi}(F,\ov\BQ)$ differ by at most one.
\end{altenumerate}
Then $\CM_{K_{\wt G}}^\naive(\wt G)_{O_{E,(\nu)}}$ is flat over $\Spec O_{E,(\nu)}$.
\end{theorem}

\begin{proof}  
After extending scalars to the $\nu$-adic completion $O_{E,(\nu)} \to O_{E,\nu}$, this follows from the \emph{local model diagram} over $O_{E,\nu}$, 
\begin{equation}\label{LMDfortildeG}
   \begin{tikzcd}[column sep=2ex]
        &  \wt{\CM}^\naive_{K_{\wt G}}(\wt G)_{O_{E,\nu}}  \ar[dl, "\pi"'] \ar[dr, "\wt\varphi"]\\
      \CM^\naive_{K_{\wt G}}(\wt G)_{O_{E,\nu}} & &  (\CM_0^\fka)_{O_{E,\nu}} \times_{O_{E,\nu}}\BM^\naive.
   \end{tikzcd}
\end{equation}
Let us briefly remark on the notation; see \cite[Ch.~6]{RZ1} or \cite[\S15]{PR2} for more details.  Let
\[
   \Lambda_p := \bigoplus_{v \in \CV_p}\Lambda_v \subset W \otimes_\BQ \BQ_p = \bigoplus_{v \in \CV_p}W_v.
\]
Let $\CL$ be the self-dual periodic multichain of $O_{F,p}$-lattices in $W \otimes_\BQ \BQ_p$ generated by $\Lambda_p$ and its dual.  A triple $(A,\iota,\lambda)$ as in the moduli problem for $\CM^\naive_{K_{\wt G}}(\wt G)$ then gives rise in a natural way to a polarized $\CL$-set of abelian varieties $\{A_\Lambda\}_{\Lambda \in \CL}$. For $S$ in $\LNSch_{/O_{E,\nu}}$, $\wt{\CM}^\naive_{K_{\wt G}}(\wt G)_{O_{E,\nu}}(S)$ is then the groupoid of objects $(A_0,\iota_0,\lambda_0,A,\iota,\lambda,\ov\eta)$ in $\CM^\naive_{K_{\wt G}}(\wt G)_{O_{E,\nu}}(S)$ equipped with an isomorphism of polarized multichains $\{\RH_1^{\text{dR}}(A_\Lambda)\}_{\Lambda \in \CL} \isoarrow \CL \otimes_{\BZ_p} \CO_S$.  The morphism $\pi$ in \eqref{LMDfortildeG} is the natural forgetful morphism; it is a torsor under $\CP_{O_{E, \nu}}$, where $\CP$ is the automorphism scheme of $\CL$ (as a polarized multichain) over $\BZ_p$, which is a smooth affine group scheme (in fact, a $\BZ_p$-model of the unitary similitude group $G^\BQ$ for $W$).  The \emph{naive local model} $\BM^\naive$ is a projective $O_{E,\nu}$-scheme attached to the multichain $\CL$ and the Shimura datum $(G^\BQ,\{h_{G^\BQ}\})$.  The group scheme $\CP_{O_{E,\nu}}$ acts naturally on $\BM^\naive$, and the morphism $\wt\varphi$ is $\CP_{O_{E,\nu}}$-equivariant and formally smooth of the same relative dimension as $\pi$.

The flatness of $\CM^\naive_{K_{\wt G}}(\wt G)_{O_{E,\nu}}$ now follows from the flatness of $\BM^\naive$ (and \'etaleness of $\CM_0^\fka$).  More precisely, by its definition, $\BM^\naive$ is a moduli space for certain $O_F \otimes_\BZ \CO_S$-linear quotient bundles of $\Lambda_p \otimes_{\BZ_p} \CO_S$.  The decomposition $O_{F_0} \otimes_\BZ \BZ_p \cong \prod_{v \in \CV_p} O_{F_0,v}$ then induces a natural decomposition 
\begin{equation}\label{decLMv}
   \BM^\naive \cong \prod_{v \in \CV_p} \BM(v)^\naive_{O_{E,\nu}},
\end{equation}
where each $\BM(v)^\naive_{O_{E,\nu}}$ is the base change to $O_{E,\nu}$ of a naive local model attached to the local $F_v/F_{0,v}$-hermitian space $W_v$, the lattice $\Lambda_v$, and the function $r|_{\cup_{w \mid v}\Hom_w(F,\ov\BQ)}$. Thus flatness of $\BM^\naive$ follows from flatness of each $\BM(v)^\naive$.  Now let $F_{0,v}^t$ denote the maximal unramified extension of $\BQ_p$ in $F_{0,v}$, and let $E(v)^\un \subset \ov\BQ_p$ denote the maximal unramified extension of the reflex field $E(v)$ of $\BM(v)^\naive$.  Consider the decomposition
\[
   O_{F_{0,v}^t} \otimes_{\BZ_p} O_{E(v)^\un} \cong \prod_{\psi_0 \in \Hom_{\BQ_p}(F_{0,v}^t, \ov\BQ_p)} O_{E(v)^\un}.
\]
After extending scalars to $O_{E(v)^\un}$, the action of $O_{F_{0,v}^t} \otimes_{\BZ_p} O_{E(v)^\un}$ on $\Lambda_v \otimes_{\BZ_p} O_{E(v)^\un}$ induces a natural decomposition
\begin{equation}\label{decLMpsi_0}
   \BM(v)_{O_{E(v)^\un}}^\naive \cong \prod_{\psi_0} \BM(v,\psi_0)^\naive_{O_{E(v)^\un}},
\end{equation}
where each $\BM(v,\psi_0)^\naive_{O_{E(v)^\un}}$ is the base change to $O_{E(v)^\un}$ of a naive local model attached to the tower $F_v/F_{0,v}/F_{0,v}^t$.  Thus the problem of flatness of $\BM^\naive$ reduces to flatness of each $\BM(v,\psi_0)^\naive_{O_{E(v)^\un}}$.

When $v$ is unramified in $F$, for each $\psi_0$, there is a further decomposition
\[
   O_{F_v^t} \otimes_{O_{F_{0,v}^t},\psi_0} O_{E(v)^\un} \cong \prod_\psi O_{E(v)^\un},
\]
where the product is over the two homomorphisms $\psi\colon F_v^t \to \ov\BQ_p$ extending $\psi_0$.  Picking one of these $\psi$, this decomposition induces an identification of $\BM(v,\psi_0)^\naive_{O_{E(v)^\un}}$ with the base change to $O_{E(v)^\un}$ of a naive local model $\BM'$ attached to the totally ramified extension $F_w/F_w^t$, the group $\Res_{F_w/F_w^t} \GL_n$, and the function $r|_{\Hom_{w,\psi}(F,\ov\BQ)}$; here $w$ is the place of $F$ over $v$ determined by $\psi$ (of course there is only ambiguity in $w$ when $v$ splits in $F$), and $\Hom_{w,\psi}(F,\ov\BQ)$ identifies with the embeddings of $F_w$ into $\ov\BQ_p$ extending $\psi$ as in \eqref{nogoodname} above.  (Replacing the choice of $\psi$ by $\ov\psi$ results in an isomorphic naive local model for $\Res_{F_{\ov w}/F_{\ov w}^t} \GL_n$.)  When $\Lambda_v$ is self-dual, $\BM'$ is the local model defined in \cite{PR1} in the case of a single lattice; by the hypotheses in \eqref{unram hypoth} and \eqref{diff hypoth}, $\BM'$ is flat by Th.~B and the following paragraph in loc.\ cit.\ (which we note relies, in turn, on a result of Weyman \cite{W}).  For general $\Lambda_v$ (still with $v$ unramified in $F$), $\BM'$ is a naive local model for $\Res_{F_w/F_w^t} \GL_n$ and $r|_{\Hom_{w,\psi}(F,\ov\BQ)}$ in the case of a periodic lattice chain $\CL_\psi$ generated by one or two lattices.  By G\"ortz \cite[\S1 Th.]{G2}, flatness of $\BM'$ in this case (or more generally, in the case of an arbitrary periodic lattice chain $\CL_\psi$) follows from flatness of the naive local model in the single lattice case.  This concludes the proof when $v$ is unramified in $F$.

When $v$ ramifies in $F$ (subject to the hypotheses in \eqref{ram hypoth} and \eqref{diff hypoth}), flatness of each $\BM(v,\psi_0)^\naive$ is proved in \cite{S2}.
\end{proof}

\begin{remark}
\begin{altenumerate}
\item\label{naive flatness rem ramification} It is conjectured just after Th.~B in \cite{PR1} that when $v$ is unramified in $F$ and $\Lambda_v$ is self-dual, the naive local model $\BM(v,\psi_0)^\naive$ appearing in the proof of Theorem \ref{naive flat O_E,nu} is flat with reduced special fiber, without any assumption on the ramification of $v$ over $p$.\footnote{\label{mwy} Since the preprint version of this paper appeared, Muthiah--Weekes--Yacobi have posted a proof \cite{MWY} of this conjecture in full generality, by proving \cite[Conj.~5.8]{PR1}.}  This would imply (again using G\"ortz \cite{G2} to pass to the case of general $\Lambda_v$) that hypothesis \eqref{inert split cond} in both of Theorems \ref{naive flat thm} and \ref{naive flat O_E,nu} can be removed.
\item We do not know if the conclusion of Theorem \ref{naive flat O_E,nu} remains valid if one allows the places $v$ which ramify in $F$ to have any ramification over $p$.  However, the assumption for such $v$ that $\Lambda_v$ is self-dual is necessary if $n > 1$.
\item\label{naive flatness rem nu} 
Condition \eqref{diff hypoth} in Theorem \ref{naive flat O_E,nu} is necessary for the conclusion to hold, cf.~\cite[Cor.~3.3]{PR1}.
\end{altenumerate}\label{naive flatness rem}
\end{remark}

\subsection{Flat subscheme of $\CM^\naive_{K_{\wt G}}(\wt G)$}\label{ss:flatness conds}
Even if $\CM^\naive_{K_{\wt G}}(\wt G)$ is not flat over $\Spec O_{E, (p)}$, we sometimes can strengthen the conditions on the abelian varieties $(A, \iota, \lambda)$ occurring in the moduli problem to define a closed substack $\CM_{K_{\wt G}}(\wt G)$ of $\CM^\naive_{K_{\wt G}}(\wt G)$ which is flat with the same generic fiber.  For simplicity, let us consider this question after base change to the completed local ring $O_{E,\nu}$ for $\nu$ a $p$-adic place of $E$.  Then, similarly to the proof of Theorem \ref{naive flat O_E,nu}, the $O_F$-action on the abelian variety $A$ induces an action of $O_F \otimes_\BZ \BZ_p \cong \prod_{w} O_{F,w}$ on $\Lie A$, and hence a canonical decomposition 
\begin{equation}\label{decLie_w}
   \Lie A = \bigoplus_{w} \Lie_w A,
\end{equation}
where $w$ runs through the $p$-adic places of $F$.  For each $w$, using the notation of Section \ref{ss:naive flat}, the $O_{F_w^t}$-action on $\Lie_w A$ similarly induces a decomposition
\begin{equation}\label{decLie_w,psi}
   \Lie_w A = \bigoplus_{\psi \in \Hom_{\BQ_p}(F_w^t,\ov\BQ_p)} \Lie_{w,\psi} A;
\end{equation}
here in fact we choose an embedding $\alpha \colon \ov\BQ \to \ov\BQ_p$ inducing $\nu$ as just before \eqref{alpha}, and the decomposition \eqref{decLie_w,psi} is defined after base changing the moduli problem to $\Spec O_{E_\nu^\un}$, where $E_\nu^\un$ denotes the maximal unramified extension of $E_\nu$ (embedded via $\alpha$) in $\ov\BQ_p$.

We now consider conditions on $\Lie_{w,\psi} A$ under various assumptions on $w$ and $\psi$, as follows.  In some cases the conditions are quite technical to formulate and we only give references.  We again write $\Hom_{w,\psi}(F,\ov\BQ) \subset \Hom(F,\ov\BQ)$ for the fiber over $\psi$ in the diagram \eqref{nogoodname}.  Furthermore, we note that cases \eqref{wedge sit}--\eqref{refined spin sit} below are disallowed when $p = 2$ by our standing assumption that the places $v \in \CV_2$ are unramified in $F$.

\begin{altenumerate}
\renewcommand{\theenumi}{\arabic{enumi}}
\item\label{banal eisenstein sit} 
Suppose that the restricted function $r|_{\Hom_{w,\psi}(F,\ov\BQ)}$ takes values in $\{0,n\}$ (a \emph{banal} signature type at $\psi$). Then there is the \emph{Eisenstein condition} on $\Lie_{w,\psi} A$ of \cite[(4.10)]{RSZ3}.\footnote{Strictly speaking, here we mean that the expression $Q_{A_\psi}(\iota(\pi))$ defined in loc.\ cit.\ is the zero endomorphism on $\Lie_{w,\psi} A$.}  (Note that, in contrast to cases \eqref{drinfeld eisenstein sit}--\eqref{refined spin sit} below, here we make no ramification assumptions on $w$; however, the Eisenstein condition at $\psi$ is already implied by the Kottwitz condition \eqref{kottcond} if $w$ is unramified over $p$.)
\item\label{drinfeld eisenstein sit} 
Suppose that $w$ is unramified over $F_0$, consider the conjugate embedding $\ov\psi \colon F_{\ov w}^t \to \ov\BQ_p$, and suppose that the restricted function $r|_{\Hom_{w,\psi}(F,\ov\BQ) \cup \Hom_{\ov w,\ov\psi}(F,\ov\BQ)}$ is of the form
\[
   r_\varphi =
   \begin{cases}
      n-1,  &  \text{for some $\varphi = \varphi_0 \in \Hom_{w,\psi}(F,\ov\BQ) \cup \Hom_{\ov w,\ov\psi}(F,\ov\BQ)$};\\
      1,  &  \text{$\varphi = \ov\varphi_0$};\\
      \text{$0$ or $n$},  &  \varphi \neq \varphi_0,\ov\varphi_0.
   \end{cases}
\]
If $\ov\varphi_0 \in \Hom_{w,\psi}(F,\ov\BQ)$, then there is the \emph{Eisenstein condition} on $\Lie_{w,\psi} A$ of \cite[(8.2)]{RZ2}.   If $\ov\varphi_0 \in \Hom_{\ov w,\ov\psi}(F,\ov\BQ)$, then there is the \emph{Eisenstein condition} on $\Lie_{\ov w,\ov\psi} A$ of \cite[(8.2)]{RZ2}. (This condition is again already implied by the Kottwitz condition if $w$ is unramified over $p$.)
\item\label{wedge sit} 
Suppose that $w$ is ramified over $F_0$ and the place $v$ of $F_0$ under $w$ is unramified over $p$; or, equivalently, that $\Hom_{w,\psi}(F,\ov\BQ)$ is of the form $\{\varphi_\psi,\ov\varphi_\psi\}$ for some $\varphi_\psi \in \Phi$.  Then there is the \emph{wedge condition} of Pappas \cite{P} at $\psi$: if $r_{\varphi_\psi} \neq r_{\ov\varphi_\psi}$, then
\begin{equation}\label{wedge cond}
   \left.
   \begin{aligned}
   \bigwedge\nolimits^{r_{\ov\varphi_\psi} + 1}\bigl(\iota(a) - \varphi_\psi(a) \mid \Lie_{w,\psi} A\bigr) &= 0\\
   \bigwedge\nolimits^{r_{\varphi_\psi} + 1}\bigl(\iota(a) - \ov\varphi_\psi(a) \mid \Lie_{w,\psi} A\bigr) &= 0
   \end{aligned}
   \right\}
   \quad\text{for all}\quad a \in O_{F,w}.
\end{equation}
Here, using that $r_{\varphi_\psi} \neq r_{\ov\varphi_\psi}$, it is easy to see that $\varphi_\psi$ and $\ov\varphi_\psi$ map $F_w$ into $E_\nu^\un$, and the expressions $\varphi_\psi(a)$ and $\ov\varphi_\psi(a)$ are then viewed as sections of the structure sheaf of the base scheme, as in \eqref{kottcond}.  (There is no condition at $\psi$ when $r_{\varphi_\psi} = r_{\ov\varphi_\psi}$.  There is an analogous condition when $w$ is unramified over $F_0$, but it is already implied by the Kottwitz condition.)
\item\label{spin sit} 
In the same situation as in \eqref{wedge sit}, suppose in addition that $n$ is even.  Then there is the \emph{spin condition} of \cite[\S8.2]{PR3} on $\Lie_{w,\psi} A$.\footnote{Strictly speaking, loc.~cit.\ only formulates the spin condition on the local  model. We will not spell out the translation of the spin condition to $\Lie_{w,\psi} A$ more explicitly; it is entirely analogous to the translation of the refined spin condition of \cite[\S2.5]{S1} to the Lie algebra of a $p$-divisible group given in \cite[\S7]{RSZ2} and just before Rem.\ 4.6 in \cite{RSZ3}.}  We note that in the special case that $\Lambda_v$ is $\pi_v$-modular (recall this means that $\Lambda_v^\vee = \pi_v\i \Lambda_v$) and $\{r_{\varphi_\psi},r_{\ov\varphi_\psi}\} = \{1,n-1\}$, and in the presence of the wedge condition at $\psi$, the spin condition at $\psi$ admits the simple formulation that the endomorphism $\iota(\pi_v) \mid \Lie_{w,\psi} A$ is nonvanishing at each point of the base, cf.\ \cite[\S6]{RSZ2} or \cite[(4.28)]{RSZ3}.
\item\label{refined spin sit}  
In the same situation as in \eqref{wedge sit}, suppose in addition that $n$ is odd.  Then there is the \emph{refined spin condition} of \cite[\S2.5]{S1}, translated to a condition on $\Lie_{w,\psi} A$ as in \cite[\S7]{RSZ2} and just before Rem.\ 4.6 in \cite{RSZ3}.   We note that these latter two references treat on the nose the special case that $\Lambda_v$ is \emph{almost $\pi_v$-modular} (i.e.,\ $\Lambda_v \subset \Lambda_v^\vee \subset \pi_v\i\Lambda_v$ with $\dim_{O_{F,v}/\pi_v O_{F,v}} \pi_v\i\Lambda_v/\Lambda_v^\vee = 1$) and $\{r_{\varphi_\psi},r_{\ov\varphi_\psi}\} = \{1,n-1\}$.
\end{altenumerate}

\begin{theorem}\label{flatsub}
Suppose that for every $p$-adic place $w$ of $F$ and every embedding $\psi\colon F_w^t \to \ov\BQ_p$, \emph{one of the following hypotheses} is satisfied and, in each case, impose the given condition on the object $(A,\iota,\lambda)$ in the moduli problem for $\CM_{K_{\wt G}}^\naive(\wt G)_{O_{E_\nu^\un}}$.  Throughout, let $v$ denote the place of $F_0$ under $w$.
\begin{altenumerate}
\renewcommand{\theenumi}{\textit{\alph{enumi}}}
\item\label{a} $w$ is unramified over $F_0$ and $v$ satisfies the ramification hypothesis in Theorem \ref{naive flat O_E,nu}(\emph{\kern-.45ex\ref{unram hypoth}}\kern-.2ex); or $w$ is ramified over $F_0$, $v$ is unramified over $p$, and the lattice $\Lambda_v$ is self-dual.  Furthermore, the $r_\varphi$'s for varying $\varphi \in \Hom_{w,\psi}(F,\ov\BQ)$ differ by at most one.  In this case, impose no further condition.
\item\label{b} $w$ and $\psi$ satisfy the assumption in \eqref{banal eisenstein sit} above.  Then impose the Eisenstein condition  of \eqref{banal eisenstein sit} on $\Lie_{w,\psi} A$.
\item\label{c} $w$ and $\psi$ satisfy the assumptions in \eqref{drinfeld eisenstein sit} above and the lattice $\Lambda_v$ is self-dual or $\pi_v$-modular.\footnote{The definition of $\pi_v$-modular when $v$ is unramified in $F$ is word-for-word the same as when $v$ ramifies in $F$, namely that $\Lambda_v^\vee = \pi_v\i\Lambda_v$.} In the notation of \eqref{drinfeld eisenstein sit}, if $\ov\varphi_0 \in \Hom_{w,\psi}(F,\ov\BQ)$, then impose the Eisenstein condition of \eqref{drinfeld eisenstein sit} on $\Lie_{w,\psi} A$;  if $\ov\varphi_0 \in \Hom_{\ov w,\ov\psi}(F,\ov\BQ)$, then impose the Eisenstein condition of \eqref{drinfeld eisenstein sit} on $\Lie_{\ov w,\ov\psi} A$.
\item\label{d} $w$ and $\psi$ satisfy the assumptions in \eqref{wedge sit} above and the lattice $\Lambda_v$ is self-dual. Then impose the wedge condition on $\Lie_{w,\psi} A$.
\item\label{e} $w$ and $\psi$ satisfy the assumptions in \eqref{spin sit} above, $\Lambda_v$ is $\pi_v$-modular, and $\{r_{\varphi_\psi},r_{\ov\varphi_\psi}\} = \{1,n-1\}$.  Then impose the wedge condition and the spin condition on $\Lie_{w,\psi} A$.
\item\label{f} $w$ and $\psi$ satisfy the assumptions in \eqref{refined spin sit} above, $\Lambda_v$ is almost $\pi_v$-modular,
and $\{r_{\varphi_\psi},r_{\ov\varphi_\psi}\} = \{1,n-1\}$.  Then impose the refined spin condition on $\Lie_{w,\psi} A$.
\end{altenumerate}
Then these conditions descend to define a closed substack $\CM_{K_{\wt G}}(\wt G)_{O_{E,\nu}}$ of $\CM_{K_{\wt G}}^\naive(\wt G)_{O_{E,\nu}}$ which is flat over $\Spec O_{E,\nu}$ with the same generic fiber as $\CM_{K_{\wt G}}^\naive(\wt G)_{O_{E,\nu}}$.
\end{theorem}

\begin{proof}
As in the proof of Theorem \ref{naive flat O_E,nu}, the statements on flatness and the generic fiber reduce to statements on the local model $\BM_{O_{E_\nu^\un}} \subset \BM_{O_{E_\nu^\un}}^\naive$ defined by the analogous conditions on $\BM_{O_{E_\nu^\un}}^\naive$.  As in \eqref{decLMv} and \eqref{decLMpsi_0}, there is a product decomposition
\begin{equation}\label{decLM with conds}
   \BM_{O_{E_\nu^\un}} \cong \prod_{\substack{v \in \CV_p \\ \psi_0\colon\! F_{0,v}^t \rightarrow \ov\BQ_p}} \BM(v,\psi_0)_{O_{E_\nu^\un}},
\end{equation}
and we reduce to proving flatness factor-by-factor on the right-hand side.  For each $p$-adic place $w$ of $F$ and embedding $\psi\colon F_w^t \to \ov\BQ_p$, the place $v$ of $F_0$ under $w$ and the restriction $\psi|_{F_0^t}$ indexes one of the factors in \eqref{decLM with conds}, and all indices in the product arise in this way as $w$ and $\psi$ vary.  Thus we consider cases based on the type of $w$ and $\psi$.  If $w$ and $\psi$ are as in \eqref{a}, then the factor $\BM(v,\psi|_{F_{0,v}^t})_{O_{E_\nu^\un}} = \BM(v,\psi|_{F_{0,v}^t})_{O_{E_\nu^\un}}^\naive$ is flat by the proof of Theorem \ref{naive flat O_E,nu}.  If $w$ and $\psi$ are as in \eqref{b}, then the factor $\BM(v,\psi|_{F_{0,v}^t})_{O_{E_\nu^\un}}$ is flat (in fact, trivial) with the same generic fiber as $\BM(v,\psi|_{F_{0,v}^t})_{O_{E_\nu^\un}}^\naive$ by \cite[App.\ B]{RSZ3}.  If $w$ and $\psi$ are as in \eqref{c}, then $\BM(v,\psi|_{F_{0,v}^t})_{O_{E_\nu^\un}}$ is flat (in fact, smooth) with the same generic fiber by \cite[Lem.\ 8.6]{RZ2}.  If $w$ and $\psi$ are as in \eqref{d}, then $\BM(v,\psi|_{F_{0,v}^t})_{O_{E_\nu^\un}}$ is flat with the same generic fiber by \cite{S2} (when $\{r_{\varphi_\psi},r_{\ov\varphi_\psi}\} = \{1,n-1\}$, this was proved by Pappas in \cite{P}).  If $w$ and $\psi$ are as in \eqref{e}, then $\BM(v,\psi|_{F_{0,v}^t})_{O_{E_\nu^\un}}$ is flat (in fact, smooth) with the same generic fiber by \cite[\S5.3]{PR3}.  If $w$ and $\psi$ are as in \eqref{f}, then $\BM(v,\psi|_{F_{0,v}^t})_{O_{E_\nu^\un}}$ is flat (in fact, smooth) with the same generic fiber by \cite[Th.\ 1.4]{S1}.  By assumption, every $w$ and $\psi$ is of one of these types, and this completes the proof of the statements on flatness and the generic fiber. The statement on descent is easy to verify.
\end{proof}

\begin{remark}
The refined spin condition in \eqref{refined spin sit} above is defined in \cite{S1} for $n$ even as well as odd, and it is shown there to imply the wedge condition and the spin condition (without changing the generic fiber).  Therefore one may treat \eqref{wedge sit}--\eqref{refined spin sit} above in a uniform way by imposing the refined spin condition in each case; the advantage of the wedge condition and (to a lesser extent) the spin condition is only that they are simpler to state.  It is conjectured in loc.\ cit.\ that the refined spin condition produces flatness for any signature type $\{r_{\varphi_\psi},r_{\ov\varphi_\psi}\}$ and any lattice type (still with the place $v$ unramified over $p$).  When $v$ ramifies in $F$ and is \emph{ramified} over $p$, nothing is known about characterizing the (flat) local model in terms of an explicit moduli problem.
\end{remark}

\begin{remark}
Suppose that $v \in \CV_p$ splits in $F$, say $v = w \ov w$, and suppose that the restricted function $r|_{\Hom_w(F,\ov\BQ)}$ is of the form
\begin{equation}\label{drinfeld function}
   r_\varphi =
   \begin{cases}
      n-1,  &  \text{$\varphi = \varphi_0$ for some $\varphi_0 \in \Hom_w(F,\ov\BQ)$};\\
      n,  &  \varphi \in \Hom_w(F,\ov\BQ) \ssm \{\varphi_0\}.
   \end{cases}
\end{equation}
Then it is possible to impose a \emph{Drinfeld level structure} at $v$ in the moduli problem appearing in Theorem \ref{flatsub}.  More precisely, let $m$ be a nonnegative integer, and define $K_{G,v}^m$ to be the principal congruence subgroup mod $\fkp_v^m$ inside $K_{G,v}$, where $\fkp_v$ denotes the prime ideal in $O_{F_0}$ determined by $v$.  Let
\[
   K_{\wt G}^m := K_{Z^\BQ}^\circ \times K_G^p \times K_{G,v}^m \times \prod_{v'\in\CV_p\ssm \{v\}}K_{G, v'} \subset K_{\wt G}.
\]
Then one can extend the definition of $\CM_{K_{\wt G}}(\wt G)_{O_{E,\nu}}$ to the case of the level subgroup $K_{\wt G}^m$ by adding a Drinfeld level-$m$ structure at $v$.  Briefly, \eqref{drinfeld function} implies that in the decomposition \eqref{decLie_w} of $\Lie A$, the summand $\Lie_w A$ has rank $n[F_w:\BQ_p] - 1$, and the summand $\Lie_{\ov w} A$ has rank $1$.  The datum we add to the moduli problem is an $O_{F,\ov w}$-linear homomorphism of finite flat group schemes, 
\[
   \varphi\colon \pi_{\ov w}^{-m}\Lambda_{\ov w}/\Lambda_{\ov w} \to \uHom_{O_{F, \ov w}}(A_0[\ov w^m],A[\ov w^m]),
\]
which is a Drinfeld $\ov w^m$-structure on the target.  Here $\Lambda_{\ov w}$ is the summand attached to $\ov w$ in the natural decomposition $\Lambda_v = \Lambda_w \oplus \Lambda_{\ov w}$, with $\Lambda_v$ the vertex lattice at $v$ chosen prior to Definition \ref{def:inttuple}.  See \cite[\S4.3]{RSZ3} (which we note interchanges the roles of $w$ and $\ov w$) for more details.
\end{remark}

\subsection{Exotic smoothness and regularity}
In some special cases the conditions introduced in Section \ref{ss:flatness conds} define integral models with good or semi-stable reduction, beyond the totally unramified situation in Theorem \ref{intmoduliforshim}\eqref{intmoduliforshim ii}.  For simplicity, we will again consider these questions after base change to the completed local ring $O_{E,\nu}$ for $\nu$ a $p$-adic place of $E$.  We again choose an embedding $\alpha\colon \ov\BQ \to \ov\BQ_p$ inducing $\nu$ and use the notation introduced before Theorem \ref{naive flat O_E,nu}.

The following gives conditions when $\CM_{K_{\wt G}}(\wt G)_{O_{E,\nu}}$ is known to be smooth.

\begin{theorem}\label{exotsm thm}
In the setting of Theorem \ref{flatsub}, suppose that every pair $(w,\psi)$ is of type \eqref{b}, \eqref{c}, \eqref{e}, \eqref{f}, or the following special case of type \eqref{a}:
\begin{altenumerate}
\renewcommand{\theenumi}{\textit{\alph{enumi}}$'$}
\item\label{a'} $w$ is unramified over $p$ and the lattice $\Lambda_v$ is self-dual or $\pi_v$-modular.
\end{altenumerate}
Then the integral model $\CM_{K_{\wt G}}(\wt G)_{O_{E,\nu}}$ defined in Theorem \ref{flatsub} is smooth over $\Spec O_{E,\nu}$.
\end{theorem}

\begin{proof}
Similarly to the proof of Theorem \ref{flatsub}, this reduces to smoothness of the local model $\BM$.  For $(w,\psi)$ of type \eqref{b}, \eqref{c}, \eqref{e}, or \eqref{f}, the factor $\BM(v,\psi|_{F_{0,v}^t})_{O_{E_\nu^\un}}$ in the decomposition \eqref{decLM with conds} (again denoting by $v$ the place of $F_0$ under $w$) is smooth by the references given in the proof of Theorem \ref{flatsub} (additionally using that smoothness of this factor in type \eqref{f} is due to Richarz \cite[Prop.~4.16]{A}).  Smoothness of this factor in type \eqref{a'} is standard (it is isomorphic to a Grassmannian for $\GL_n$).
\end{proof}

The fact that smoothness can occur in types \eqref{c}, \eqref{e}, and \eqref{f}, when ramification is present, is a surprising phenomenon termed \emph{exotic smoothness} in \cite{RSZ1,RSZ2}.

The following theorem gives conditions when $\CM_{K_{\wt G}}(\wt G)_{O_{E,\nu}}$ is known to have semi-stable reduction, and hence to be regular.

\begin{theorem}\label{sstab thm}
Let $n \geq 2$.  In the setting of Theorem \ref{flatsub}, suppose that there is a $p$-adic place $w_0$ of $F$, with place $v_0$ of $F_0$ under it, of the following special case of type \eqref{a}:
\begin{altenumerate}
\renewcommand{\theenumi}{\textit{\alph{enumi}}$''$}
\item\label{a''} $w_0$ is unramified over $p$, and the restricted function $r|_{\Hom_{w_0}(F,\ov\BQ)\cup \Hom_{\ov w_0}(F,\ov\BQ)}$ is of the form
\begin{equation}\label{sstab sig assumpt}
   r_\varphi =
   \begin{cases}
      \text{arbitrary},  &  \text{for some $\varphi = \varphi_0,\ov\varphi_0 \in \Hom_{w_0}(F,\ov\BQ) \cup \Hom_{\ov w_0}(F,\ov\BQ)$};\\
      \text{$0$ or $n$},  &  \varphi \neq \varphi_0,\ov\varphi_0.
   \end{cases}
\end{equation}
Furthermore, $\Lambda_{v_0}$ is \emph{almost self-dual} (i.e.,\ $\Lambda_{v_0} \subset \Lambda_{v_0}^\vee \subset \pi_{v_0}\i \Lambda_{v_0}$ with $\rank_{O_{F,v_0}/\pi_v O_{F,v_0}} \Lambda_{v_0}^\vee/\Lambda_{v_0} = 1$) or \emph{almost $\pi_{v_0}$-modular} (i.e.,\ $\Lambda_{v_0} \subset \Lambda_{v_0}^\vee \subset \pi_{v_0}\i \Lambda_{v_0}$ with $\rank_{O_{F,v_0}/\pi_v O_{F,v_0}} \Lambda_{v_0}^\vee/\Lambda_{v_0} = n-1$), or $\{r_{\varphi_0}, r_{\ov\varphi_0}\}=\{1, n-1\}$. 
\end{altenumerate}
In addition, suppose that every pair $(w,\psi)$ as in Theorem \ref{flatsub} for which $w \neq w_0,\ov w_0$ is of type \eqref{b}, \eqref{c}, or type \eqref{a'} in Theorem \ref{exotsm thm}.  Furthermore, suppose that $E_\nu$ is unramified over $\BQ_p$. Then $\CM_{K_{\wt G}}(\wt G)_{O_{E,\nu}}$ has semi-stable reduction over $\Spec O_{E,\nu}$.
\end{theorem}

\begin{proof}
As before, the proof is via the local model, using in particular the product decomposition \eqref{decLM with conds}.  We first consider the factors in \eqref{decLM with conds} corresponding to $v_0$. Note that $F_{0, v_0}=F^t_{0, v_0}$ by hypothesis \eqref{a''}. Regarding $\varphi_0$ as an embedding $F_{w_0} \to \ov\BQ_p$ via $\alpha_*$ as in \eqref{nogoodname}, consider the restriction $\varphi_0|_{F_{0,v_0}}$, and let $E_0 \subset \ov\BQ_p$ denote the reflex field of the factor $\BM(v_0,\varphi_0|_{F_{0,v_0}})$.  By G\"ortz \cite[\S 4.4.5]{G1}, $\BM(v_0,\varphi_0|_{F_{0,v_0}})$ has semi-stable reduction over $O_{E_0}$ under the assumption that $\Lambda_{v_0}$ is almost self-dual or or almost $\pi_{v_0}$-modular, or under the assumption that $\{ r_{\varphi_0}, r_{\ov\varphi_0}\}=\{1, n-1\}$.  Since $E_\nu$ is unramified over $\BQ_p$, $\BM(\varphi_0|_{F_{0,v_0}})_{O_{E_\nu^\un}}$ has semi-stable reduction over $O_{E_\nu^\un}$.  It follows from the assumption on $r$ in \eqref{sstab sig assumpt} that the factors $\BM(v_0,\psi_0)_{O_{E_\nu^\un}}$ for $\psi_0 \neq \varphi_0|_{F_{0,v_0}}$ are isomorphic to $\Spec O_{E_\nu^\un}$ \cite[App.\ B]{RSZ3}.  By Theorem \ref{exotsm thm}, the factors $\BM(v,\psi_0)_{O_{E_\nu^\un}}$ for $v \neq v_0$ are smooth, and the theorem follows.
\end{proof}

\begin{remark}
The proof of Theorem \ref{exotsm thm} shows that if $(w,\psi)$ is of type \eqref{e} or \eqref{f}, then the factor of the local model in \eqref{decLM with conds} obtained from $(w,\psi)$ is also smooth.  However, we cannot allow the presence of such types in Theorem \ref{sstab thm}, since they would result in $E_\nu$ being \emph{ramified} over $\BQ_p$, which would destroy semi-stable reduction of the factor $\BM(v_0,\varphi_0|_{F_{0,v_0}})$ after extending scalars from $O_{E_0}$ to $O_{E_\nu}$.
\end{remark}

\begin{remark}
As is transparent from the above, smoothness and semi-stability of the $p$-integral models of the Shimura variety follow from the corresponding property of the local models.  We refer to \cite{HPR} for a classification,  under certain hypotheses, of general local models which are smooth, resp.\ have semi-stable reduction.  In particular, let us single out one case of a deeper level subgroup (cf.\ Remark \ref{r:deeper level}) in which semi-stable reduction arises.  Let $n$, $w_0$, and $v_0$ be as in Theorem \ref{sstab thm}, and modify the definition of type \eqref{a''} to require that $\{r_{\varphi_0}, r_{\ov\varphi_0}\}=\{1, n-1\}$ and to allow the level subgroup $K_{G,v_0}$ at $v_0$ to be the stabilizer in $\RU(W)(F_{0,v_0})$ of \emph{any} self-dual periodic lattice chain.  Then, by work of Drinfeld \cite{Dr} (see also \cite[\S 4.4.5]{G1}), the factor $\BM(v_0,\varphi_0|_{F_{0,v_0}})$ of the local model has semi-stable reduction.  Hence, provided that the other factors of the local model are smooth and that $E_\nu$ is unramified over $\BQ_p$, the corresponding moduli stack will have semi-stable reduction over $\Spec O_{E,\nu}$.
\end{remark}

\section{Global integral models of RSZ Shimura varieties}\label{s:global}
It is sometimes of interest to construct models of $M_{K_{\wt G}}(\wt G)$ over $\Spec O_E$.  Rather than striving for maximal generality, in this section we single out two situations where this can be done.  In both cases, we take the signature function $r$ for the $n$-dimensional space $W$ to be of fake Drinfeld type relative to a fixed element $\varphi_0\in\Phi$, cf.\ Example \ref{eg:special types}\eqref{eg:fake drinfeld}.  Recall that, in this case, $\varphi_0$ embeds $F\to E$ for $n\geq 2$, cf.\ Example \ref{eg:fake drinfeld E}. Hence each finite place $\nu$ of $E$ induces a place $w_\nu$ of $F$ and a place $v_\nu$ of $F_0$ via $\varphi_0$ for such $n$. We set
\begin{equation*}
   \CV_\ram := \{\, \text{finite places $v$ of $F_0$} \mid \text{$v$ ramifies in $F$} \,\},
\end{equation*}
and we assume that all $v \in \CV_\ram$ are unramified over $\BQ$ and do not divide $2$.

\subsection{Integral models with exotic good reduction}\label{s:glob exot good}
In this subsection we define global integral models which, when $n \geq 2$, have so-called exotic good reduction at all places $\nu$ of $E$ such that the induced place $v_\nu$ of $F_0$ ramifies in $F$ (and which, when $n = 1$, are \'etale over $\Spec O_E$).  We fix $\fka$, $\sqrt\Delta$, $\xi$, $\Lambda_0$, and $W_0$ as before \eqref{compisom}. As usual, we set $V = \Hom_F(W_0,W)$, and we endow $W$ with the $\BQ$-valued alternating form $\tr_{F/\BQ}\sqrt\Delta\i \sform$. We fix an $O_F$-lattice $\Lambda \subset W$ whose localization is a vertex lattice with respect to this form at every finite place $v$:
\[
   \Lambda_v \subset \Lambda_v^\vee \subset \pi_v\i\Lambda_v.
\]
Of course, as for any $O_F$-lattice in $W$, the localization $\Lambda_v$ is necessarily self-dual for all but finitely many $v$.  We define the finite set
\begin{equation}\label{V_in^Lambda}
   \CV^\Lambda_\mathrm{un} := \{\, \text{finite places $v$ of $F_0$}\mid \text{$v$ is  unramified in $F$ and $\Lambda_v \varsubsetneq \Lambda_v^\vee \varsubsetneq \pi_v\i\Lambda_v$} \,\}.
\end{equation}
In addition to our assumptions on $\CV_\ram$ at the beginning of this section, we impose on the tuple $(F/F_0, W, \Lambda)$ the following conditions.
\begin{altitemize}
\item
\it All $v \in \CV^\Lambda_{\mathrm{un}}$ are unramified over $\BQ$.
\item
{\it If $v\in\CV_\ram$, then the localization  $\Lambda_v$ of $\Lambda$ is $\pi_v$-modular if $n$ is even, and almost $\pi_v$-modular if $n$ is odd (see Theorem \ref{flatsub}\eqref{e}\eqref{f} for the definitions of these terms).}
\end{altitemize}
Starting from an arbitrary CM extension $F/F_0$ and $n$-dimensional hermitian space $W$, we note that if $n$ is odd, then such a $\Lambda$ always exists in $W$; whereas if $n$ is even, then such a $\Lambda$ exists if and only if $W_v$ is a split hermitian space for all $v \in \CV_\ram$ and for all finite $v$ which are inert in $F$ and ramified over $\BQ$.  We set
$$
   K_G^\circ := \bigl\{\, g\in G(\BA_f) \bigm| g(\Lambda\otimes \wh\BZ)=\Lambda\otimes \wh\BZ \,\bigr\} ,
$$
and, as usual, we define $K_{\wt G}^\circ := K^\circ_{Z^\BQ}\times K_G^\circ$. 

We formulate a moduli problem $\CF_{K^\circ_{\wt G}}(\wt G)$ over $\Spec O_{E}$ as follows.  As earlier in the paper, to lighten  notation, we suppress the dependence on the ideal $\fka$ and the element $\xi$. 

\begin{definition}\label{def RSZ glob}
The category functor $\CF_{K^\circ_{\wt G}}(\wt G)$ associates to each $O_{E}$-scheme $S$ the groupoid of tuples $(A_0,\iota_0,\lambda_0,A,\iota,\lambda)$, where
\begin{altitemize}
\item $(A_0,\iota_0,\lambda_0)$ is an object of $\CM_0^{\fka, \xi}(S)$; 
\item $A$ is an abelian scheme over $S$;
\item $\iota\colon O_F \to \End(A)$ is an action satisfying the Kottwitz condition \eqref{kottcond} of signature type $r$ on $O_F$; and
\item $\lambda$ is a polarization on $A$ whose Rosati involution satisfies condition \eqref{Ros} on $O_F$.
\end{altitemize}
 We also impose that the kernel of the polarization $\lambda$ is of the type prescribed in Definition \ref{def:inttuple}\eqref{sglob cond i} for every $p$, relative to the lattice $\Lambda$ fixed above. Furthermore, we impose for every finite place $\nu$ of $E$ that after base-changing $(A,\iota,\lambda)$ to $S\otimes_{O_E}O_{E, \nu}$, the resulting triple satisfies the conditions on $\Lie A$ imposed in the definition of $\CM_{K_{\wt G}^\circ}(\wt G)_{O_{E,\nu}}$ in Theorem \ref{flatsub}.  In particular, we note that when $n \geq 2$ and $\nu$ is such that the induced place $w_\nu$ of $F$ is ramified over $F_0$, this entails imposing the wedge condition and spin condition in Theorem \ref{flatsub}\eqref{e} when $n$ is even, and the refined spin condition in Theorem \ref{flatsub}\eqref{f} when $n$ is odd, on the appropriate summand of $\Lie A$.
 
 Finally, we impose the sign condition that at every geometric point $\ov s$ of $S$,
\begin{equation}\label{signcond2}
   \inv^r_v(A_{0,\ov s}, \iota_{0,\ov s}, \lambda_{0,\ov s}, A_{\ov s}, \iota_{\ov s}, \lambda_{\ov s}) = \inv_v(V) ,
\end{equation}
for every finite place $v$ of $F_0$ which is non-split in $F$. 

A morphism $(A_0,\iota_0,\lambda_0,A,\iota,\lambda) \to (A_0',\iota_0',\lambda_0',A',\iota',\lambda')$ in this groupoid is given by an isomorphism $\mu_0\colon (A_0,\iota_0,\lambda_0) \isoarrow (A_0',\iota_0',\lambda_0')$ in $\CM_0^{\fka,\xi}(S)$ and an $O_F$-linear isomorphism $\mu\colon A \isoarrow A'$ of abelian schemes pulling $\lambda'$ back to $\lambda$.
\end{definition}

The following theorem (the extension of \cite[Th.\ 5.2]{RSZ3} to the present setting) shows that the moduli functor $\CF_{K^\circ_{\wt G}}(\wt G)$ defines an extension of $M_{K^\circ_{\wt G}}(\wt G)$ over $\Spec O_E$ with good properties. It follows immediately from Theorem \ref{flatsub}  and Theorems  \ref{exotsm thm} and \ref{sstab thm}.

\begin{theorem}\label{globalnolevel}
The moduli problem $\CF_{K^\circ_{\wt G}}(\wt G)$ is representable by a Deligne--Mumford stack $\CM_{K^\circ_{\wt G}}(\wt G)$ flat over $\Spec O_{E}$. For every finite place $\nu$ of $E$, the base change of $\CM_{K^\circ_{\wt G}} (\wt G)$ to $\Spec O_{E,\nu}$ is isomorphic to the $\nu$-adic integral moduli space of Theorem \ref{flatsub} in the case of the level subgroup $K^\circ_{\wt G}$.
Hence:
\begin{altenumerate}
\item If $n \geq 2$, then $\CM_{K^\circ_{\wt G}}(\wt G)$ is smooth of relative dimension $n-1$ over the open subscheme of $\Spec O_E$ obtained by removing all finite places $\nu$ for which the induced place $v_\nu$ of $F_0$ lies in $\CV^\Lambda_{\mathrm{un}}$. If $n = 1$, then $\CM_{K^\circ_{\wt G}}(\wt G)$ is finite \'etale over all of $\Spec O_E$.
\item If $n \geq 2$, then $\CM_{K^\circ_{\wt G}}(\wt G)$ has semi-stable reduction over the open subscheme of $\Spec O_E$ obtained by removing all finite places $\nu$ ramified over $\BQ$ for which the induced place $v_\nu$ lies in $\CV^\Lambda_{\mathrm{un}}$.\qed
\end{altenumerate}   
\end{theorem}

\begin{remark}\label{level str recipe}
The isomorphism in Theorem \ref{globalnolevel} between the base change $\CM_{K^\circ_{\wt G}}(\wt G) \otimes_{O_E} O_{E,\nu}$ and the moduli space of Theorem \ref{flatsub} can be made \emph{canonical} in terms of the lattices $\Lambda_0 \subset W_0$ and $\Lambda \subset W$ fixed prior to Definition \ref{def RSZ glob}. Indeed, given an object $(A_0,\iota_0,\lambda_0,A,\iota,\lambda)$ in the moduli problem of Definition \ref{def RSZ glob} over an $O_{E,\nu}$-scheme, one defines the prime-to-$p$ level $\ov\eta^p$ to be the set of all isometries $\wh\RV^p(A_0,A) \isoarrow V \otimes_F \BA_{F,f}^p$ carrying $\wh\RT^p(A_0,A)$ to $\Hom_{O_F}(\Lambda_0,\Lambda) \otimes_{O_F} \wh O_F^p$.
\end{remark}

\begin{remark}\label{prod stack}
Consider the moduli problem that associates to each $O_E$-scheme $S$ the groupoid of triples $(A,\iota,\lambda)$ as in the last three bullet points of Definition \ref{def RSZ glob}, where the kernel of $\lambda$ is of the type prescribed in Definition \ref{def:inttuple}\eqref{sglob cond i} for every $p$ with respect to our fixed $\Lambda$, and such that for every finite place $\nu$ of $E$, the base change of $(A,\iota,\lambda)$ to $S\otimes_{O_E}O_{E, \nu}$ satisfies the conditions on $\Lie A$ imposed in Theorem \ref{flatsub}.  Via essentially the same proof as for Theorem \ref{flatsub}, this moduli problem is represented by a Deligne--Mumford stack $\CM_r$ which is flat over $\Spec O_E$.  Then the stack $\CM_{K^\circ_{\wt G}}(\wt G)$ of Theorem \ref{globalnolevel} admits the simple description as the open and closed substack of
\begin{equation*}
   \CM_0^{\fka,\xi} \times_{\Spec O_E} \CM_r
\end{equation*}
where the sign condition \eqref{signcond2} holds pointwise.

We note that $\CM_r$ is an integral model for a finite disjoint union of copies of the  Shimura variety $S(G^\BQ,  \{h_{G^\BQ}\})$ of Kottwitz type for \emph{maximal level structure}; therefore the  previous description  explains the relation between  the  integral model  $\CM_r$ and the integral model $\smash[b]{\CM_{K^\circ_{\wt G}}(\wt G)}$ of the RSZ Shimura variety $S(\wt G,  \{h_{\wt G}\})$  for  level structure  $K^\circ_{\wt G}$.
\end{remark}

\begin{remark}
\begin{altenumerate}
\item\label{inert sign}
Let $v$ be a finite place of $F_0$ which is non-split in $F$.  Let $\ell$ denote the residue characteristic of $v$. Let $k$ be an $O_E$-algebra which is an algebraically closed field of characteristic not $\ell$. Generalizing somewhat from the setting of Definition \ref{def RSZ glob}, let $(A_0,\iota_0,\lambda_0,A,\iota,\lambda)$ be a tuple consisting of an abelian variety $A_0$ over $\Spec k$ of dimension $[F_0 : \BQ]$, an action $\iota_0\colon O_{F,(\ell)} \to \End_{(\ell)}(A_0)$ up to prime-to-$\ell$ isogeny, a quasi-polarization $\lambda_0$ on $A_0$ such that $\Ros_\lambda(\iota_0(a)) = \iota_0(\ov a)$ for all $a \in O_{F,(\ell)}$, and a triple $(A,\iota,\lambda)$ of the same form, except where $A$ has dimension $n\cdot [F_0 : \BQ]$.  Consider the $O_{F,v}$-lattice $\RT_v(A_0)$ inside the one-dimensional $F_v$-vector space $\RV_v(A_0)$, and let $\RT_v(A_0)^\vee$ denote the dual lattice with respect to $\lambda_0$ and the Weil pairing. Similarly define $\RT_v(A)^\vee$ inside the $n$-dimensional vector space $\RV_v(A)$.  Say that
\begin{equation}\label{m}
   \RT_v(A_0)^\vee = \pi_v^m \RT_v(A_0)
\end{equation}
inside $\RV_v(A_0)$. Recall the $F_v/F_{0,v}$-hermitian space $\RV_v(A_0,A) = \Hom_{F_v}(\RV_v(A_0),\RV_v(A))$ from \eqref{V_v(A_0,A)}. Then the $O_{F,v}$-lattice
\[
   \RT_v(A_0,A) := \Hom_{O_{F,v}}\bigl(\RT_v(A_0),\RT_v(A)\bigr) \subset \RV_v(A_0,A)
\]
has hermitian dual
\begin{equation}\label{hd1}
   \RT_v(A_0,A)^* = \Hom_{O_{F,v}}\bigl(\RT_v(A_0)^\vee,\RT_v(A)^\vee\bigr) = \pi_v^m \Hom_{O_{F,v}}\bigl(\RT_v(A_0),\RT_v(A)^\vee\bigr).
\end{equation}
Similarly, for the moment let $\Lambda_{0,v} \subset W_{0,v}$ and $\Lambda_v \subset W_v$ be any $O_{F,v}$-lattices, and suppose that $\Lambda_{0,v}^\vee = \pi_v^m\Lambda_{0,v}$. Then the $O_{F,v}$-lattice
\[
   L_v := \Hom_{O_{F,v}}(\Lambda_{0,v},\Lambda_v) \subset V_v
\]
has hermitian dual
\begin{equation}\label{hd2}
   L_v^* = \Hom_{O_{F,v}}(\Lambda_{0,v}^\vee,\Lambda_v^\vee) = \pi_v^m \Hom_{O_{F,v}}(\Lambda_{0,v},\Lambda_v^\vee).
\end{equation}
If the quasi-polarization $\lambda$ is such that the relative position of $\RT_v(A)$ and $\RT_v(A)^\vee$ in $\RV_v(A)$ is the same as that of $\Lambda_v$ and $\Lambda_v^\vee$ in $W_v$, then it follows from \eqref{hd1} and \eqref{hd2} that the relative position of $\RT_v(A_0,A)$ and $\RT_v(A_0,A)^*$ in $\RV_v(A_0,A)$ is the same as that of $L_v$ and $L_v^*$ in $V_v$.  In particular, this will be the case if $(A_0,\iota_0,\lambda_0)$ arises from a $k$-point on $\CM_0^{\fka,\xi}$, the lattice $\Lambda_{0,v}$ is the localization at $v$ of the global lattice $\Lambda_0 \subset W_0$ fixed prior to Definition \ref{def RSZ glob}, and $\lambda$ induces an honest polarization of $v$-divisible groups $A[v^\infty] \to A^\vee[v^\infty]$ whose kernel satisfies the condition in Defintion \ref{def:inttuple}\eqref{sglob cond i} relative to $\Lambda_v$.

Now suppose that $v$ is \emph{inert} in $F$. Then the split and non-split $n$-dimensional $F_v/F_{0,v}$-hermitian spaces are distinguished by whether $\ord_v (\det B)$ is respectively even or odd, for $B$ the change-of-basis matrix going from any $F_v$-basis of the vector space to its dual basis with respect to the hermitian form (independent of the choice of basis).  Hence $\RT_v(A_0,A)$ and $\RT_v(A_0,A)^*$ having the same relative position as $L_v$ and $L_v^*$ implies that $\RV_v(A_0,A)$ and $V_v$ are isometric hermitian spaces.  Hence the sign condition \eqref{signcond2} for $v$ is automatically satisfied at all points $\ov s$ of residue characteristic not $\ell$.  It then follows from flatness of the moduli problem\footnote{More precisely, the analogous moduli problem defined without the sign condition is also flat over $\Spec O_E$, since flatness is purely a question of the local model at each finite place of $E$.} over $\Spec O_E$ in Theorem \ref{globalnolevel} and local constancy of $\inv_v^r$ as a function on general base schemes \cite[Prop.\ A.1]{RSZ3} that the sign condition is automatically satisfied everywhere, for all inert $v$.

Now suppose that $n$ is \emph{even} and $v$ \emph{ramifies} in $F$.  Then the split and non-split $n$-dimensional $F_v/F_{0,v}$-hermitian spaces are distinguished by whether they respectively do or do not contain a $\pi_v^i$-modular lattice $M$ with respect to the hermitian form (meaning that $M^* = \pi_v^{-i}M$) for some, or equivalently any, odd integer $i$.  Similarly to the previous paragraph, this implies that the sign condition \eqref{signcond2} is automatically satisfied at $v$ (the main new point being that at points of the moduli space of residue characteristic not $\ell$, the integer $m$ in \eqref{m} must be even when $v$ is ramified, and the condition on $\lambda$ in Definition \ref{def:inttuple}\eqref{sglob cond i} at $v$ relative to a $\pi_v$-modular $\Lambda_v$ then forces $\RV_v(A_0,A)$ to be split).  We conclude that \emph{the full sign condition \eqref{signcond2} is automatically satisfied when $n$ is even}. In particular, the open and closed embedding
\[
   \CM_{K^\circ_{\wt G}}(\wt G) \subset \CM_0^{\fka,\xi} \times_{\Spec O_E} \CM_r
\]
of Remark \ref{prod stack} is an equality when $n$ is even.
\item\label{Q sign}
Suppose that $F_0 = \BQ$.  Then the sign condition can be replaced by the condition that for every geometric point $\ov s$ of $S$, there exists an isomorphism of hermitian $O_{F,\ell}$-lattices
\[
   \Hom_{O_{F,\ell}}\bigl(\RT_\ell(A_{0,\ov s}),\RT_\ell(A_{\ov s})\bigr) \simeq \Hom_{O_{F,\ell}}(\Lambda_{0,\ell},\Lambda_\ell)
\]
for every prime number $\ell \neq \charac \kappa(\ov s)$, cf.\ \cite[\S2.3]{BHKRY}.  Indeed, this follows from the product formula and the Hasse principle for hermitian forms, cf.\ Remark \ref{excone}. 
\end{altenumerate}\label{redsign}
We finally note that these remarks using flatness  are  also applicable to the  semi-global moduli problems  in Theorem \ref{flatsub}. 
\end{remark}

\subsection{Integral models for principal polarization}
In this subsection we generalize the integral models of \cite{BHKRY} from the case $F_0 = \BQ$ to the case of arbitrary totally real $F_0$.  
Throughout, we take $\fka = O_{F_0}$ and assume that $M_0^{O_{F_0}} \neq \emptyset$; recall from Remark \ref{remMa}\eqref{rem exa satisfied} that this assumption is satisfied whenever $F/F_0$ is ramified at some finite place.  We let $\sqrt\Delta$, $\xi$, $\Lambda_0$, $W_0$, and $V$ all be as in Section \ref{s:glob exot good}. We assume that $W$ contains an $O_F$-lattice $\Lambda$ which is \emph{self-dual} for $\tr_{F/\BQ}\sqrt\Delta\i \sform$, and we fix such a $\Lambda$ once and for all. (Note that this implies that the set $\CV^\Lambda_{\mathrm{un}}$ defined as in \eqref{V_in^Lambda} is empty in the present case.) We set, as in the previous subsection, 
\[
   K_G^\circ := \bigl\{\, g\in G(\BA_f)\bigm| g(\Lambda\otimes \wh\BZ)=\Lambda\otimes \wh\BZ \,\bigr\} ,
\]
and, as usual, we define $K_{\wt G}^\circ=K^\circ_{Z^\BQ}\times K_G^\circ$. 

In the present situation, we formulate the following variant of the moduli problem in Definition \ref{def RSZ glob}.  As usual, we suppress the ideal $\fka = O_{F_0}$ and the element $\xi$ in the notation.

\begin{definition}\label{glob princ pol}
The category functor $\CF_{K^\circ_{\wt G}}(\wt G)$ associates to each $O_{E}$-scheme $S$ the groupoid of tuples $(A_0,\iota_0,\lambda_0,A,\iota,\lambda)$, where $(A_0,\iota_0,\lambda_0,A,\iota)$ is as in Definition \ref{def RSZ glob}, and where
\begin{altitemize}
\item $\lambda$ is a  \emph{principal} polarization whose Rosati involution satisfies condition \eqref{Ros} on $O_F$.
\end{altitemize}
Again, we impose the sign condition \eqref{signcond2} for every finite non-split place of $F_0$.  Likewise, we impose for every finite place $\nu$ of $E$ that after base-changing $(A,\iota,\lambda)$ to $S\otimes_{O_E}O_{E, \nu}$, the resulting triple satisfies the conditions on $\Lie A$ imposed in the definition of $\CM_{K_{\wt G}^\circ}(\wt G)_{O_{E,\nu}}$ in Theorem \ref{flatsub}.  In particular, we note that for the places $w$ of $F$ which are ramified over $F_0$ and of the same residue characteristic as $\nu$, this entails imposing the wedge condition in Theorem \ref{flatsub}\eqref{d} on each summand $\Lie_{w,\psi} A$ (when the signature function $r|_{\Hom_{w,\psi}(F,\ov\BQ)}$ is of banal type, it is equivalent to impose the Eisenstein condition in Theorem \ref{flatsub}\eqref{b}).

The morphisms in this groupoid are as in Definition \ref{def RSZ glob}.
\end{definition}

\begin{theorem}\label{BHKRYmod}
The moduli problem $\CF_{K^\circ_{\wt G}}(\wt G)$ is representable by a Deligne--Mumford stack $\CM_{K^\circ_{\wt G}}(\wt G)$ flat over $\Spec O_{E}$. For every finite place $\nu$ of $E$, the base change of $\CM_{K^\circ_{\wt G}} (\wt G)$ to $\Spec O_{E,\nu}$ is canonically isomorphic to the $\nu$-adic integral moduli space of Theorem \ref{flatsub} in the case of the level subgroup $K^\circ_{\wt G}$. Furthermore:
\begin{altenumerate}
\item\label{BHKRY i} 
If $n \geq 2$, then $\CM_{K^\circ_{\wt G}}(\wt G)$ is smooth of relative dimension $n-1$ over the open subscheme of $\Spec O_E$ obtained by removing the set $\CV_\ram(E)$ of finite places $\nu$ for which the induced place $v_\nu$ of $F_0$ lies in $\CV_\ram$.  If $n = 1$, then $\CM_{K^\circ_{\wt G}}(\wt G)$ is finite \'etale over all of $\Spec O_E$.
\item\label{BHKRY ii} 
If $n \geq 2$, then the fiber of $\CM_{K^\circ_{\wt G}}(\wt G)$ over a place $\nu \in \CV_\ram(E)$ has only isolated singularities. If $n \geq 3$, then blowing up these isolated points for all such $\nu$ yields a model  $\CM^\nat_{K^\circ_{\wt G}}(\wt G)$ which has semi-stable reduction, and hence is regular, over the open subscheme of $\Spec O_E$ obtained by removing the places $\nu \in \CV_\ram(E)$ which are ramified over $F$.  This model represents the moduli problem $\CF^\nat_{K^\circ_{\wt G}}(\wt G)$ formulated below.
\end{altenumerate}   
\end{theorem}

\begin{proof}
Representability of $\CF_{K^\circ_{\wt G}}(\wt G)$ is standard, and the statement on the base change to $\Spec O_{E,\nu}$ is obvious, with level structures defined as in Remark \ref{level str recipe} (thus the isomorphism over $\Spec O_{E,\nu}$ is made canonical via the choice of $\Lambda_0$ and $\Lambda$). Flatness follows from Theorem \ref{flatsub}.  Assertion \eqref{BHKRY i} follows from Theorem \ref{exotsm thm}. Assertion \eqref{BHKRY ii} reduces to a statement on the local model.  More precisely, let $\nu \in \CV_\ram(E)$, let $p$ denote the residue characteristic of $\nu$, and consider the decomposition of the local model $\BM_{O_{E_\nu^\un}}$ in \eqref{decLM with conds} relative to the choice of an embedding $\alpha\colon \ov\BQ \to \ov\BQ_p$ inducing $\nu$.  Since the signature type is of fake Drinfeld type, the factors $\BM(v,\psi_0)_{O_{E_\nu^\un}}$ for all $v\neq v_\nu$ are trivial by the Eisenstein condition.  Since we assume that every $v \in \CV_\ram$ is unramified over $\BQ$, we have $F_{0,v_\nu} = F_{0,v_\nu}^t$, and similarly $\BM(v_\nu, \psi_0)_{O_{E_\nu^\un}}$ is trivial for all embeddings $\psi_0\colon F_{0,v_\nu} \to \ov\BQ_p$ except the single embedding $\psi_{0,\alpha}$ induced by $\alpha \circ \varphi_0$.  The factor $\BM(v_\nu, \psi_{0,\alpha})_{O_{E_\nu^\un}}$ is then the base change to $\Spec O_{E_\nu^\un}$ of the local model for $\GU_n(F_{v_\nu}/F_{0,v_\nu})$ in the case of a self-dual lattice and signature type $(n-1,1)$.  By Pappas \cite[Th.~4.5 \& its proof]{P}, when $n \geq 2$, this local model is singular at a single point, with blowup at this point of semi-stable reduction.  When $n \geq 3$, the reflex field of this local model identifies with $F_{v_\nu}$, and therefore semi-stable reduction is preserved after the base change $O_{F,v_\nu} \to O_{E_\nu^\un}$ provided $\nu$ is unramified over $F$.
\end{proof}

\begin{remark}
When $n = 2$, the reflex field of the local model for $\GU_2(F_{v_\nu}/F_{0,v_\nu})$ appearing in the proof of Theorem \ref{BHKRYmod} is $F_{0,v_\nu}$, not $F_{v_\nu}$.  In this case, the local model itself has semi-stable reduction over $\Spec O_{F_0,v_\nu}$, without needing to blow up.  However, since $F_{v_\nu}$ maps into $E_\nu$, the extension $O_{F_0,v_\nu} \to O_{E_\nu^\un}$ is necessarily ramified, and hence semi-stable reduction is lost upon base change.
\end{remark}

\begin{remark}
There is an obvious analog of Remark \ref{prod stack} in the present situation, where the moduli problem for $\CM_r$ is replaced by the analogous one with respect to our self-dual lattice $\Lambda$ (in particular, the polarization $\lambda$ in the resulting moduli problem is principal; when $F_0 = \BQ$, the resulting stack is denoted $\CM_{(n-1,1)}^{\textrm{Pap}}$ in \cite[\S2.3]{BHKRY}).
Furthermore, by the same argument as in Remark \ref{redsign}\eqref{inert sign}, the sign condition \eqref{signcond2} imposed in Definition \ref{glob princ pol} is redundant at all inert places $v$.  However, it is no longer the case that the sign condition is redundant at ramified places $v$ when $n$ is even (since in the ramified case, for any $n \in \BZ_{>0}$, both isometry types of $n$-dimensional $F_v/F_{0,v}$-hermitian spaces contain a self-dual lattice).  Finally, Remark \ref{redsign}\eqref{Q sign} transposes word-for-word to the present situation, cf.\ \cite[\S2.3]{BHKRY}.
\end{remark}

Here is the moduli problem mentioned in Theorem \ref{BHKRYmod}\eqref{BHKRY ii} (the \emph{Kr\"amer model}).  Let $n \geq 2$.

\begin{definition}\label{def:Kraemer}
The category functor $\CF^\nat_{K^\circ_{\wt G}}(\wt G)$ associates to each $O_{E}$-scheme $S$ the groupoid of tuples $(A_0,\iota_0,\lambda_0,A,\iota,\lambda, \CP)$, where $(A_0,\iota_0,\lambda_0,A,\iota,\lambda)\in \CF_{K^\circ_{\wt G}}(\wt G)(S)$, and where $\CP\subset \Lie A$ is an $O_F$-stable $\CO_S$-submodule which, Zariski-locally on $S$, is an $\CO_S$-free direct summand satisfying the \emph{rank condition} \eqref{rankcond}  and  the \emph{Kr\"amer--Eisenstein condition} \eqref{KE} below.
\end{definition}

The rank condition and the Kr\"amer--Eisenstein condition mentioned in Definition \ref{def:Kraemer} are conditions for every finite place $\nu$ of $E$.  Fixing $\nu$, let $p$ denote the residue characteristic of $\nu$, and choose an embedding $\alpha\colon \ov\BQ\to\ov\BQ_p$ inducing $\nu$ as before \eqref{alpha}.  For any $p$-adic place $w$ of $F$ and any $\BQ_p$-embedding $\psi\colon F_{w}^t\to\ov\BQ_p$, define $\Hom_{w,\psi}(F,\ov\BQ) \subset \Hom(F, \ov\BQ)$ as in \eqref{nogoodname}.  We recall that the resulting partition $\Hom(F, \ov\BQ) = \coprod_{w,\psi} \Hom_{w,\psi}(F,\ov\BQ)$ as in \eqref{Hom(F,ovBQ) decomp} depends only on $\nu$ up to labeling of the sets on the right-hand side.  Let
\[
   r_{w, \psi}^\Phi := \sum_{\varphi\in\Hom_{w,\psi}(F, \ov\BQ)\cap \Phi} r_\varphi .
\]
Then the rank condition for $\nu$ on the $O_F\otimes_\BZ \CO_S$-module $\CP$ states that in the decomposition analogous to \eqref{decLie_w} and \eqref{decLie_w,psi},
\begin{equation}\label{CPdec}
  \CP = \bigoplus_{w, \psi} \CP_{w,\psi},
\end{equation}
we have
\begin{equation}\label{rankcond}
   \rank_{\CO_S}\CP_{w,\psi}=r_{w, \psi}^\Phi
\end{equation}
for all $w,\psi$. Here, as in \eqref{decLie_w,psi}, the decomposition \eqref{CPdec} is defined when $S$ is an $O_{E_\nu^\un}$-scheme, and the rank condition for all $\nu$ then descends to a condition on $O_E$-schemes.

The Kr\"amer--Eisenstein condition for $\nu$ is similarly a condition on each summand $\CP_{w,\psi}$ in \eqref{CPdec}.  The statement of the condition depends on the restricted function $r|_{w,\psi} := r|_{\Hom_{w,\psi}(F,\ov\BQ)}$, and involves polynomials closely related to those in the formulation of the Eisenstein condition in \cite[\S8]{RZ2}.  For each $w,\psi$, let
\begin{equation}\label{C sets}
\begin{aligned}
   C_{w,\psi}^\Phi &:= \bigl\{\, \varphi \in \Hom_{w,\psi}(F,\ov \BQ) \cap \Phi \bigm| r_\varphi \neq 0 \,\bigr\},\\
   C_{w,\psi}^{\ov\Phi} &:= \bigl\{\, \varphi \in \Hom_{w,\psi}(F,\ov \BQ) \cap \ov\Phi \bigm| r_\varphi \neq 0 \,\bigr\}.
\end{aligned}
\end{equation}
Let $\pi$ be a uniformizer in $F_w$, and define the polynomials in $\ov\BQ_p[T]$,
\begin{equation}
   Q_{C_{w,\psi}^\Phi}(T) := \prod_{\varphi\in C_{w,\psi}^\Phi} \bigl(T - \varphi(\pi)\bigr) \in \ov\BQ_p[T]
   \quad\text{and}\quad
   Q_{C_{w,\psi}^{\ov\Phi}}(T) := \prod_{\varphi\in C_{w,\psi}^{\ov\Phi}} \bigl(T - \varphi(\pi)\bigr).
\end{equation}
Here and below we implicitly use $\alpha$ to identify $C_{w,\psi}^\Phi$ and $C_{w,\psi}^{\ov\Phi}$ with subsets of $\Hom_{\BQ_p}(F_w,\ov\BQ_p)$.
Then the Kr\"amer--Eisenstein condition on $\CP_{w,\psi}$ is that
\begin{equation}\label{KE}
   Q_{C_{w,\psi}^\Phi}(\pi \otimes 1)|_{\CP_{w,\psi}} = 0
   \quad\text{and}\quad
   Q_{C_{w,\psi}^{\ov\Phi}}(\pi \otimes 1)|_{\Lie_{w,\psi} A/ \CP_{w,\psi}} = 0;
\end{equation}
here the condition is defined when $S$ is a scheme over $\Spec O_L$ for any subfield $L \subset \ov\BQ_p$ large enough to contain the image of $E_\nu$ under $\alpha$, the image of $F_w^t$ under $\psi$, and the coefficients of the polynomials $Q_{C_{w,\psi}^\Phi}(T)$ and $Q_{C_{w,\psi}^{\ov\Phi}}(T)$.  The Kr\"amer--Eisenstein condition \eqref{KE} is independent of the choice of uniformizer $\pi$. Indeed, this  follows immediately from the next lemma, upon taking the field $L$  to be large enough to contain $\varphi(F_w)$ for all $\varphi \in C_{w,\psi}^\Phi \cup C_{w,\psi}^{\ov\Phi}$.

\begin{lemma}
Let $\pi$ and $\pi'$ be two uniformizers of $F_w$. Then the elements $\pi\otimes1-1\otimes \pi$ and $\pi'\otimes1-1\otimes \pi'$ are unit multiples of each other in the ring $O_{F_w} \otimes_{O_{F_w^t}} O_{F_w}$.
\end{lemma}

\begin{proof}
  Say $\pi' = u\pi$, with $u \in O_{F_w}^\times$. Let $e := [F_w:F_w^t]$.  Then there exist unique $u_0, \dotsc, u_{e-1}\in O_{F_w^t}$, with $u_0\in O_{F_w^t}^\times$, such that $u=u_0+u_1\pi+\cdots+u_{e-1}\pi^{e-1}$.  Hence
\[
   \pi'\otimes 1-1\otimes \pi' = u_0(\pi\otimes 1-1\otimes\pi)+\dotsb + u_{e-1}(\pi^e\otimes 1-1\otimes\pi^e) .
\]
Every term on the right-hand side is divisible by $\pi\otimes 1-1\otimes\pi$. Factoring out, we obtain an equation of the form
\[
   \pi'\otimes 1-1\otimes \pi' = (\pi\otimes 1-1\otimes \pi)(u_0+a) ,
\]
where $a\in O_{F_w}\otimes_{O_{F_w^t}}O_{F_w}$ lies in the maximal ideal generated by $\pi\otimes 1$ and $1\otimes\pi$. This completes the proof.
\end{proof}

The Kr\"amer--Eisenstein condition for $\nu$ is that \eqref{KE} holds for all $w,\psi$ with $w$ of the same residue characteristic as $\nu$.  The (full) Kr\"amer--Eisenstein condition is that the Kr\"amer--Eisenstein condition for $\nu$ holds for all $\nu$; this condition again descends to $O_E$-schemes.

As a first step towards understanding the rank and Kr\"amer--Eisenstein conditions, let $\wt E$ denote the composite of $E$ and the normal closure of $F$ in $\ov\BQ$, and let $O_{\wt E}'$ be the ring obtained from $O_{\wt E}$ by inverting the finitely many rational primes $p$ which ramify in $F$.  Then
\[
   O_F \otimes_\BZ O_{\wt E}' \cong \prod_{\varphi \in \Hom(F,\ov\BQ)} O_{\wt E}'.
\]
If $S$ is an $O_{\wt E}'$-scheme and $(A_0,\iota_0,\lambda_0,A,\iota,\lambda)$ is an $S$-point on $\CF_{K^\circ_{\wt G}}(\wt G)$, the $O_F$-action $\iota$ thus induces a canonical decomposition
\[
   \Lie A = \bigoplus_{\varphi \in \Hom(F,\ov\BQ)} \Lie_\varphi A.
\]
By the Kottwitz condition, $\rank_{\CO_S} \Lie_\varphi A = r_\varphi$.  By the definition of $O_{\wt E}'$, the set $\Hom_{w,\psi}(F,\ov\BQ)$ is a singleton set $\{\varphi\}$ for all places $w$ such that $S$ has a nonempty fiber of the same residue characteristic as $w$.  Therefore the rank condition imposes that $\CP_\varphi$ equals $\Lie_\varphi A$ or $0$ according as $\varphi \in \Phi$ or $\varphi\notin\Phi$.  Furthermore, the resulting $\CO_S$-module $\CP = \bigoplus_\varphi \CP_\varphi$ obviously satisfies the Kr\"amer--Eisenstein condition.  We conclude that the datum of $\CP$ in the moduli problem $\CF^\nat_{K^\circ_{\wt G}}(\wt G)$ is redundant over $\Spec O_{\wt E}'$---and hence, by descent, over $\Spec O_E'$, where $O_E'$ is the ring obtained from $O_E$ by inverting the primes $p$ that ramify in $F$.  In fact, a stronger statement is true.

\begin{lemma}\label{KE redundant lemma}
Keep the notation above, and assume that $S$ is a scheme over $\Spec O_L$ for $L \subset \ov\BQ_p$ sufficiently large. If $r|_{w,\psi}$ is banal or $w$ is unramified over $F_0$, then there exists a unique summand $\CP_{w,\psi} \subset \Lie_{w,\psi} A$ satisfying the rank condition \eqref{rankcond} and the Kr\"amer--Eisenstein condition \eqref{KE}.
\end{lemma}

\begin{proof}
To show existence and uniqueness of $\CP_{w,\psi}$, it suffices to solve the analogous problem on the local model, as in the proofs of Theorems \ref{naive flat O_E,nu} and \ref{flatsub}.  The local model is a moduli functor of $O_F \otimes_\BZ \CO_S$-linear quotients $\Lambda \otimes_\BZ \CO_S \surj \CQ$ satisfying certain conditions.  Since $L$ is sufficiently large, we have the usual direct sum decompositions
\[
   \bigoplus_{w',\psi'} \Lambda_{w',\psi',S}  \surj \bigoplus_{w',\psi'} \CQ_{w',\psi'},
\]
where $\Lambda_{w',\psi',S} := \Lambda \otimes_{O_F} O_{F,w'} \otimes_{O_{F_{w'}^t},\psi'} \CO_S$.  Our problem is to show that there exists a unique subbundle $\CP_{w,\psi}' \subset \CQ_{w,\psi}$ satisfying the (analogs on the local model of) the rank condition and the Kr\"amer--Eisenstein condition.  Throughout the rest of the proof, we will make use of the set
\[
   A_{w,\psi} := \bigl\{\, \varphi \in \Hom_{w,\psi}(F,\ov \BQ)  \bigm| r_\varphi = n \,\bigr\}
\]
and the polynomial
\begin{equation}\label{Q_A_w,psi}
   Q_{A_{w,\psi}}(T) := \prod_{\varphi\in A_{w,\psi}} \bigl(T - \varphi(\pi)\bigr),
\end{equation}
cf.\ \cite[(2.3), (2.8)]{RZ2} (in the case of the extension $F_w/\BQ_p$).  As before, in \eqref{Q_A_w,psi} we have implicitly used $\alpha$ to identify $A_{w,\psi}$ with a subset of $\Hom_{\BQ_p}(F_w,\ov\BQ_p)$.

First suppose that $r|_{w,\psi}$ is banal.  Then $C_{w,\psi}^\Phi \cup C_{w,\psi}^{\ov\Phi} = A_{w,\psi}$, and, by the definition of the moduli problem, $\CQ_{w,\psi}$ is required to satisfy the Eisenstein condition $Q_{A_{w,\psi}}(\pi \otimes 1)|_{\CQ_{w,\psi}} = 0$, cf.\ \cite[(B.5)]{RSZ3}. Furthermore, the Kottwitz condition in the banal case implies that $\CQ_{w,\psi}$ has $\CO_S$-rank $n \cdot \# A_{w,\psi}$. This forces
\[
   \ker\bigl[\Lambda_{w,\psi,S} \surj \CQ_{w,\psi}\bigr] = Q_{A_{w,\psi}}(\pi \otimes 1) \cdot \Lambda_{w,\psi,S} \subset Q_{C_{w,\psi}^\Phi}(\pi \otimes 1) \cdot \Lambda_{w,\psi,S},
\]
cf.\ \cite[(B.6)]{RSZ3}.  Hence
\begin{equation}\label{coks}
   \cok\bigl(Q_{C_{w,\psi}^\Phi}(\pi \otimes 1)|_{\CQ_{w,\psi}}\bigr) \cong \cok\bigl(Q_{C_{w,\psi}^\Phi}(\pi \otimes 1)|_{\Lambda_{w,\psi,S}}\bigr).
\end{equation}
Since the right-hand side of \eqref{coks} is a free $\CO_S$-module of rank $n \cdot \# C_{w,\psi}^\Phi = r_{w,\psi}^\Phi$, we conclude that $\ker(Q_{C_{w,\psi}^\Phi}(\pi \otimes 1)|_{\CQ_{w,\psi}})$ is a direct summand of $\CQ_{w,\psi}$ of the same rank $r_{w,\psi}^\Phi$.  The rank condition \eqref{rankcond} and the first relation in the Kr\"amer--Eisenstein condition \eqref{KE} then force the equality
\begin{equation}\label{also}
   \CP_{w,\psi}' = \ker\bigl(Q_{C_{w,\psi}^\Phi}(\pi \otimes 1)|_{\CQ_{w,\psi}}\bigr).
\end{equation}
(One sees easily that \eqref{also} also equals $Q_{C_{w,\psi}^{\ov\Phi}}(\pi \otimes 1) \cdot \CQ_{w,\psi}$.)  This completes the proof in the banal case.

Now suppose that $w$ is unramified over $F_0$ and $r|_{w,\psi}$ is non-banal.  Then $(w,\psi)$ is one of the pairs $(w_\nu,\psi_\alpha)$ and $(\ov w_\nu,\ov\psi_\alpha)$, where $\psi_\alpha$ denotes the embedding $F_{w_\nu}^t \to \ov\BQ_p$ induced by $\alpha\circ\varphi_0$. (These pairs are distinct, by unramifiedness).  In the case of $(\ov w_\nu,\ov\psi_\alpha)$, since the signature type is of fake Drinfeld type, we have $\ov\varphi_0 \in \Hom_{\ov w_\nu, \ov\psi_\alpha}(F,\ov\BQ)$ with $r_{\ov\varphi_0} = 1$, and $\CQ_{\ov w_\nu, \ov\psi_\alpha}$ satisfies the Eisenstein condition given in \cite[(8.2)]{RZ2}.  Let
\[
   \CK_{\ov w_\nu, \ov\psi_\alpha} := \ker\Bigl[\Lambda_{\ov w_\nu,\ov\psi_\alpha,S} \overset{f}\surj \CQ_{\ov w_\nu, \ov\psi_\alpha}\Bigr].
\]
Taking $V = Q_{A_{\ov w_\nu,\ov \psi_\alpha}}(\pi \otimes 1) \cdot \Lambda_{\ov w_\nu,\ov\psi_\alpha,S}$ and $W = \Lambda_{\ov w_\nu,\ov\psi_\alpha,S}$ in \cite[Lem.~4.10]{RZ2}, and using the Eisenstein condition, we conclude that $\CK_{\ov w_\nu, \ov\psi_\alpha} \subset Q_{A_{\ov w_\nu,\ov \psi_\alpha}}(\pi \otimes 1) \cdot \Lambda_{\ov w_\nu,\ov\psi_\alpha,S}$; comp.\ the proof of \cite[Lem.\ 8.6]{RZ2}.  From here, since $C_{\ov w_\nu,\ov\psi_\alpha}^\Phi \subset A_{\ov w_\nu,\ov\psi_\alpha}$ in the present case, the same argument as in the previous paragraph shows that we again must have $\CP_{\ov w_\nu,\ov\psi_\alpha}' = \ker(Q_{C_{\ov w_\nu,\ov\psi_\alpha}^\Phi}(\pi \otimes 1)|_{\CQ_{\ov w_\nu,\ov\psi_\alpha}})$.

In the case of the pair $(w_\nu,\psi_\alpha)$, we have $C_{w_\nu,\psi_\alpha}^{\ov\Phi} \subset A_{w_\nu,\psi_\alpha}$.  Therefore the rank condition \eqref{rankcond} and the second relation in the Kr\"amer--Eisenstein condition \eqref{KE} force that the inverse image of $\CP_{w_\nu,\psi_\alpha}'$ in $\Lambda_{w_\nu,\psi_\alpha,S}$ is $Q_{C_{w_\nu,\psi_\alpha}^{\ov\Phi}}(\pi \otimes 1) \cdot \Lambda_{w_\nu,\psi_\alpha,S}$.  This uniquely determines $\CP_{w_\nu,\psi_\alpha}'$ if it exists.  In turn, existence holds if and only if
\begin{equation}\label{whocares}
   \CK_{w_\nu,\psi_\alpha} \subset Q_{C_{w_\nu,\psi_\alpha}^{\ov\Phi}}(\pi \otimes 1) \cdot \Lambda_{w_\nu,\psi_\alpha,S},
\end{equation}
where $\CK_{w_\nu,\psi_\alpha} := \ker[\Lambda_{w_\nu,\psi_\alpha,S} \surj \CQ_{w_\nu, \psi_\alpha}]$.  To show the containment \eqref{whocares}, it suffices to show that
\begin{equation}\label{want}
   \CK_{w_\nu,\psi_\alpha} \subset Q_{A_{w_\nu,\psi_\alpha}}(\pi \otimes 1) \cdot \Lambda_{w_\nu,\psi_\alpha,S}.
\end{equation}
Now, by the formalism of local models, self-duality of the lattice $\Lambda$ gives rise to a perfect pairing
\[
   \Lambda_{w_\nu,\psi_\alpha,S} \times \Lambda_{\ov w_\nu,\ov\psi_\alpha,S} \to \CO_S
\]
under which the $O_F$-actions on the two factors are conjugate-adjoint, and such that $\CK_{w_\nu,\psi_\alpha} \subset \Lambda_{w_\nu,\psi_\alpha,S}$ and $\CK_{\ov w_\nu, \ov\psi_\alpha} \subset \Lambda_{\ov w_\nu,\ov\psi_\alpha,S}$ are the perp-modules of each other.  Thus \eqref{want} is equivalent to the containment
\begin{equation}\label{have}
   \bigl(Q_{A_{w_\nu,\psi_\alpha}}(\pi \otimes 1) \cdot \Lambda_{w_\nu,\psi_\alpha,S}\bigr)^\perp \subset \CK_{\ov w_\nu, \ov\psi_\alpha}.
\end{equation}
It is a pleasant exercise to show that the left-hand side in \eqref{have} equals
\[
   Q_0(\ov\pi \otimes 1) Q_{A_{\ov w_\nu,\ov\psi_\alpha}}(\ov\pi \otimes 1) \cdot \Lambda_{\ov w_\nu,\ov\psi_\alpha,S},
\]
where the polynomials $Q_{A_{\ov w_\nu,\ov\psi_\alpha}}$ and $Q_0(T) := T - \ov\varphi_0(\ov\pi)$ are defined for the field $F_{\ov w_\nu}$ with respect to the uniformizer $\ov\pi$.  Thus the containment \eqref{have} holds by the Eisenstein condition \cite[(8.2)]{RZ2}, which completes the proof.
\end{proof}

It follows from Lemma \ref{KE redundant lemma} that the natural forgetful morphism
\begin{equation}\label{forget}
   \CF^\nat_{K^\circ_{\wt G}}(\wt G) \to \CF_{K^\circ_{\wt G}}(\wt G)
\end{equation}
is an isomorphism over the open locus $\Spec (O_E[\CV_\ram\i]) \subset \Spec (O_E)$.  On the other hand,  let $\nu$ now be a place lying over some $v_\nu\in\CV_\ram$.  Then the functor $\CF^\nat_{K^\circ_{\wt G}}(\wt G)_{O_{E, (\nu)}}$ can be understood via the corresponding local model for $\CM_{K^\circ_{\wt G}}(\wt G)_{O_{E, (\nu)}}$.  As in the proof of Theorem \ref{BHKRYmod},  the base change to $\Spec O_{E_\nu^\un}$ of the local model $\BM$ is a product
$$
   \BM_{O_{E_\nu^\un}}=\prod_{v, \psi_0}\BM(v, \psi_0)_{O_{E_\nu^\un}}. 
$$
By Lemma \ref{KE redundant lemma}, for all factors except the one indexed by $(v_\nu, \psi_{0,\alpha})$, the datum of $\CP$ is redundant; here $\psi_{0,\alpha}$ is as in the proof of Theorem \ref{BHKRYmod}. The factor $\BM(v_\nu, \psi_{0,\alpha})_{O_{E_\nu^\un}}$ is the base change to $O_{E_\nu^\un}$ of the Pappas local model for $\GU_n(F_{w_\nu}/F_{0,v_\nu})$, for a self-dual lattice and signature type $(n-1,1)$, cf.\ \cite{P}.   The datum of $\CP$ corresponds to  a point of the \emph{Kr\"amer local model} \cite{Kr} mapping to the Pappas local model.

\appendix
\section{Local model diagram for unramified PEL data of type $A$}\label{s:LM}
In this appendix, we change notation from the main body of the paper.  Let $(F, B, V, \sform)$ be rational data of PEL type over $\BQ_p$ in the sense of \cite[\S1.38]{RZ1}. Let $*$ be the induced involution on $B$. Let $O_B$ be a maximal order of $B$ invariant under $*$. Let $F_0$ denote the invariants of $*$ in $F$, and let $O_{F_0}$ and $O_F$ denote the respective rings of integers. Our main purpose is to prove the following theorem, which is a special case of \cite[Th.~3.16]{RZ1} when $p\ne 2$. We adopt the terminology of loc.~cit.

\begin{theorem}\label{normal forms}
Assume that $O_F$ is an \'etale $O_{F_0}$-algebra free of rank $2$. 

Let $\CL$ be a self-dual multichain of $O_B$-lattices in $V$. Let $T$ be a $\BZ_p$-scheme on which $p$ is locally nilpotent. Let $\{M_\Lambda\mid \Lambda\in\CL\}$ be a polarized multichain of $O_B\otimes_{\BZ_p}\CO_T$-modules of type $(\CL)$. Then locally for the \'etale topology on $T$, the polarized multichain $\{M_\Lambda\}$ is isomorphic to $\CL\otimes_{\BZ_p}\CO_T$. 

Furthermore, if $\{M'_\Lambda\}$ is a second polarized multichain of $O_B\otimes_{\BZ_p}\CO_T$-modules of type $(\CL)$, then the functor of isomorphisms of polarized multichains on the category of $T$-schemes,
$$
   T'\mapsto \Isom\bigl(\{M_\Lambda\otimes_{\CO_T}\CO_{T'}\}, \{M'_\Lambda\otimes_{\CO_T}\CO_{T'}\}\bigr),
$$
is representable by a smooth affine $T$-scheme. 
\end{theorem}

\begin{proof} We may assume that $F_0$ is a field. If $F/F_0$ is split, i.e., we are in the case (I) of \cite[p.~135]{RZ1}, then the first claim follows from \cite[Th.~3.11]{RZ1} by \cite[Lem.~A.8]{RZ1}, which reduces this case to the unpolarized case. The proofs of these facts are valid for $p=2$, and the trivialization even exists locally for the Zariski topology on $T$. That the Isom-functor is representable by an affine scheme of finite type is trivial. Smoothness follows from \cite[Th.~3.11]{RZ1}. 

Now let $F/F_0$ be an unramified field extension. Let $F^t$ denote the maximal unramified subextension of $\BQ_p$ in $F$, and let $O_{F^t}$ denote its ring of integers.  Set
\begin{equation*}
\begin{gathered}
   \wt F_0:=F_0\otimes_{\BQ_p} F^t, \quad 
      \wt F:=F\otimes_{\BQ_p} F^t, \quad 
      \wt B:=B\otimes_{\BQ_p} F^t, \quad 
      \wt V:=V\otimes_{\BQ_p} F^t, \quad
      \wt{\sform} := \sform \otimes \tr_{F^t/\BQ_p},  \\ 
   O_{\wt F_0}:=O_{F_0}\otimes_{\BZ_p} O_{F^t}, \quad 
      O_{\wt F}:=O_F\otimes_{\BZ_p} O_{F^t}, \quad 
      O_{\wt B}:=O_B\otimes_{\BZ_p} O_{F^t}, \quad 
      \wt\CL:=\CL\otimes_{\BZ_p} O_{F^t} .
\end{gathered}
\end{equation*}
Then $(\wt F, \wt B, \wt V, \wt{\sform}, O_{\wt B}, \wt \CL)$ are integral data of PEL type\footnote{In particular, $O_{\wt B}$ is indeed a maximal order since the extension $F^t/\BQ_p$ is unramified.  Also, strictly speaking, $\wt\CL = \{\Lambda \otimes_{\BZ_p} O_{F^t} \mid \Lambda \in \CL\}$ is not a multichain of $O_{\wt B}$-lattices in the precise sense of \cite[Def.\ 3.4]{RZ1}, since $\wt B$ has more simple factors than $B$.  But $\wt\CL$ gives rise to a notion of multichain of $O_{\wt B} \otimes_{\BZ_p} \CO_T$-modules of type $(\wt\CL)$ in exactly the same way as \cite[Def.\ 3.10]{RZ1}, and this notion is the same as the one for the honest multichain of lattices generated by $\wt \CL$ in $\wt V$.} such that $\wt F=\wt F_0\times \wt F_0$, i.e., $\wt F/\wt F_0$ is a product of split quadratic  extensions. Set $\wt  M_\Lambda := M_\Lambda\otimes_{\BZ_p} O_{F^t}$. Then $\{\wt M_\Lambda\}$ is a polarized multichain of $O_{\wt B}\otimes_{\BZ_p}\CO_T$-modules of type $(\wt\CL)$. From the previous case, we obtain that locally for the Zariski topology on $T$, there exists an isomorphism between $\{\wt M_\Lambda\}$ and $\wt\CL\otimes_{\BZ_p}\CO_T$. Set $\wt T := T\times_{\Spec \BZ_p}\Spec O_{F^t}$. Then $\wt T$ is an \'etale covering of $T$, and there are natural isomorphisms of polarized multichains of $O_{B}\otimes_{\BZ_p}\CO_{\wt T}$-modules of type $(\CL)$,
$$
   \{ M_\Lambda\otimes_{\CO_T}\CO_{\wt T}\}\cong \{\wt M_\Lambda\}, \quad \{\CL\otimes_{\BZ_p}\CO_{\wt T}\}\cong \{\wt\CL\otimes_{\BZ_p}\CO_T\} .
$$
Therefore there is a trivialization of the multichain $\{ M_\Lambda\otimes_{\CO_T}\CO_{\wt T}\}$, locally on $\wt T$. In the same way, the smoothness of the Isom-scheme follows from the previous case. 
\end{proof}

As a consequence, the entire formalism of local models and the local model diagram in \cite{RZ1} carries over to the $p=2$ case for PEL data as in Theorem \ref{normal forms}.  In particular, in the context of the main body of the paper, this applies when all $2$-adic places of the totally real field are unramified in the CM field.


\begin{thebibliography}{99}
\bibitem{A}{K. Arzdorf, \textit{On local models with special parahoric level structure}, Michigan Math. J. \textbf{58} (2009), no. 3, 683--710.}

\bibitem{BHKRY}{J. Bruinier, B. Howard, S. Kudla, M. Rapoport, T. Yang,  \textit{Modularity of generating series of divisors on unitary Shimura varieties}, to appear in Ast\'erisque, \href{https://arxiv.org/abs/1702.07812}{\texttt{arXiv:1702.07812 [math.NT]}}.}

\bibitem{DelBour}{P. Deligne, \textit{Travaux de Shimura}. In \textit{S\'eminaire Bourbaki, 23\`eme ann\'ee (1970/71)}, Exp. No. 389, Lecture Notes in Math. \textbf{244}, Springer-Verlag, Berlin, 1971, pp.\ 123--165.}

\bibitem{DCorv}{P. Deligne, \textit{Vari\'et\'es de Shimura: interpr\'etation modulaire et techniques de construction de mod\`eles canoniques}, Proc. Symp. Pure Math. \textbf{33} (1979), part 2,  247--290.}

\bibitem{Dr}{V. G. Drinfeld, \textit{Coverings of $p$-adic symmetric domains}, Funkcional. Anal. i Prilo\v zen. \textbf{10} (1976), no. 2, 29--40 (Russian).}

\bibitem{GGP}{W. T. Gan, B. Gross, and D. Prasad, \textit{Symplectic local root numbers, central critical
$L$-values, and restriction problems in the representation theory of classical groups},  Ast\'erisque \textbf{346} (2012), 1--109.}

\bibitem{G1}{U. G\"ortz, \textit{On the flatness of models of certain Shimura varieties of PEL-type}. Math. Ann. \textbf{321} (2001), no. 3, 689--727.}

\bibitem{G2}{U. G\"ortz, \textit{Topological flatness of local models in the ramified case}, Math. Z. \textbf{250} (2005), no. 4, 775--790.}

\bibitem{HT}{M. Harris and R. Taylor, \textit{The geometry and cohomology of some simple Shimura varieties}, with an appendix by V. Berkovich, Annals of Mathematics Studies, vol. \textbf{151}, Princeton University Press, Princeton, NJ, 2001.}

\bibitem{HPR}{X. He, G. Pappas, and M. Rapoport, \textit{Good and semi-stable reductions of Shimura varieties}, J. \'Ec. polytech. Math. \textbf{7} (2020),  497--571.}

\bibitem{Ho-kr}{B. Howard, \textit{Complex multiplication cycles and Kudla--Rapoport divisors}, Ann. of Math. (2) \textbf{176} (2012), no. 2, 1097--1171.}

\bibitem{Jac}{R. Jacobowitz, \textit{Hermitian forms over local fields}, Amer. J. Math. \textbf{84} (1962), 441--465.}

\bibitem{KP}{M. Kisin and G. Pappas, \textit{Integral models of Shimura varieties with parahoric level structure}, Publ. math. IHES \textbf{128} (2018), 121--218.}

\bibitem{K-points}{R. Kottwitz, \textit{Points on some Shimura varieties over finite fields}, J. Amer. Math. Soc.  \textbf{5}  (1992),  no. 2, 373--444.}

\bibitem{Kr}{N. Kr\"amer, \textit{Local models for ramified unitary groups}, Abh. Math. Sem. Univ. Hamburg {\bf 73} (2003), 67--80.}

\bibitem{KM}{S. Kudla and J. Millson, \textit{Intersection numbers of cycles on locally symmetric spaces and Fourier coefficients of holomorphic modular forms in several complex variables}, Pub. math. IHES \textbf{71} (1990), 121--172.}

\bibitem{KRloc}{S. Kudla and M. Rapoport, \textit{Special cycles on unitary Shimura varieties, I: Unramified local theory}, Invent. Math. \textbf{184} (2011), no. 3, 629--682.}

\bibitem{KR-U2}{S. Kudla and M. Rapoport, \textit{Special cycles on unitary Shimura varieties, II: Global theory},  J. Reine Angew. Math. \textbf{697} (2014), 91--157.}

\bibitem{KRnew}{S. Kudla and M. Rapoport, \textit{New cases of $p$-adic uniformization}, Ast\'erisque \textbf{370} (2015), 207--241.}

\bibitem{KRZ}{S. Kudla, M. Rapoport, and Th. Zink, \textit{On the $p$-adic   uniformization of unitary Shimura curves}, in preparation.}

\bibitem{Lan}{K.-W. Lan, \textit{Arithmetic compactifications of PEL-type Shimura varieties}, London Mathematical Society Monographs Series, \textbf{36}, Princeton University Press, Princeton, NJ, 2013.}

\bibitem{Liu} {Y. Liu, \textit{Fourier--Jacobi cycles and arithmetic relative trace formula}, preprint, 2018.}

\bibitem{MWY}{D. Muthiah, A. Weekes, and O. Yacobi, \textit{On a conjecture of Pappas and Rapoport about the standard local model for $GL_d$}, preprint, 2019, \href{https://arxiv.org/abs/1912.06822}{\texttt{arXiv:1912.06822 [math.AG]}}.}

\bibitem{P}{G. Pappas, \textit{On the arithmetic moduli schemes of PEL Shimura varieties}, J. Algebraic Geom. \textbf{9} (2000), no. 3, 577--605.}

\bibitem{PR1}{G. Pappas and M. Rapoport, \textit{Local models in the ramified case, I: The EL-case}, J. Algebraic Geom. \textbf{12} (2003),  no. 1, 107--145.}

\bibitem{PR2}{G. Pappas and M. Rapoport, \textit{Local models in the ramified case, II: Splitting models}, Duke Math. J. \textbf{127} (2005), no. 2, 193--250.}

\bibitem{PR-TLG}{G. Pappas and M. Rapoport, \textit{Twisted loop groups and their affine flag varieties}, Adv. Math. \textbf{219} (2008), no. 1, 118--198.}

\bibitem{PR3}{G. Pappas and M. Rapoport, \textit{Local models in the ramified case, III: Unitary groups}, J. Inst. Math. Jussieu \textbf{8} (2009), no. 3, 507--564.}

\bibitem{RSZ1}{M. Rapoport, B. Smithling, and W. Zhang, \textit{On the arithmetic transfer conjecture for exotic smooth formal moduli spaces}, Duke Math. J. \textbf{166} (2017), no. 12, 2183--2336.} 

\bibitem{RSZ2}{M. Rapoport, B. Smithling, and W. Zhang, \textit{Regular formal moduli spaces and arithmetic transfer conjectures}, Math. Ann. \textbf{370} (2018), no.\ 3--4, 1079--1175.}

\bibitem{RSZ3}{M. Rapoport, B. Smithling, and W. Zhang, \textit{Arithmetic diagonal cycles on unitary Shimura varieties},  to appear in Compos.\ Math.,\ \href{http://arxiv.org/abs/1710.06962}{\texttt{arXiv:1710.06962 [math.NT]}}.}

\bibitem{RZ1}{M. Rapoport and Th. Zink, \textit{Period spaces for $p$-divisible groups}, Annals of Mathematics Studies, vol. \textbf{141}, Princeton University Press, Princeton, NJ, 1996.}

\bibitem{RZ2}{M. Rapoport and Th. Zink, \textit{On the Drinfeld moduli problem of $p$-divisible groups},  Camb. J. Math. \textbf{5} (2017), no. 2, 229--279.}

\bibitem{S1}{B. Smithling, \textit{On the moduli description of local models for ramified unitary groups}, Int. Math. Res. Not. \textbf{2015} (2015), no. 24, 13493--13532.}

\bibitem{S2}{B. Smithling, \textit{Orthogonal analogs of some schemes considered by De Concini}, in preparation.}

\bibitem{W}{J. Weyman, \textit{Two results on equations of nilpotent orbits}, J. Algebraic Geom. \textbf{11} (2002), no. 4, 791--800.}
 
\end{thebibliography}
\end{document}